\documentclass{amsart}
\usepackage{eurosym}
\usepackage{amssymb}
\usepackage{amsmath}
\usepackage{amsfonts}
\usepackage{graphicx}
\usepackage{epsfig,euscript,enumerate,cancel}

\newtheorem{theorem}{Theorem}[section]

\newtheorem{conjecture}[theorem]{Conjecture}
\newtheorem{corollary}[theorem]{Corollary}

\newtheorem{definition}[theorem]{Definition}
\newtheorem{example}[theorem]{Example}

\newtheorem{lemma}[theorem]{Lemma}

\newtheorem{proposition}[theorem]{Proposition}
\newtheorem{remark}[theorem]{Remark}

\numberwithin{equation}{section}
\numberwithin{figure}{section}
\input{tcilatex}

\begin{document}
\title[Quantum Integrability and Quantum Schubert Calculus]{Quantum
Integrability and Generalised Quantum Schubert Calculus}
\author{Vassily Gorbounov}
\address[V. G.]{ Institute of Mathematics, University of Aberdeen, Aberdeen
AB24 3UE, UK}
\email[V.G.]{ v.gorbounov@abdn.ac.uk}
\author{Christian Korff}
\address[C. K.]{ School of Mathematics and Statistics, University of
Glasgow, Glasgow G12 8QQ, UK}
\email[C. K.]{ christian.korff@glasgow.ac.uk}
\date{August, 2014}

\begin{abstract}
We introduce and study a new mathematical structure in the generalised
(quantum) cohomology theory for Grassmannians. Namely, we relate the
Schubert calculus to a quantum integrable system known in the physics
literature as the asymmetric six-vertex model. Our approach offers a new
perspective on already established and well-studied special cases, for
example equivariant K-theory, and in addition allows us to formulate a
conjecture on the so-far unknown case of quantum equivariant K-theory. 
\end{abstract}

\maketitle


\section{Introduction}

Generalised complex oriented cohomology first appeared in the work of
Novikov \cite{Novikov} and Quillen \cite{Quillen} who realised that formal
groups naturally enter in algebraic topology. Such a theory is known to be
completely characterised by the isomorphism $h^{\ast }(\mathbb{C}P^{\infty
})\cong h^{\ast }(\text{pt})[x]$, where $x$ is the first Chern class of the
canonical line bundle over the infinite complex projective space $\mathbb{C}%
P^{\infty }$, and the K\"{u}nneth formula, $h^{\ast }(\mathbb{C}P^{\infty
}\times \mathbb{C}P^{\infty })\cong h^{\ast }(\text{pt})[x,y]$, which
implies that the first Chern class of the tensor product of two line bundles
obeys a formal group law \cite{Adams}. There are three known types of formal
group laws which come from the one-dimensional connected algebraic groups,
the additive group, the multiplicative group, and elliptic curves,
describing respectively (ordinary) cohomology, K-theory and elliptic
cohomology.

On the other hand to each of the mentioned groups one can associate
rational, trigonometric and elliptic solutions of the Yang-Baxter equation
which are linked to the appropriate quantum groups. It was first suggested
in \cite{GKV} that there should be a connection between the latter and the
mentioned generalised cohomology theories.

The study of solutions of the Yang-Baxter equation is at the heart of the
area of quantum integrable systems. Based on earlier pioneering works of
Hans Bethe \cite{Bethe} and Rodney Baxter \cite{Baxter}, the Faddeev School 
\cite{Faddeevetal} developed the \emph{algebraic Bethe ansatz} or \emph{%
quantum inverse scattering method}, where starting from a solution of the
Yang-Baxter equation one constructs the quantum integrals of motion of the
physical system as a commutative subalgebra, now often called the \emph{%
Bethe algebra}, within a larger non-commutative \emph{Yang-Baxter algebra}.
Historically, Yang-Baxter algebras were the origin for the later definition
of quantum groups by Drinfeld \cite{Drinfeld} and Jimbo \cite{Jimbo}. Using
the commutation relations of the Yang-Baxter algebra the Bethe ansatz
culminates in the derivation of a coupled set of -- in our setting --
polynomial equations, whose solutions describe the spectrum of the commuting
transfer matrices which generate the Bethe algebra. In general solving these
equations analytically is regarded as an intractable problem within the
integrable systems community except for a few special cases.

The use of quantum integrability in the study of
quantum cohomology of full flag varieties and quantum K-theory goes back to works
of Givental and Kim and Givental and Lee; see \cite{Givental,GK,Kim,Kim2}
and \cite{Givental2,Lee}. In recent work of Nekrasov and Shatashvili \cite%
{NekrasovShatashvili} which was further developed mathematically by
Braverman, Maulik and Okounkov \cite{BMO,MaulikOkounkov} it was established
that the Bethe ansatz equations of some well known integrable systems
related to the quantum groups known as Yangians describe the quantum
cohomology and quantum K-theory for a large class of algebraic varieties,
the Nakajima varieties. Particular examples are the cotangent spaces of
partial flag varieties, see the work \cite{Gorbetal}, the simplest case
being the contangent space of the Grassmannian. This opens up an exciting
new perspective on the connection made in \cite{GKV}.

In this article we shall instead investigate the above connection for the
Grassmannians $\limfunc{Gr}_{n,N}=\limfunc{Gr}_{n}(\mathbb{C}^{N})$
themselves rather than their cotangent spaces based on the earlier findings
in \cite{KorffStroppel}, \cite{VicOsc} and \cite{GoKo}; see also the work on
non-quantum $GL(N)$-equivariant cohomology in \cite{Rimanyietal}. In the present setting 
it is initially not clear which quantum group to expect. So instead we start out with special solutions to the Yang-Baxter
equation which are tied to certain exactly solvable or quantum integrable
lattice models in statistical mechanics and consider their associated
Yang-Baxter algebras as our \textquotedblleft quantum group". Despite the
models being physically motivated, they are special degenerations of the
asymmetric six-vertex model which describes ferroelectrics such as ice,
their resulting Bethe algebras -- for certain special cases -- describe
rings which have been defined previously in the setting of algebraic
topology and geometry where they are of great mathematical interest.
Specialising the parameters of the quantum integrable model in different
ways, we are able to identify them as the \emph{quantum equivariant
cohomology} \cite{Mihalcea} and the (non-equivariant) \emph{quantum K-theory} of the 
Grassmannians \cite{BM}.

These special cases prompt us to conjecture that our main result, the
description of a complex oriented generalised quantum cohomology and its
equivariant version for the Grassmannians, also covers the so far unknown
case of \emph{equivariant quantum K-theory}. At the same time this
description can be seen as solving the well-posed mathematical problem of
finding the solution to the Bethe ansatz equations: we state the coordinate
ring defined by the equations, identify a special basis in it and explicitly
describe the multiplication of two basis elements in terms of a generalised
Schubert calculus within the framework of Goresky-Kottwitz-MacPherson theory
which we show to extend to the quantum case. 

\subsection{Statement of results}

Denote by $\limfunc{Gr}_{n,N}=\limfunc{Gr}_{n}(\mathbb{C}^{N})$ the
Grassmannian of $n$-dimensional hyperplanes in $\mathbb{C}^{N}$ with $N\geq
3 $ and fix a maximal torus $\mathbb{T}\subset GL(N)$. We describe
generalised $\mathbb{T}$-equivariant quantum cohomology rings $qh_{n}^{\ast
}=qh^{\ast }(\limfunc{Gr}_{n,N};\beta )$ for $n=0,1,\ldots ,N$ using the
theory of exactly solvable lattice models in statistical mechanics \cite%
{Baxter}. While the latter appear in theoretical physics, we shall use them
here as abstract combinatorial objects -- they define a weighted counting of
non-intersecting lattice paths as described for $\beta =0$ in \cite{VicOsc}
-- which can be rigorously defined in purely mathematical terms using
Yang-Baxter algebras. The weights or probabilities attached to the lattice
models depend on

\begin{itemize}
\item a variable $\beta $ (the anisotropy parameter of the six-vertex model)
entering the multiplicative formal group law \cite{Quillen}, \cite%
{Bukhshtaberetal} and its inverse, 
\begin{equation}
x\oplus y=x+y+\beta xy\;\;\;\text{and}\;\;\;x\ominus y=\frac{x-y}{1+\beta y}%
\;,  \label{group_law}
\end{equation}

\item a \textquotedblleft quantum" parameter $q$ (the twist parameter
related to quasiperiodic boundary conditions on the lattice) as well as

\item the equivariant parameters $t=(t_{1},\ldots ,t_{N})$ (so-called
inhomogeneities in the lattice) which are connected to the natural $\mathbb{T%
}$-action on $\limfunc{Gr}_{n,N}$.
\end{itemize}

The case $\beta =0$, which corresponds to the additive group law and in
physics terminology to the so-called \emph{free fermion} point of the
lattice models, has been treated previously for the homogeneous case ($%
t_{j}=0$) in \cite{VicOsc} and recently been extended to the equivariant
setting in \cite{GoKo}.

Our approach does not require any background knowledge in statistical
mechanics, the lattice models are constructed in terms of special solutions
to the quantum (as opposed to classical) Yang-Baxter equation, hence they
are called \emph{quantum integrable}, and their description is purely
algebraic. However, we find it noteworthy that they are degenerations of the
asymmetric six-vertex model -- as mentioned previously -- and their
combinatorial description analogous to the one in \cite{VicOsc} provides a
powerful computational tool. For the latter to work we require the
previously mentioned restriction $N\geq 3$.

From these special solutions of the Yang-Baxter equation we construct
Yang-Baxter algebras, which in our case are bi-algebras only and not full
Hopf algebras. The so-called \emph{row-to-row transfer matrices} of the
lattice model generate a commutative subalgebra within the larger
non-commutative Yang-Baxter algebra which decomposes into the direct sum $%
\tbigoplus\nolimits_{n=0}^{N}qh_{n}^{\ast }$ of rings, which have the
following presentation.

Set $\mathcal{R}(\mathbb{T)}=\mathcal{R}(t_{1},\ldots ,t_{N})$ where $%
\mathcal{R}$ is the ring of rational functions in $\beta $ which are regular
at $\beta =0$ and $\beta =-1$. Define $qh_{n}^{\ast }$ %
to be the polynomial algebra generated by $\{e_{r}\}_{r=1}^{n}$, $%
\{h_{r}\}_{r=1}^{N-n}$ over $\mathcal{R}(\mathbb{T},q\mathbb{)}=\mathbb{Z}%
[\![q]\!]\otimes \mathcal{R}(\mathbb{T)}$ subject to the relations obtained
by expanding the following functional relation in the variable $x$,%
\begin{equation}
h(x)e(\ominus x)=\left( \tprod_{i=1}^{n}t_{i}\ominus x\right) \left(
\tprod_{i=1}^{N-n}x\ominus t_{i+n}\right) (1+\beta h_{1})+q,
\label{ideal def}
\end{equation}%
where $1$ is the unit element and, setting $h_{0}=e_{0}=1,$ $%
h_{N+1-n}=e_{n+1}=0$,%
\begin{eqnarray}
h(x) &=&\sum_{r=0}^{N-n}\left( h_{r}+\beta h_{r+1}\right)
\tprod_{i=1}^{N-n-r}\left( x\ominus t_{N+1-i}\right)  \label{H def} \\
e(x) &=&\sum_{r=0}^{n}\left( e_{r}+\beta e_{r+1}\right)
\tprod_{i=1}^{n-r}\left( x\oplus t_{i}\right) \;.  \label{E def}
\end{eqnarray}%
%
For the non-experts we recall that the Grothendieck $K$-functor assigns to
each smooth compact manifold $\mathcal{X}$ a ring which is built out of
complex vector bundles on $\mathcal{X}$ \cite{AA}. It is the value of this
functor and its quantum analogue $QK$ for $\mathcal{X}=\limfunc{Gr}_{n,N}$
which we shall simply refer to as (quantum) ``K-theory'' of the
Grassmannians throughout this article.

Denote by $\{e^{\varepsilon _{j}}\}_{j=1}^{N}$ the (formal) exponentials
generating the character ring of $\mathfrak{gl}(N)$.

\begin{theorem}
We have the following special cases:

\begin{itemize}
\item[(i)] $qh_{n}^{\ast }/\langle \beta \rangle $ is isomorphic to the
equivariant quantum cohomology $QH_{\mathbb{T}}^{\ast }(\limfunc{Gr}_{n,N})$
in the presentation given by Mihalcea \cite[Thm 1.1]{Mihalcea}.

\item[(ii)] $qh_{n}^{\ast }/\langle \beta +1,t_{1},\ldots ,t_{N}\rangle $ is
isomorphic to $QK(\limfunc{Gr}_{n,N}) $ as studied in \cite{BM}.

\item[(iii)] $qh_{n}^{\ast }/\langle \beta +1,t_{j}+e^{\varepsilon
_{N+1-j}}-1,q\rangle $ is isomorphic to $K_{\mathbb{T}}(\limfunc{Gr}_{n,N})$
where $K_{\mathbb{T}}$ denotes the equivariant K-functor.
\end{itemize}
\end{theorem}

Each of the cases (i)--(iii) is interesting in its own right and we compare
our findings against existing presentations of these rings in the
literature. In particular, in case (i) our results are linked to previous
(unpublished) work by Peterson \cite{Peterson} and the affine nil-Hecke ring
of Kostant and Kumar \cite{KostantKumar}: we explicitly construct a family
of operators whose matrix elements give the structure constants of $%
qh_{n}^{\ast }$ and which for $\beta =0$ can be identified with Peterson's
basis; see \cite{GoKo} for details. The other cases can then be seen as a
generalisation of this construction to $K$-theory.

To establish (ii) we compare our ring structure against the Pieri rules
derived in \cite{Lenart, LP} for $q=0$ and the quantum Pieri and
Giambelli formulae of Buch and Mihalcea \cite{BM} for $q\neq 0$. The new
result in our article is the coordinate ring presentation which follows from
(\ref{ideal def}).

Finally, we show (iii) by defining a generalisation of
Goresky-Kottwitz-MacPherson theory \cite{GKM}: we identify McNamara's
factorial Grothendieck polynomials \cite{McNamara} with localised Schubert
classes using the Bethe ansatz of quantum integrable models. 


Based on the above special cases we have the following\footnote{
After submission of this article the work \cite{BCMP} appeared in which the equivariant 
quantum K-theory of cominiscule varieties is investigated. Specialising Thm 3.9 in {\em loc. cit.} to the 
case of the Grassmannian one recovers the quantum equivariant Pieri rule, see \eqref{H1} in this article, 
which completely fixes the ring structure and, hence, proves our conjecture; see also Remark 5.13 in \cite{BCMP}.}:

\begin{conjecture}
$qh_{n}^{\ast }/\langle \beta +1,t_{j}+e^{\varepsilon _{N+1-j}}-1\rangle $
describes the value $QK_{\mathbb{T}}(\limfunc{Gr}_{n,N})$ of the equivariant
quantum $K$-functor for the Grassmannians.
\end{conjecture}

\begin{remark}
We note that we can define $qh^*_n$ also over the ring of Laurent
polynomials in $\beta$ instead, which would introduce a natural $\mathbb{Z}$%
-grading. This suggests that our framework might also be used to describe
the actual $\mathbb{Z}$-graded quantum equivariant K-theory which is
obtained from the K-functor in conjunction with the Bott Periodicity
Theorem. However, there is currently not sufficient evidence available to
further substantiate this claim, hence we state this here as a mere
observation and not as a conjecture.
\end{remark}

Besides providing a complete description of $QK_{\mathbb{T}}(\limfunc{Gr}%
_{n,N})$, which has so far been missing in the literature, the new aspects
in our approach are

\begin{enumerate}
\item that our ring is defined for general $\beta $ which allows us to treat
all these special cases at once in a unified setting of a quantum
generalised cohomology theory as first defined in \cite{CG} and

\item that we reveal an underlying quantum group structure in terms of
Yang-Baxter algebras which we show to commute with the natural symmetric
group action on the idempotents of these rings.
\end{enumerate}

As a byproduct of our investigation we also derive new combinatorial results
such as a generalised Jacobi-Trudy formula and Cauchy identity for factorial
Grothendieck polynomials.

\subsection{Outline of the article}

\begin{description}
\item[Section 2] We introduce the necessary combinatorial objects and
notations we will use throughout the article. In particular, we review
McNamara's definition of factorial Grothendieck polynomials which play a
central role in our approach and derive several new results which we need to
describe our generalised cohomology ring for the Grassmannian.

\item[Section 3] Starting from special solutions to the Yang-Baxter
equation, so-called $L$-operators, we define the Yang-Baxter algebra in
terms of endomorphisms over some vector space $\mathcal{V}$ which will be
identified with the direct sum of the generalised cohomology rings $%
\tbigoplus\nolimits_{n=0}^{N}qh_{n}^{\ast }$. We describe the commutation
relations of the Yang-Baxter algebra and define the transfer matrices which
generate a commutative subalgebra, the so-called Bethe algebra. The action
of the latter on $\mathcal{V}$ is described combinatorially using toric skew
diagrams. We also show that the transfer matrices obey the functional
relation (\ref{ideal def}).

\item[Section 4] We derive the spectrum of the Bethe algebra by constructing
their eigenvectors and computing their eigenvalues using the algebraic Bethe
ansatz. Both, eigenvectors and eigenvalues, are described in terms of the
solutions of a set of coupled equations, called the Bethe ansatz equations,
which we show can be solved in terms of formal power series in the quantum
deformation parameter $q$ of $qh_{n}^{\ast }$. We then initially define the
generalised cohomology ring by identifying the eigenbasis of the transfer
matrices as the primitive, central orthogonal idempotents of $qh_{n}^{\ast }$%
. We also define a bilinear form which turns $qh_{n}^{\ast }$ into a
Frobenius algebra. Having identified the eigenvectors of the transfer
matrices as idempotents, we then fix the analogue of the Schubert basis and
describe the product in this geometrically motivated basis instead. This
allows us to state a residue formula for the structure constants in the
Schubert basis in terms of the solutions of the Bethe ansatz equations and
show that they obey a recurrence formula which is derived from an
equivariant quantum Pieri-Chevalley formula for $qh_{n}^{\ast }$.

\item[Section 5] Employing the description of $qh_{n}^{\ast }$ in terms of
its idempotents leads to a formulation of the ring in terms of column
vectors whose components can be thought of as generalised localised Schubert
classes where the localisation points are identified with the solutions of
the Bethe ansatz equations. We show that these generalised Schubert classes
obey generalised Goresky-Kottwitz-MacPherson conditions which derive from an
action of the symmetric group. Interestingly, the latter emerges naturally
from solutions of the Yang-Baxter equation discussed in Section 3. The symmetric group action also gives
rise to a representation of a generalised Iwahori-Hecke algebra and commutes
with the action of the Yang-Baxter algebra. Using this framework of GKM
theory we prove the special cases mentioned in the introduction, that is we
show that our ring $qh_{n}^{\ast }$ can be specialised to equivariant
quantum cohomology and quantum K-theory. This section also gives the proof
of the presentation of $qh_{n}^{\ast }$ as polynomial algebra modulo the
relations (\ref{ideal def}).\bigskip
\end{description}

\noindent \textbf{Acknowledgment}. The authors would like to thank the Max
Planck Institute for Mathematics Bonn, where part of this work was carried
out, for hospitality. They are grateful to Leonardo Mihalcea and Alexander
Varchenko for comments on a draft version of the article. C. K. also
gratefully acknowledges discussions with Gwyn Bellamy, Christian Voigt, Paul
Zinn-Justin and would like to thank the organisers Anita Ponsaing and Paul
Zinn-Justin for their kind invitation to the workshop \emph{Combinatorics
and Integrability}, Presqu'\^{\i}le de Giens, 23-27 June 2014, where the
results of this article were presented.

\section{Preliminaries}

This section introduces the combinatorial notions needed in the description
of Schubert calculus in the rest of this paper. We also collect known as
well as a number of new results on factorial Grothendieck polynomials.

\subsection{Minimal coset representatives}

Denote by $\mathbb{S}_{N}$ the symmetric group in $N$-letters and choose $%
k,n\in \mathbb{N}_{0}$ such that $N=n+k$. A set of minimal length coset
representatives $w$ for classes $[w]$ in $\mathbb{S}_{N}/\mathbb{S}%
_{n}\times \mathbb{S}_{k}$ is given by the permutations for which $%
w(1)<w(2)<\cdots <w(n)$ and $w(n+1)<w(n+2)<\cdots <w(N)$. For instance, the
coset representative with $w^{0}(n+1)=1,$ $w^{0}(n+2)=2,\ldots ,$ $%
w^{0}(N)=k $ and $w^{0}(1)=k+1,$ $w^{0}(2)=k+2,\ldots ,$ $w^{0}(n)=N$ is
given by 
\begin{equation}
w^{0}=s_{n}s_{n+1}\cdots s_{N-1}~\cdots ~s_{2}s_{3}\cdots
s_{k+1}~s_{1}s_{2}\cdots s_{k}~.
\end{equation}

\subsection{Binary strings}

The $w$'s are in bijection with $01$-words or binary strings $%
b(w)=b_{1}b_{2}\cdots b_{N}\in \{0,1\}^{N}$ of length $N$, where $%
n=|b|:=\sum_{j}b_{j}$ is the number of $1$-letters, $k=N-|b|$ the number of $%
0$-letters and 
\begin{equation}
b_{j}(w)=\left\{ 
\begin{array}{cc}
1, & j\in I_{w} \\ 
0, & j\in \lbrack N]\backslash I_{w}%
\end{array}%
\right.
\end{equation}%
with $I_{w}:=\{w(1),\ldots ,w(n)\}$ and $[N]=\{1,2,\ldots N\}$. So, in the
case of $w^{0}$ we have $I_{w^{0}}=\{k+1,\ldots ,N\}$ and $b(w^{0})=0\cdots
01\cdots 1$ is the binary string with $k$ 0-letters in front, followed by $n$
1-letters. The identity $w=1$ on the other hand corresponds to the binary
string $b(1)=1\cdots 10\cdots 0$ instead. Note that under the above
bijection the natural $\mathbb{S}_{N}$-action on $\mathbb{S}_{N}/\mathbb{S}%
_{n}\times \mathbb{S}_{k}$ via $s_{j}\cdot \lbrack w]=[s_{j}w]$ coincides
with the natural $\mathbb{S}_{N}$-action on binary strings, where $s_{j}$
permutes the $j$th and $(j+1)$th letter in $b(w)$.

\subsection{Boxed partitions\label{sec:partitions}}

Each binary string $b(w)$, and thus each minimal length representative $w$,
is in one-to-one correspondence with a partition $\lambda =(\lambda
_{1},\lambda _{2},\ldots ,\lambda _{n})$ which has at most $n$ parts and for
which $\lambda _{1}\leq k$. That is, the corresponding Young diagram lies in
a bounding box of height $n$ and width $k$ and we will denote this by
writing $\lambda \subset (k^{n})$. The bijection is given by the relation%
\begin{equation}
w(i)=\lambda _{n+1-i}+i,\qquad i=1,2,\ldots ,n\;.  \label{w2lambda}
\end{equation}%
For $w=w^{0}$ we have that $\lambda =(k^{n})$, the partition with $n$ parts
equal to $k$, and for $w=1$ we obtain the empty partition denoted by $%
\lambda =\emptyset $. N.B. the bijection is defined for fixed $n,k$, so each
partition comes with a bounding box of fixed dimensions. For different $n,k$
one may obtain the same partition $\lambda $ but the dimensions of the
bounding box will then be different. We therefore refer to $\lambda $ as a 
\emph{boxed partition} as the dimensions of the bounding box enter in the
bijection (\ref{w2lambda}).

Throughout this article we will use these various labellings of the same
coset $[w]$ interchangeably writing $b(w)$, $\lambda (w)$ for the images
under the above bijection and $b(\lambda ),w(\lambda )$ for the pre-images.
By abuse of notation we shall also write $s_{j}b$ and $s_{j}\lambda $ for
the binary string and partition corresponding to the coset $[s_{j}w]$.

\subsection{Cylindric loops and toric skew diagrams\label{sec:torictab}}

We briefly recall the definition of cylindric loops $\lambda \lbrack r]$
associated with a partition $\lambda $ and \emph{toric skew diagrams}; see 
\cite{GesselKrattenthaler, Postnikov}.

Let $\lambda =(\lambda _{1},\ldots ,\lambda _{n})\subset (k^n)$, then the
associated \emph{cylindric loop} $\lambda \lbrack r]$ for any $r\in \mathbb{Z%
}$ is defined as%
\begin{equation}
\lambda \lbrack r]:=(\ldots ,\underset{r}{\lambda _{n}+r+k},\underset{r+1}{%
\lambda _{1}+r},\ldots ,\underset{r+n}{\lambda _{n}+r},\underset{r+n+1}{%
\lambda _{1}+r-k},\ldots )\;.  \label{cyl_loop}
\end{equation}%
For $r=0$ the cylindric loop can be visualised as a path in $\mathbb{Z}%
\times \mathbb{Z}$ determined by the outline of the Young diagram of $%
\lambda $ which is periodically continued with respect to the vector $(n,-k)$%
. For $r\neq 0$ this line is shifted $r$ times in the direction of the
lattice vector $(1,1)$.

Given two boxed partitions $\lambda ,\mu \subset (k^{n})$ denote by $\lambda
/d/\mu $ the set of squares between the two lines $\lambda \lbrack d]$ and $%
\mu \lbrack 0]$ modulo integer shifts by $(n,-k)$, 
\begin{equation}
\lambda /d/\mu :=\{\langle i,j\rangle \in \mathbb{Z}\times \mathbb{Z}/(n,-k)%
\mathbb{Z}~|~\lambda \lbrack d]_{i}\geq j>\mu \lbrack 0]_{i}\}\;.
\label{cyl_diagram}
\end{equation}%
We shall refer to $\theta =\lambda /d/\mu $ as a \emph{cylindric skew-diagram%
} of degree $d=d(\theta )$. Postnikov introduced \cite{Postnikov} the
terminology \emph{toric skew-diagram }for those\emph{\ }$\theta $ where the
number of boxes in each row does not exceed $k$. Note that $\lambda /0/\mu
=\lambda /\mu $, that is cylindric or toric skew-diagrams contain ordinary
skew diagrams as special cases.

A cylindric skew diagram $\theta $ which has at most one box in each column
will be called a \emph{toric} \emph{horizontal strip} and one which has at
most one box in each row a \emph{toric} \emph{vertical strip}. The \emph{%
length} of such strips will be the number of boxes within the skew diagram,
where we identify squares modulo integer shifts by $(n,-k)$ and choose as
representatives those squares $s=\langle i,j\rangle $ with $1\leq j\leq n$.
In what follows this identification is always understood implicitly if we
talk about a square in a toric strip.

\subsection{Bases in equivariant cohomology and K-theory\label{sec:schubert}}

We are interested in describing equivariant quantum cohomology ($\beta =0$)
and K-theory ($\beta =-1$) as special cases of our generalised cohomology
theory for $\operatorname{Gr}_{n,N}=\operatorname{Gr}(n,\mathbb{C}^N)$. The 
equivariant cohomology \cite{KostantKumar} and K-theory \cite{KostantKumar2} of flag varieties -- of
which Grassmannians are a special case -- was studied by Kostant and Kumar.
The equivariant quantum cohomology of flag varieties was computed in \cite%
{Givental,Kim, GK, Kim2} and quantum K-theory in \cite%
{Givental2, Lee, GL} and since then has been discussed by
numerous authors.

Specialising $\beta =0$ we identify $\mathcal{R}(\mathbb{T)}$ with the
equivariant cohomology $H_{\mathbb{T}}^{\ast }(\limfunc{pt})=\mathbb{Z}%
[t_{1},\ldots ,t_{N}]$ of a point by mapping each $f_{\beta }\in \mathcal{R}(%
\mathbb{T)}$ to its value at $\beta =0$. Fix the standard basis $\{e_i\}_{i=1}^N$ in $\mathbb{C}^N$ 
and the standard flag $0=F_0\subset F_1\subset \cdots\subset F_N=\mathbb{C}^N$ with $F_i$ being 
the linear span of $\{e_j\}_{j=1}^i$ if $i>0$. For any partition $\lambda\subset (k^n)$ define 
$C_\lambda=\{V\in\operatorname{Gr}_{n,N}~|~\dim\left[(V\cap F_i)/(V\cap F_{i-1})\right]
=b_i(\lambda),\,i=1,\dots,N\}$ with $b(\lambda)$ being the binary string corresponding to $\lambda$. 
The closure of the Schubert cell $C_\lambda$ is the {\em Schubert variety} $X_{\lambda }$. 

Consider 
the natural action of the symmetric group $\mathbb{S}_N$ on the standard basis which for each 
permutation $w\in\mathbb{S}_N$ produces a new flag $F^w$ and similar as before one defines a 
Schubert cell $C_\lambda^w$ and its closure $X^w_\lambda$. Define the \emph{opposite Schubert 
varieties} as $X^\lambda=X^{w_0}_{\lambda ^{\vee }}$ where $\lambda ^{\vee }$ is obtained by 
reversing the binary string $b(\lambda )$ and $w_{0}$ is the \emph{longest element} in $\mathbb{S}_{N}$. 
We also recall the definition of the \emph{%
Richardson variety} $X_{\mu }^{\lambda }=X_{\mu }\cap X^{\lambda }$. All
three varieties are left invariant under the torus action and, thus, determine equivariant Schubert classes
$\{[X_{\lambda }]\}_{\lambda \subset (k^{n})}$ and $\{[X^{\lambda }]\}_{\lambda \subset (k^{n})}$. Consider 
the pairing $H_{\mathbb{T}}^{\ast }(\operatorname{Gr}_{n,N})\otimes_{H_{\mathbb{T}}^{\ast }(\limfunc{pt})} H_{\mathbb{T}}^{\ast }(\operatorname{Gr}_{n,N})\rightarrow H_{\mathbb{T}}^{\ast }(\limfunc{pt})$ given by $(\sigma,\tau)=\pi_\ast(\sigma\cup\tau)$, 
where $\pi:\operatorname{Gr}_{n,N}\rightarrow\operatorname{pt}$ is the $\mathbb{T}$-equivariant map to some point 
and $\pi_\ast$ is the induced equivariant Gysin map $H_{\mathbb{T}}^{\ast }(\operatorname{Gr}_{n,N})\rightarrow H_{\mathbb{T}}^{\ast }(\limfunc{pt})$; see e.g. \cite[Sec. 3]{Mihalcea06} for details. Both bases are dual with respect to the pairing; see \cite[Prop 3.1]{Mihalcea06}. As explained in {\em loc. cit.}, unlike the non-equivariant case $[X_{\lambda^\vee}]\neq[X^\lambda]$, but instead one has to invoke an isomorphism $\varphi^\ast:H_{\mathbb{T}}^{\ast }(\operatorname{Gr}_{n,N})\rightarrow H_{\mathbb{T}}^{\ast }(\operatorname{Gr}_{n,N})$ 
induced by the isomorphism $\varphi:\operatorname{Gr}_{n,N}\rightarrow \operatorname{Gr}_{n,N}$ corresponding to left multiplication with the longest element $w_0$. This map $\varphi$ is not $\mathbb{T}$-equivariant but sends the elements $t$ in $H_{\mathbb{T}}^{\ast }(\operatorname{pt})$ to $t'=(t_N,\ldots,t_2,t_1)$ and one then has the relation $[X^\lambda]=\varphi^\ast[X_{\lambda^\vee}]$. The 
definition of the equivariant cohomology ring can be generalised to the definition of the (small) equivariant quantum cohomology ring as $\mathbb{Z}[q]\otimes\mathbb{Z}[t_1,\ldots,t_n]$ module extending the notions of Schubert basis and pairing; see \cite{Kim} for the original reference and the discussion in \cite[Sec. 5]{Mihalcea06}.

One is interested in the
computation of the \emph{3-point genus 0 equivariant Gromov-Witten invariants%
} $C_{\lambda \mu }^{\nu }(t,q)$ which appear in the product%
\begin{equation}
\lbrack X_{\lambda }][X_{\mu }]=\sum_{\nu\subset (k^n)}C_{\lambda \mu
}^{\nu}(t,q)[X_{\nu }]  \label{GWinv}
\end{equation}%
and for the Grassmannian are monomials in $q$, i.e. $C_{\lambda \mu }^{\nu
}(t,q)=q^{d}C_{\lambda \mu }^{\nu ,d}(t)$. The invariants for $d=0$ also
appear in the expansion%
\begin{equation}
\lbrack X_{\mu }^{\lambda }]=\sum_{\nu\subset (k^n) }C_{\mu \nu }^{\lambda ,0}(t)[X^{\nu
}]\;,  \label{Richardson}
\end{equation}%
and, thus, $C_{\mu \nu }^{\lambda ,0}(t)=c_{\lambda \mu }^{\nu }(t)$ are the
analogue of Littlewood-Richardson coefficients for factorial Schur functions 
\cite{MolevSagan}.

In the case of $K$-theory we specialise $\beta =-1$ and set $%
t_{j}=1-e^{\varepsilon _{N+1-j}}$ where the (formal) exponentials $%
\{e^{\varepsilon _{j}}\}_{j=1}^{N}$ generate the character ring of $%
\mathfrak{gl}(N)$. Mapping each $f_{\beta }\in \mathcal{R}(\mathbb{T)}$ to
its value at $\beta =-1$ then gives us $K_{\mathbb{T}}(\limfunc{pt})=%
\limfunc{Rep}(\mathbb{T})$, the representation ring of $\mathbb{T}$ which is
canonically isomorphic to the group algebra of the free abelian group of
characters $e^{\varepsilon _{i}}$. The ring $K_{\mathbb{T}}(\limfunc{Gr}%
_{n,N})$ is generated by the classes $[\mathcal{O}_{\lambda }]$ 
of the \emph{structure sheaves} $\mathcal{O}_{\lambda }$ 
of the Schubert varieties within the Grothendieck group of coherent sheaves on the
Grassmannian. Their product expansion%
\begin{equation}
\lbrack \mathcal{O}_{\lambda }][\mathcal{O}_{\mu }]=\sum_{\nu \subset
(k^{n})}c_{\lambda \mu }^{\nu }(t)[\mathcal{O}_{\nu }], \label{Kstructure}
\end{equation}%
define the K-theoretic Littlewood-Richardson coefficients $c_{\lambda \mu
}^{\nu }(t)$, which we denote by the same symbol as the Littlewood-Richardson coefficients for $\beta=0$. There are known positivity statements for these structure
constants, see \cite{GR} and \cite[Sec 5]{AGM} as well as references therein. We shall refer to the K-classes 
$\{[\mathcal{O}_{\lambda }]\}_{\lambda\subset (k^n)}$ as {\em Schubert basis} or simply {\em Schubert classes}. The classes $[\mathcal{O}^\lambda]$ 
of structure sheaves of the opposite Schubert varieties provide an alternative basis; see e.g. \cite{GrKu}. 

Similar as in the case of equivariant cohomology, $\beta=0$ one fixes a bilinear form $\rho_\ast:K_{\mathbb{T}}(\limfunc{Gr}_{n,N})\otimes_{K_{\mathbb{T}}(\limfunc{pt})}
K_{\mathbb{T}}(\limfunc{Gr}_{n,N})\rightarrow K_{\mathbb{T}}(\limfunc{pt})$ which is induced by a map $\rho:\operatorname{Gr}_{n,N}\rightarrow \operatorname{pt}$. In contrast to the case $\beta =0$, the construction of a suitable dual basis of the classes $[\mathcal{O}_{\lambda }]$ with respect to this pairing is now more involved. One has to introduce additional classes $[\xi ^{\lambda }]$ which can also be
defined in terms of sheaves (see \cite[Prop 2.1]{GrKu}) and which can be related to the classes $[\mathcal{O}_\lambda]$ 
as follows \cite[Proof of Prop 4.2]{BCMP},\footnote{We would like to thank the referee for bringing this formula to our attention.}
\begin{equation}
\lbrack \xi ^{\lambda }]=\frac{(1-[\mathcal{O}_{1}])\varphi^\ast [\mathcal{O}_{\lambda ^{\vee}}]}{1-[\mathcal{O}_{1}]_\lambda}
\,,  \label{equiKdualbasis}
\end{equation}   
where $[\mathcal{O}_{1}]$ is the K-class of the Schubert divisor and $\varphi^\ast:K_{\mathbb{T}}(\operatorname{Gr}_{n,N})\rightarrow K_{\mathbb{T}}(\operatorname{Gr}_{n,N})$ is the map induced by 
multiplication with the longest element $w_0$ as explained above.  The additional weight factor $(1-[\mathcal{O}_{1}]_\lambda)^{-1}\in  K_{\mathbb{T}}(\limfunc{pt})$ is the class of the Schubert divisor localised at the fixed point labelled by $\lambda$. In the non-equivariant case, $K(\limfunc{Gr}_{n,N})=K_{\mathbb{T}}(\limfunc{Gr}%
_{n,N})/\langle t_{1},\ldots ,t_{N}\rangle $, one has the simpler relation \cite[Sec 8]{BuchKtheory} 
\begin{equation}
\lbrack \xi ^{\lambda }]=(1-[\mathcal{O}_{1}])[\mathcal{O}_{\lambda ^{\vee
}}]  \label{Kdualbasis}
\end{equation}%

In our construction of $qh^*_n$ via a quantum integrable model, we will identify below the Schubert basis 
for $\beta=0$ and $\beta=-1$ with what we call the {\em spin basis} \eqref{spin basis}; see also \eqref{Bethe2spin} 
and \eqref{Bethe2spin0} for its expressions in terms of idempotents when $q\neq 0$ and $q=0$, respectively. We also introduce a bilinear form \eqref{bilinear_form} (or alternatively \eqref{bilinear_form_spin}) which in our setting is fixed by asking $qh^\ast_n$ to be a semi-simple Frobenius algebra and the orthogonality \eqref{Cauchy2} and completeness relation \eqref{res of 1} of the Bethe ansatz. The construction of a dual basis then follows in Prop \ref{prop:dual_basis} in terms of what we call the {\em opposite spin basis}; see \eqref{op_spin_basis} for $q\neq 0$ and \eqref{op_spin_basis0} for $q=0$.  This allows us to prove Prop \ref{prop:dual_basis} which shows that our algebraic definitions match the geometric ones from this section, in particular that we work with the same bilinear form. Cor \ref{cor:Zexpansion} then states a generalisation of the expansion \eqref{Richardson} for $\beta\neq 0$, whose geometric interpretation we leave to future work.
 
\subsection{Discrete symmetries\label{sec:symmetries}}

Throughout this article we will make use of several involutions and a
natural $\mathbb{Z}_{N}$-action defined on the set of cosets in $\mathbb{S}%
_{N}/\mathbb{S}_{n}\times \mathbb{S}_{k}$ where $k=N-n$ as before. These
will induce mappings between elements in the Schubert basis, in some cases
from different rings, and since they in turn lead to non-trivial
transformation properties of the structure constants of $qh_{n}^{\ast }$, we
refer to them as \textquotedblleft symmetries\textquotedblright .

\subsubsection{Poincar\'{e} Duality}

As we will see below $qh^*_n$ possesses a basis $\{g_\lambda\}_{\lambda\subset (k^n)}$ labelled by boxed partitions or binary strings; see Section \ref{sec:coordring}. Define an involution $\vee :qh_{n}^{\ast }\rightarrow qh_{n}^{\ast }$ by
reversing a binary string, i.e. $b_{i}^{\vee }=b_{N+1-i}$. We shall denote
the corresponding permutation and partition by $w^{\vee }$ and $\lambda
^{\vee }$, respectively. One easily verifies that the Young diagram of $%
\lambda ^{\vee }$ is the complement of the Young diagram of $\lambda $ in
the $n\times k$ bounding box.

\subsubsection{Level-Rank Duality}

Using the same basis of $qh^*_n$ as above, define an involution $\ast :qh_{n}^{\ast }\rightarrow qh_{N-n}^{\ast }$ by
swapping 0 and 1-letters in binary strings, i.e. $b_{i}^{\ast }=1-b_{i}$.
The corresponding partition $\lambda ^{\ast }$ is obtained by taking first
the conjugate partition $\lambda ^{\prime }$ and then its complement in the
bounding box or vice versa, i.e. $\lambda ^{\ast }=(\lambda ^{\prime
})^{\vee }=(\lambda ^{\vee })^{\prime }$. So, in particular we can define
the composite involution $qh_{n}^{\ast }\rightarrow qh_{N-n}^{\ast }$ by $%
\lambda \mapsto \lambda ^{\prime }$ and shall denote the corresponding
binary string and permutation respectively by $b^{\prime }$ and $w^{\prime }$%
.

%

\subsection{Set-Valued Tableaux and Grothendieck polynomials\label%
{sec:grothendieck}}

We recall some of the necessary combinatorial objects and the definition of
factorial Grothendieck polynomials. This is based on earlier work by Buch 
\cite{BuchKtheory} and McNamara \cite{McNamara}, but we shall also derive
several new results which are not contained in the latter works.

Let $n$ be some non-negative integer. We will use the notation $%
[n]=\{1,\ldots ,n\}$ and $\mathbb{P}_{n}=\mathbb{P}([n])$ for the power set
of $[n]$, the set of all subsets of $[n]$. Denote by $\theta $ a skew Young
diagram with at most $|\theta |\leq n$ boxes which we identify with a subset
of $\mathbb{Z}^{2}$.

\begin{definition}[\cite{BuchKtheory}]
A set-valued tableau is a map $T:\theta \rightarrow \mathbb{P}_{n}$ such
that the following conditions hold%
\begin{equation}
\max T(i,j)\leq \min T(i,j+1)\;\;\;\text{and}\;\;\;\max T(i,j)<\min
T(i+1,j)\;.
\end{equation}
\end{definition}

Denote by $|T|$ the sum over the cardinalities of all the subsets in the
image of $T$ and let $\limfunc{SVT}(\theta )$ be the set of all set-valued
tableau of shape $\theta $. Then we have the following definition of
factorial Grothendieck polynomials due to McNamara \cite{McNamara} which is
an extension of Buch's earlier realisation \cite{BuchKtheory} of ordinary
(skew) Grothendieck polynomials as sum over set-valued tableaux.

\begin{definition}[\cite{McNamara}]
The factorial (skew) Grothendieck polynomial is the weighted sum%
\begin{equation}
G_{\theta }(x|t)=\sum_{T}\beta ^{|T|-|\theta |}\prod_{\substack{ (i,j)\in
\theta  \\ r\in T(i,j)}}x_{r}\oplus t_{r+j-i}  \label{facG}
\end{equation}%
over all set-valued tableaux $T\in \limfunc{SVT}(\theta )$.
\end{definition}

N.B. the factorial Grothendieck polynomials are in general defined for an 
\emph{infinite} sequence $(t_j)_{j\in\mathbb{Z}}$ of parameters. For this
section only we shall assume these parameters to be nonzero for all $j$ but
then set $t_j=0$ unless $1\leq j\leq N$ and identify them with the
equivariant parameters mentioned in the introduction.

Employing (\ref{group_law}) define the $\beta $-deformed factorial power 
\begin{equation}
(x_{j}|t)^{r}:=\prod_{i=1}^{r}x_{j}\oplus t_{i}\;.  \label{facpower}
\end{equation}%
The following determinant formula is stated in \cite[Eqn (2.12)]{IN}. Its
proof follows along similar lines as indicated in \emph{loc. cit.} where the
focus is on the symplectic case.

\begin{proposition}[\cite{IN}]
\begin{equation}
G_{\theta }(x|t)=\frac{\det \left[ (x_{j}|t)^{\theta _{i}+n-i}(1+\beta
x_{j})^{i-1}\right] _{1\leq i,j\leq n}}{\det [x_{j}^{n-i}]_{1\leq i,j\leq n}}
\label{facGdet}
\end{equation}%
where the denominator is the Vandermonde determinant $\Delta
(x)=\prod_{i<j}(x_{i}-x_{j})$.%
\end{proposition}

We recall the following known specialisations of factorial Grothendieck
polynomials.

Setting $t_{j}=0$ for all $j$ one recovers the (ordinary) \emph{Grothendieck
polynomial} which has the following determinant presentation,%
\begin{equation}
G_{\theta }(x)=\frac{\det \left( x_{j}^{\theta _{i}+n-i}(1+\beta
x_{j})^{i-1}\right) _{1\leq i,j\leq n}}{\det \left( x_{j}^{n-i}\right)
_{1\leq i,j\leq n}}\;.  \label{G}
\end{equation}

Setting $\beta =0$ one obtains the \emph{factorial Schur function} (see e.g. 
\cite{Macdonald} and \cite[Ch. I.3, Ex. 20]{MacdonaldBook} as well as
references therein),%
\begin{equation}
s_{\theta }(x|t)=\frac{\det [(x_{j}|t)^{\theta _{i}+n-i}]_{1\leq i,j\leq n}}{%
\det [(x_{j}|t)^{n-i}]_{1\leq i,j\leq n}}~,\qquad (x_{j}|t)^{r}\overset{%
\beta =0}{=}\prod_{i=1}^{r}(x_{j}+t_{i})\;.  \label{facSchur}
\end{equation}

We collect further properties of factorial Grothendieck polynomials which we
will use throughout this article.

We use the determinant formula (\ref{facGdet}) to derive the following
equation which is a generalisation of the known straightening rule for Schur
functions $s_{\theta }$, $s_{\ldots ,\theta _{i},\theta _{i+1},\ldots
}=-s_{\ldots ,\theta _{i+1}-1,\theta _{i}+1,\ldots }$ \cite[Ch I.3]%
{MacdonaldBook}. The latter -- through repeated application -- allows one to
express a Schur function indexed by a composition in terms of Schur
functions indexed by partitions. We will use the straightening rule for
factorial Grothendieck polynomials for the same purpose.

\begin{corollary}[straightening rule]
We have the following relations%
\begin{multline}
G_{\ldots ,\theta _{i},\theta _{i+1},\ldots }=  \label{Gstraight} \\
-\beta G_{\ldots ,\theta _{i}+1,\theta _{i+1},\ldots }-\frac{1+\beta
t_{n+\theta _{i}-i+1}}{1+\beta t_{n+\theta _{i+1}-i}}~(G_{\ldots ,\theta
_{i+1}-1,\theta _{i}+1,\ldots }+\beta G_{\ldots ,\theta _{i+1},\theta
_{i}+1,\ldots }),
\end{multline}%
where $G_{\theta }=G_{\theta }(x|t)$ with $\theta =(\theta _{1},\ldots
,\theta _{n})$.
\end{corollary}

\begin{proof}
Without difficulty one verifies the identity%
\begin{equation*}
(1+\beta t_{m+1})(x|t)^{m}(1+\beta x)^{r}=(x|t)^{m}(1+\beta x)^{r-1}+\beta
(x|t)^{m+1}(1+\beta x)^{r-1}\;.
\end{equation*}%
Applying the latter first to the $i$th and then to the $(i+1)$th row of the
determinant in the numerator of (\ref{facGdet}), the assertion follows.
\end{proof}

Given a boxed partition $\lambda \subset (k^{n})$ we introduce the shorthand
notations%
\begin{eqnarray}
t_{\lambda } &=&(t_{\lambda _{n}+1},\ldots ,t_{\lambda _{i}+n+1-i},\ldots
,t_{\lambda _{1}+n})  \notag \\
\ominus t_{\lambda } &=&(\ominus t_{\lambda _{n}+1},\ldots ,\ominus
t_{\lambda _{i}+n+1-i},\ldots ,\ominus t_{\lambda _{1}+n})  \label{tlambda}
\end{eqnarray}%
where $\ominus x:=0\ominus x=-x/(1+\beta x)$ for any formal variable $x$;
compare with (\ref{group_law}). The following is a generalisation of the
Vanishing Theorem for factorial Schur functions \cite{MolevSagan} to
factorial Grothendieck polynomials; see \cite[Thm 4.4]{McNamara}.

\begin{theorem}[\cite{McNamara}]
Let $\lambda ,\mu $ be partitions with at most $n$ parts then%
\begin{equation}
G_{\lambda }(\ominus t_{\mu }|t)=\left\{ 
\begin{array}{cc}
0, & \lambda \nsubseteq \mu \\ 
\tprod\limits_{\langle i,j\rangle \in \lambda }t_{n+j-\lambda _{j}^{\prime
}}\ominus t_{\lambda _{i}+n+1-i}, & \lambda =\mu%
\end{array}%
\right.  \label{Gvanish}
\end{equation}%
and in general $G_{\lambda }(\ominus t_{\mu }|t)$ will be non-zero if $%
\lambda \subset \mu $.
\end{theorem}

Following \cite{McNamara} we introduce for simplicity the notation%
\begin{equation}
\Pi (x)=\prod_{i=1}^{n}(1+\beta x_{i})\;.  \label{Pi}
\end{equation}%
We recall the following results \cite[Ex 4.2 and Prop 4.8]{McNamara}.

\begin{lemma}[\cite{McNamara}]
We have the identity%
\begin{equation}
1+\beta G_{1}(x|t)=\sum_{i=0}^{n}\beta ^{i}e_{i}(x\oplus t)=\Pi (x)\Pi
(t_{\emptyset })\;,  \label{G1}
\end{equation}%
where the $e_{i}$'s denote the elementary symmetric polynomials.
\end{lemma}

\begin{proposition}[\cite{McNamara}]
We have the expansion%
\begin{equation}
\Pi (x)G_{\lambda }(x|t)=\Pi (\ominus t_{\lambda })\sum_{\lambda
\rightrightarrows \mu }\beta ^{|\mu /\lambda |}G_{\mu }(x|t),
\label{G1Pieri}
\end{equation}%
where the notation $\lambda \rightrightarrows \mu $ indicates that the sum
runs over all partitions $\mu $ which contain $\lambda $ and for which the
skew diagram $\mu /\lambda $ has at most one box in each column or row.
\end{proposition}

Denote by $\Lambda _{n}\otimes \mathbb{Z}(\beta ,t_{1},\ldots ,t_{N})$ the
linear space spanned by the monomial symmetric functions $\{m_{\lambda
}\}_{\lambda \subset (k^{n})}$, then the following result is \cite[Thm 4.6]%
{McNamara}.

\begin{theorem}[\cite{McNamara}]
The set $\{G_{\lambda }(x|t)\}$ with $\lambda $ having at most $n$ parts is
a basis of $\Lambda _{n}\otimes \mathbb{Z}(\beta ,t_{1},\ldots ,t_{N})$.
\end{theorem}

\subsubsection{New results for factorial Grothendieck polynomials}

We expect the Grothendieck polynomials indexed by partitions which either
consist of a single column, $\lambda =1^{r}$, or row, $\lambda =r$, to be
the elementary building blocks for general $\lambda $. The following lemma
states a generating function for the $G_{1^{r}}(x|t)$'s.

\begin{proposition}
We have the equality%
\begin{equation}
\Pi (t_{\emptyset })\prod_{i=1}^{n}(u-
x_{i})=(u|t)^{n}+\sum_{r=1}^{n}(-1)^{r}G_{1^{r}}(x|t)(u|t)^{n-r}%
\prod_{i=1}^{r-1}(1+\beta u\oplus t_{n+1-i})  \label{-e}
\end{equation}%
and the identity%
\begin{equation}
G_{1^{r}}(x|t)=\sum_{j=1}^{n+1-r}\frac{\prod_{i=1}^{n}x_{i}\oplus t_{j}}{%
\prod_{i=1,i\neq j}^{n+1-r}t_{j}\ominus t_{i}}~,  \label{G1r}
\end{equation}%
where $r=1,\ldots ,n$.
\end{proposition}

\begin{proof}
First one derives the following equality involving the Vandermonde
determinant via induction,%
\begin{equation}
a_{n}(x|t)=\det [(x_{j}|t)^{n-i}]_{1\leq i,j\leq n}=\Delta
(x)\prod_{i=1}^{n}(1+\beta t_{i})^{n-i}\;.  \label{vandermonde}
\end{equation}
Then it follows that 
\begin{equation*}
\frac{a_{n+1}(u,x_{1},\ldots ,x_{n}|t)}{a_{n}(x_{1},\ldots ,x_{n}|t)}=\Pi
(t_{\emptyset })\prod_{i=1}^{n}(u-x_{i})\;.
\end{equation*}%
Expanding the determinant $a_{n+1}(u,x_{1},\ldots ,x_{n}|t)$ with respect to
the first column one obtains the first formula. Setting $u=\ominus t_{i} $
with $i=1,2,\ldots ,n$ results in a linear system with lower triangular
matrix which can be solved to obtain the second formula.
\end{proof}

Using the last result we now derive an alternative generating function for
the $G_{1^{r}}(x|\ominus t)$'s which will play an important role in what
follows.

\begin{corollary}
We have 
\begin{equation}
\prod_{i=1}^{n}(u\oplus x_{i})=(u|t)^{n}+\sum_{r=1}^{n}(u|t)^{n-r}(1+\beta
~u\oplus t_{n+1-r})G_{1^{r}}(x|\ominus t)\,.  \label{e}
\end{equation}%
Setting $t_{j}=0$ for all $j$ this becomes 
\begin{equation}
\prod_{i=1}^{n}(u\oplus x_{i})=u^{n}+(1+\beta
u)\sum_{r=1}^{n}u^{n-r}G_{1^{r}}(x_{1},\ldots ,x_{n})  \label{e0}
\end{equation}%
which implies for $r=1,2,\ldots ,n$ the identities%
\begin{equation}
e_{r}(x_{1},\ldots ,x_{n})=\sum_{s=r}^{n}(-\beta )^{s-r}\binom{s-1}{s-r}%
G_{1^{s}}(x_{1},\ldots ,x_{n})\;,  \label{er2G1r}
\end{equation}%
where $e_{r}(x_{1},\ldots ,x_{n})$ are the elementary symmetric polynomials.
\end{corollary}

\begin{proof}
Let $f(u)=\prod_{i=1}^{n}(u\oplus x_{i})$. This is a polynomial in $u$ of
degree $n$ with the coefficient of $u^{n}$ being $\Pi (x)$. Setting
successively $u=\ominus t_{1},\ominus t_{2},\ldots ,\ominus t_{n}$ one finds%
\begin{eqnarray*}
f(u) &=&(u|t)^{n}+(1+\beta u)\sum_{r=1}^{n}(u|t)^{n-r}(1+\beta
t_{n+1-r})f_{r} \\
&=&(u|t)^{n}(1+\beta f_{1})+(u|t)^{n-1}(f_{1}+\beta f_{2})+\cdots +f_{n}
\end{eqnarray*}%
with 
\begin{equation*}
f_{n+1-r}=\sum_{i=1}^{r}\frac{f(\ominus t_{i})}{\prod_{j=1,j\neq
i}^{r}t_{j}\ominus t_{i}}~,\quad r=1,2,\ldots ,n\;.
\end{equation*}%
The identity (\ref{e}) then follows from (\ref{G1}) and (\ref{G1r}). Setting 
$t_{1}=\cdots =t_{n}=0$ in (\ref{e}) we arrive at (\ref{e0}).

Finally, we have 
\begin{eqnarray*}
\prod_{i=1}^{n}(u-x_{i}) &=&(1+\beta u)^{n}(-1)^{n}f(\ominus u) \\
&=&u^{n}+\sum_{r=1}^{n}(-1)^{r}(1+\beta
u)^{r-1}u^{n-r}G_{1^{r}}(x_{1},\ldots ,x_{n})
\end{eqnarray*}%
and comparing powers of $u$ on both sides of the equality sign the last
assertion now follows.
\end{proof}

As in the case of factorial Schur functions \cite[Chap I.3, Ex. 20]%
{MacdonaldBook} define a \emph{shift operator} $\tau $ by 
\begin{equation}
(x|\tau ^{m}t)^{n}=\prod_{j=1}^{n}(x\oplus t_{j+m}),\qquad m\in \mathbb{Z}\,.
\label{shifted power}
\end{equation}%
We wish to derive an analogue of the Jacobi-Trudy identity for factorial
Schur functions. To this end we require the following result first.

\begin{lemma}
We have the expression 
\begin{equation}
G_{r}(x|t)=\sum_{i=1}^{n}(x_{i}|t)^{n+r-1}\prod_{j\neq i}\frac{1}{%
x_{i}\ominus x_{j}}  \label{Gr}
\end{equation}%
and the following equality between determinants, 
\begin{equation}
\det [G_{\lambda _{i}-i+j}(x|\tau ^{1-j}t)]_{1\leq i,j\leq n}=\frac{\det
[(x_{j}|t)^{n+\lambda _{i}-i}]_{1\leq i,j\leq n}}{\det
[(x_{j}|t)^{n-i}]_{1\leq i,j\leq n}}  \label{JTG}
\end{equation}%
where $(x|t)^{m}$ is defined in (\ref{facpower}).
\end{lemma}

\begin{proof}
The proof follows along the same steps as the proof for the analogous
identities in the case of factorial Schur functions; see e.g. the section on
the \textquotedblleft 6th variation\textquotedblright\ in \cite{Macdonald}
and \cite[Ch. I.3, Ex. 20]{Macdonald}. We therefore omit the details.
\end{proof}

While it would be desirable to have a single determinant in the $G_{r}$'s
expressing the Grothendieck polynomial $G_{\lambda }$, this seems in general
not possible. Instead we obtain an expression in terms of sums of
determinants which involve the polynomials in (\ref{JTG}) 
\begin{equation}
F_{\lambda }(x|t)=\frac{\det [(x_{j}|t)^{n+\lambda _{i}-i}]_{1\leq i,j\leq n}%
}{\det [(x_{j}|t)^{n-i}]_{1\leq i,j\leq n}}  \label{Fdet}
\end{equation}%
Note that $F_{\lambda }(x|t)=s_{\lambda }(x|t)$ for $\beta =0$ and $%
F_{\lambda }(x|0)=s_{\lambda }(x)$, that is the $F_{\lambda }$'s do not
specialise to the ordinary (non-factorial) Grothendieck polynomial for $%
t_{j}=0$. We shall therefore treat this case separately.

Before we can state the expansion formula of $G_{\lambda }$ into $F_{\lambda
}$'s we require the following technical result.

\begin{lemma}
\begin{equation}
(1+\beta u)^{r}(u|\ominus t)^{n-r}=\sum_{i=0}^{r}(u|\ominus t)^{n-i}\Gamma
_{i}(r,n)  \label{facpower_id}
\end{equation}%
where the coefficients are given by%
\begin{equation}
\Gamma _{i}(r,n)=\beta ^{r-i}\prod_{j=i}^{r-1}(1+\beta t_{n-j})\sum_{i-1\leq
j_{1}\leq \cdots \leq j_{i}\leq r-1}\prod_{l=1}^{i}(1+\beta t_{n-j_{l}})
\label{Gamma}
\end{equation}%
Explicitly, 
\begin{eqnarray*}
\Gamma _{0}(r,n) &=&\beta ^{r}\prod_{j=0}^{r-1}(1+\beta t_{n-j}) \\
\Gamma _{1}(r,n) &=&\beta ^{r-1}\prod_{j=1}^{r-1}(1+\beta
t_{n-j})\sum_{j=0}^{r-1}(1+\beta t_{n-j}) \\
\Gamma _{2}(r,n) &=&\beta ^{r-2}\prod_{j=2}^{r-1}(1+\beta
t_{n-j})\sum_{j=1}^{r-1}(1+\beta t_{n-j})\sum_{i=1}^{j}(1+\beta t_{n-i}) \\
&&\vdots \\
\Gamma _{r}(r,n) &=&(1+\beta t_{n+1-r})^{r}
\end{eqnarray*}
\end{lemma}

\begin{proof}
Use the simple identity%
\begin{equation*}
(1+\beta u)(u|\ominus t)^{n}=(1+\beta t_{n+1})[(u|\ominus t)^{n}+\beta
(u|\ominus t)^{n+1}]
\end{equation*}%
to find the recurrence relation%
\begin{equation*}
\Gamma _{i}(r,n)=(1+\beta t_{n+1-r})(\Gamma _{i-1}(r-1,n-1)+\beta \Gamma
_{i}(r-1,n))\;.
\end{equation*}%
Here $\Gamma _{i}=0$ for $i<0$. Defining%
\begin{equation*}
\Gamma _{i}(r,n)=\gamma _{i}(r,n)\beta ^{r-i}\prod_{j=i}^{r-1}(1+\beta
t_{n-j})
\end{equation*}%
The recurrence relation simplifies to 
\begin{equation*}
\gamma _{i}(r,n)=(1+\beta t_{n+1-r})\gamma _{i-1}(r-1,n-1)+\gamma
_{i}(r-1,n))
\end{equation*}%
and can now be successively solved starting from $\gamma _{0}(r,n)=1$.
\end{proof}

We now state a generalised Jacobi-Trudy identity for factorial Grothendieck
polynomials which simplifies for $\beta=0$ to the known Jacobi-Trudy
identity for factorial Schur functions. We state it for the parameters $%
\ominus t$ as it is in this form that we will use the identity later on in
this article, but making the replacement $t\rightarrow \ominus t$ in the
formula and the coefficients (\ref{Gamma}) is straightforward.

\begin{proposition}
\label{prop:Gexp}Let $x=(x_{1},\ldots ,x_{n})$ and $\lambda $ a partition
with at most $n$ parts. Then%
\begin{equation}
G_{\lambda }(x|\ominus t)=\sum_{\alpha }\beta ^{|\alpha |}\phi _{\alpha
}(\lambda )F_{\lambda +\alpha }(x|\ominus t),  \label{G2F}
\end{equation}%
where the sum runs over all \emph{compositions} $\alpha =(0,\alpha
_{2}\ldots ,\alpha _{n})$ with $0\leq \alpha _{i}\leq i-1$ and%
\begin{equation}
\phi _{\alpha }(\lambda )=\frac{\prod_{i=2}^{n}\varphi _{\alpha
_{i}}(\lambda _{i})}{\prod_{i=1}^{n}(1+\beta t_{i})^{n-i}},\qquad \beta
^{\alpha _{i}}\varphi _{\alpha _{i}}(\lambda _{i})=\Gamma _{i-1-\alpha
_{i}}(i-1,n+\lambda _{i}-1)\;.  \label{Gamma2}
\end{equation}
\end{proposition}

N.B. the determinant formula (\ref{Fdet}) for $F_{\alpha }$ is well-defined
for \emph{compositions} $\alpha $ by which we mean finite sequences of
non-negative integers which are not necessarily weakly decreasing. Any such $%
F_{\alpha }$ can be expressed in terms of $F_{\lambda }$'s indexed by
partitions $\lambda $ using the same straightening rules which hold for
Schur functions,%
\begin{equation}
F_{(\ldots ,a,b,\ldots )}=-F_{(\ldots ,b-1,a+1,\ldots )}\qquad \text{%
and\qquad }F_{(\ldots ,a,a+1,\ldots )}=0\;.  \label{Fstraight}
\end{equation}%
Both rules should be obvious from (\ref{Fdet}), the first rule follows from
exchanging two rows in the determinant in the numerator of (\ref{Fdet}),
while the second is simply a result of two rows being linearly dependent.

\begin{proof}
Employ the previous lemma and the formula (\ref{facGdet}) to find%
\begin{eqnarray*}
(x_{j}|t)^{\lambda _{i}+n-i}(1+\beta x_{j})^{i-1} &=&\sum_{\alpha
_{i}=0}^{i-1}(x_{j}|\ominus t)^{n+\lambda _{i}-1-\alpha _{i}}\Gamma _{\alpha
_{i}}(i-1,n+\lambda _{i}-1) \\
&=&\sum_{\alpha _{i}=0}^{i-1}\beta ^{\alpha _{i}}\varphi _{\alpha
_{i}}(\lambda _{i})(x_{j}|\ominus t)^{n+\lambda _{i}-i+\alpha _{i}}\;.
\end{eqnarray*}%
The assertion now follows from row-linearity of the determinant and (\ref%
{vandermonde}). 
\end{proof}

\begin{example}
Let $n=2$. Then the compositions $\alpha $ in the sum in (\ref{G2F}) are $%
\alpha =(0,0)$ and $\alpha =(0,1)$. We find from (\ref{Gamma}) and (\ref%
{Gamma2}) that $\Gamma _{0}(0,\lambda _{1}+1)=1$, $\Gamma _{1}(1,\lambda
_{2}+2)=1+\beta t_{\lambda _{2}+1}$ and $\Gamma _{0}(1,\lambda _{2}+2)=\beta
(1+\beta t_{\lambda _{2}+1})$. Hence, we arrive at%
\begin{equation}
G_{\lambda _{1},\lambda _{2}}(x|\ominus t)=\frac{1+\beta t_{\lambda _{2}+1}}{%
1+\beta t_{1}}~(F_{\lambda _{1},\lambda _{2}}+\beta F_{\lambda _{1},\lambda
_{2}+1})\;.  \label{G2Fex}
\end{equation}
\end{example}

The analogous expansion of $G_{\lambda }$ for the non-factorial case
corresponds to an expansion into Schur functions instead.

\begin{proposition}
Set $t_{j}=0$ then%
\begin{equation}
G_{\lambda }(x)=\sum_{\alpha }\beta ^{|\alpha |}\prod_{i=1}^{n-1}\binom{i}{%
\alpha _{i}}~s_{\lambda +\alpha }(x),\;  \label{G2Schur}
\end{equation}%
where the sum runs over all \emph{compositions} $\alpha =(0,\alpha
_{1},\ldots ,\alpha _{n-1})$ with $0\leq \alpha _{i}\leq i$ and $s_{\alpha
}(x)=\det (x_{j}^{n+\alpha _{i}-i})/\det (x_{j}^{n-i})$ is the (generalised)
Schur function with $\alpha $ being a composition.
\end{proposition}

\begin{proof}
Use the binomial theorem and row-linearity of the determinant.
\end{proof}

Using the Yang-Baxter algebra we will prove below the following special case
of a Cauchy identity.

\begin{proposition}\label{prop:cauchy0}
Let $\mu \subset (k^{n})$ be a partition inside the $n\times k$ bounding
box, then%
\begin{eqnarray}
\prod_{i=1}^{n}\prod_{j\in I_{\mu ^{\ast }}}x_{i}\ominus t_{j}
&=&\sum_{\lambda \subset (k^{n})}G_{\lambda }(x|\ominus t)G_{\lambda ^{\vee }}(t_{\mu
}|\ominus t^{\prime })\frac{\Pi (t_{\mu })}{\Pi (t_{\lambda })}
\label{Cauchy0} \\
&=&\sum_{\lambda \subset (k^{n}) }G_{\lambda }(x|\ominus t)G_{\lambda ^{\ast }}(\ominus
t_{\mu ^{\ast }}|t)\frac{\Pi (t_{\lambda ^{\ast }})}{\Pi (t_{\mu ^{\ast }})}
\label{Cauchy0a}
\end{eqnarray}%
where the second equality follows from the stronger identity 
\begin{equation}
G_{\lambda ^{\vee }}(\ominus t_{\mu }|t^{\prime })=G_{\lambda ^{\ast
}}(t_{\mu ^{\ast }}|\ominus t),\qquad t^{\prime }=w_{0}t\;.
\label{Gduality0}
\end{equation}
\end{proposition}

\begin{proof}
See Corollary \ref{cor:cauchy0} and Prop \ref{prop:Gduality}.
\end{proof}

\section{Yang-Baxter Algebras}

This section contains the main algebraic setup for the definition of the
hierarchy of generalised equivariant quantum cohomologies $qh_{n}^{\ast }$.
As explained earlier these are realised as commutative subalgebras of a
larger non-commutative algebra, the Yang-Baxter algebra, which then
naturally acts on the direct sum $\tbigoplus\nolimits_{n=0}^{N}qh_{n}^{\ast
} $. 

The Yang-Baxter algebras are constructed in terms of a graphical calculus, 
non-intersecting lattice paths, which we now briefly recall from \cite[Sec 3]{VicOsc}.

\subsection{Non-intersecting lattice paths and 5-vertex models}
Fix two integers $N>0$ and $0\leq r\leq N$ and consider the square lattice 
\begin{equation}
\mathbb{L}_r:=\{\langle i,j\rangle \in \mathbb{Z}^{2}|0\leq i\leq r+1,\;0\leq
j\leq N+1\}\,.  \label{lattice}
\end{equation}%
Denote by $\mathbb{E}=\{(p,p^{\prime })\in \mathbb{L}_r^{2}~|~p_{1}+1=p_{1}^{%
\prime },\,p_{2}=p_{2}^{\prime }\,\text{or}\,p_{1}=p_{1}^{\prime
},\,p_{2}+1=p_{2}^{\prime }\}$ the set of horizontal and vertical edges. A \emph{lattice configuration}\ $\mathcal{C}:\mathbb{E}\rightarrow \{0,1\}$
is an assignment of values $0$ or $1$ to the lattice edges. We are interested in the weighted counting of lattice configurations. 

The weight of 
a lattice configuration will be fixed in terms of weights of its vertex configurations: 
consider the vertex obtained by intersecting the $i^{%
\text{th}}$ horizontal lattice line with the $j^{\text{th}}$ vertical one, then a {\em vertex configuration} 
is a 4-tuple $\mathrm{v}_{i,j}=(\varepsilon_W,\varepsilon_N,\varepsilon_E,\varepsilon_S)$ where respectively $\varepsilon_W,\varepsilon_N,\varepsilon_E,\varepsilon_S=0,1$ are the values 
of the W, N, E, S edges at the lattice point $\langle i,j\rangle $. Obviously, each lattice 
configuration fixes uniquely the configuration at a vertex and specifying a configuration 
at each vertex fixes a lattice configuration. 

Define two sets of vertex weights 
$\operatorname{wt}(\mathrm{v}_{i,j})\in\mathcal{R}(\mathbb{T})[x_1,\ldots,x_r]$ and
$\operatorname{wt}'(\mathrm{v}_{i,j})\in\mathcal{R}(\mathbb{T})[x_1,\ldots,x_r]$
as shown on the top and bottom of Figure \ref{fig:5vmodels}, where  $x=(x_{1},x_{2},\ldots )$ is a set of 
commuting indeterminates which we call \emph{spectral variables}. In each case there are only 5 vertex
configurations with nonzero weights, for all the other possible vertex configurations we set the weight 
to zero and call the corresponding vertex configuration `forbidden'. The indices $i=1,\ldots,r$ and 
$j=1,\ldots,N$ in the weights of Figure \ref{fig:5vmodels} refer to the row and column number of the 
lattice, respectively. In the area of statistical physics these weights would be evaluated in the positive real numbers by 
fixing concrete values for the $x_i$ and $t_j$'s and be interpreted as a Boltzmann weight, 
fixed by the energy of a local vertex configuration of the physical model at hand.

The {\em weight of a lattice configuration} $\mathcal{C}$ can now be defined as the product over
its vertex weights, 
\begin{equation}
\operatorname{wt}(\mathcal{C})=\prod_{(i,j)\in \mathbb{L}}\operatorname{wt}(\mathrm{v}%
_{i,j})\;,  \label{wtC}
\end{equation}%
and the resulting allowed lattice configurations, i.e. those with nonzero weight, describe non-intersecting 
lattice paths. Due to the different nature of the resulting non-intersecting paths, the ones originating from the weights in the 
top row of Figure \ref{fig:5vmodels} have been called {\em vicious walkers} and the ones originating from the weights in the bottom row {\em osculating walkers} in the physics literature on random walks. Their weighted counting leads to the following partition functions
\begin{equation}
\mathcal{Z}(x_1,\ldots,x_r|t_1,\ldots,t_N)=\sum_{\mathcal{C}}\operatorname{wt}(\mathcal{C})
\in \mathcal{R}(\mathbb{T})[x_{1},\ldots ,x_{r}] \label{partition_function}\,.
\end{equation}
For the simplest case $r=1$, a single lattice row, we will use the partition functions to define matrix elements of 
the generators of certain Yang-Baxter algebras. To do so, we first rewrite the sum over lattice configurations in 
\eqref{partition_function} in terms of matrix products by introducing 
a suitable vector space and interpreting the values assigned to lattice edges as labels of basis vectors in that space.

\subsection{Quantum space and spin bases}\label{sec:quantumspace}

Let $V=%
\mathbb{Z}v_{0}\oplus \mathbb{Z}v_{1}$ and denote by $\sigma ^{-}=\left( 
\begin{smallmatrix}
0 & 1 \\ 
0 & 0%
\end{smallmatrix}%
\right) ,\;\sigma ^{+}=\left( 
\begin{smallmatrix}
0 & 0 \\ 
1 & 0%
\end{smallmatrix}%
\right) ,\;\sigma ^{z}=\left( 
\begin{smallmatrix}
1 & 0 \\ 
0 & -1%
\end{smallmatrix}%
\right) $ the fundamental representation of $sl_{2}$, the Pauli matrices,
acting on $V$ via $\sigma ^{-}v_{1}=v_{0}$, $\sigma ^{+}v_{0}=v_{1}$ and $%
\sigma ^{z}v_{\alpha }=(-1)^{\alpha }v_{\alpha },~\alpha =0,1$. Define the following \textquotedblleft spin
basis\textquotedblright\ $\{v_{\lambda (b)}\}\subset V^{\otimes N}$ where $b$
runs over all binary strings of length $N$ and%
\begin{equation}
v_{\lambda (b)}=v_{b_{1}}\otimes v_{b_{2}}\otimes \cdots \otimes v_{b_{N}}\;.
\label{spin basis}
\end{equation}%
We will also need the dual spin basis which we shall denote by $\{\tilde{v}%
_{\lambda }\}\subset \tilde{V}^{\otimes N},$ with $\tilde{V}$ being the dual
space of $V$, and use the familiar bracket notation $\langle \tilde{v}%
_{\lambda }|v_{\mu }\rangle =\delta _{\lambda \mu }$.  

Given a lattice configuration $\mathcal{C}$ the vertical edge values in each row 
fix a spin basis vector and our goal is to write the partition function (\ref{partition_function}) as a polynomial in the spectral variables $x_i$ whose coefficients are matrix elements of certain operators (which we are going to define below) with respect to the spin basis (\ref{spin basis}).  Since the weights of lattice configurations are rational functions in the equivariant parameters $t_j$ (and polynomials in the spectral parameters $x_i$) we introduce the following tensor product 
\begin{equation}
\mathcal{V}=\tbigotimes_{j=1}^{N}V(t_{j})\cong \mathcal{R}\mathbb{(}%
t_{1},\ldots ,t_{N})\otimes V^{\otimes N}~,  \label{quantum_space}
\end{equation}%
where $V(t_{j}):=\mathcal{R}(t_{j})\otimes V$ etc. Here we have
dropped the dependence on $\beta $ in the notation to simplify formulae. The latter space 
is called the \emph{quantum space} in the area of quantum integrable systems, the Yang-Baxter algebra will be defined as a subalgebra 
$\subset\operatorname{End}\mathcal{V}$. 

There is a natural $U(sl_{2})$-action on $V^{\otimes N}$. Fix a Cartan
subalgebra $\mathfrak{h}$ then we have the decomposition $V^{\otimes
N}=\tbigoplus\limits_{0\leq n\leq N}V_{n}$ into $U(\mathfrak{h})$-weight
spaces where $V_{n}\subset V^{\otimes N}$ denotes the subspace which is
spanned by $\{v_{\lambda (b)}\}_{|b|=n}$, i.e. the basis vectors indexed by
binary strings with $n$ 1-letters. This induces an analogous decomposition
of the quantum space $\mathcal{V}$ into the subspaces $\mathcal{V}_{n}=%
\mathcal{R}(\mathbb{T})\otimes V_{n}$. Below we shall identify for each
subspace $\mathcal{V}_{n}$ the basis (\ref{spin basis}) with the Schubert
basis in $qh_{n}^{\ast }$ and (as vector spaces) $\tbigoplus\nolimits_{0\leq n\leq
N}qh_{n}^{\ast }$ with $\mathcal{V}$.

\subsection{Solutions to the Yang-Baxter equation}\label{sec:YBA}

One might wonder what singles out the particular weights from Figure \ref{fig:5vmodels}. We now show that 
they define solutions to the Yang-Baxter equation and as a result the partition function (\ref{partition_function}), 
for certain boundary conditions, will be a symmetric polynomial in the spectral variables $x_i$. In short, the 
Yang-Baxter equation encodes how the partition function changes under a permutation of the $x_i$'s or 
under a permutation of the $t_j$'s. Both transformation properties will be important. 

As in the case of the equivariant
parameters $t_{j}$ we set $V(x_{i}):=\mathcal{R}(x_{i})\otimes V$. Define the
following $L$-operators in $\operatorname{End}_{\mathcal{R}(x_i,t_j)}[V(x_{i})\otimes V(t_{j})]$ by setting%
\begin{equation}
L(x_i|t_j)=\left( 
\begin{array}{cc}
\sigma^{+}\sigma^{-}+x_{i}\ominus t_{j}~\sigma^{-}\sigma^{+} & (1+\beta
x_{i}\ominus t_{j})\sigma^{+} \\ 
\sigma^{-} & \sigma^{-}\sigma^{+}%
\end{array}%
\right)  \label{L}
\end{equation}%
and%
\begin{equation}
L^{\prime }(x_i|t_j)=\left( 
\begin{array}{cc}
\sigma^{-}\sigma^{+}+x_{i}\oplus t_{j}~\sigma^{+}\sigma^{-} & \sigma^{+} \\ 
(1+\beta x_{i}\oplus t_{j})\sigma^{-} & \sigma^{+}\sigma^{-}%
\end{array}%
\right)  \label{L'}
\end{equation}%
where the matrix notation is to be read as the decomposition $$%
L(x_i|t_j)=\sum_{a,b=0,1}e_{ab}\otimes L^{(ab)}(x_i|t_j)$$ with respect to the first factor in $%
V(x_{i})\otimes V(t_{j})$ and $e_{ab}$ denote the $2\times 2$ unit matrices. So, $L^{(01)}(x_i|t_j)=\sigma^+$, 
$L^{(11)}(x_i|t_j)=\sigma^+\sigma^-$ etc.

The matrix elements of these $L$-operators can be identified with the weights
for the vertex configurations from Figure \ref{fig:5vmodels} using the same conventions as in \cite[Sec 3]%
{VicOsc}. Namely, define $L(x_i|t_j)_{\varepsilon _{1}\varepsilon
_{2}}^{\varepsilon _{1}^{\prime }\varepsilon _{2}^{\prime }}$ and $%
L^{\prime }(x_i|t_j)_{\varepsilon _{1}\varepsilon _{2}}^{\varepsilon _{1}^{\prime
}\varepsilon _{2}^{\prime }}$ via the expansion 
$$L(x_i|t_j)v_{\varepsilon
_{1}}\otimes v_{\varepsilon _{2}}=\sum_{\varepsilon _{1}^{\prime
},\varepsilon _{2}^{\prime }=0,1}L(x_i|t_j)_{\varepsilon _{1}\varepsilon
_{2}}^{\varepsilon _{1}^{\prime }\varepsilon _{2}^{\prime }}\,v_{\varepsilon
_{1}^{\prime }}\otimes v_{\varepsilon _{2}^{\prime }}$$ 
with $\varepsilon
_{i},\varepsilon _{i}^{\prime }=0,1$. Then the coefficients can be
explicitly computed from (\ref{L}), (\ref{L'}). They are the weights of the
vertex configurations given in Figure \ref{fig:5vmodels} where $\varepsilon
_{1},\varepsilon _{2},\varepsilon _{1}^{\prime },\varepsilon _{2}^{\prime }$
are the values of the W, N, E and S edge of the vertex. For those vertex configurations 
that are not shown in Figure \ref{fig:5vmodels} we have $L(x_i|t_j)_{\varepsilon _{1}\varepsilon
_{2}}^{\varepsilon _{1}^{\prime }\varepsilon _{2}^{\prime }}=0$ and $L\rq{}(x_i|t_j)_{\varepsilon _{1}\varepsilon
_{2}}^{\varepsilon _{1}^{\prime }\varepsilon _{2}^{\prime }}=0$.

\begin{figure}[tbp]
\begin{equation*}
\includegraphics[scale=0.32]{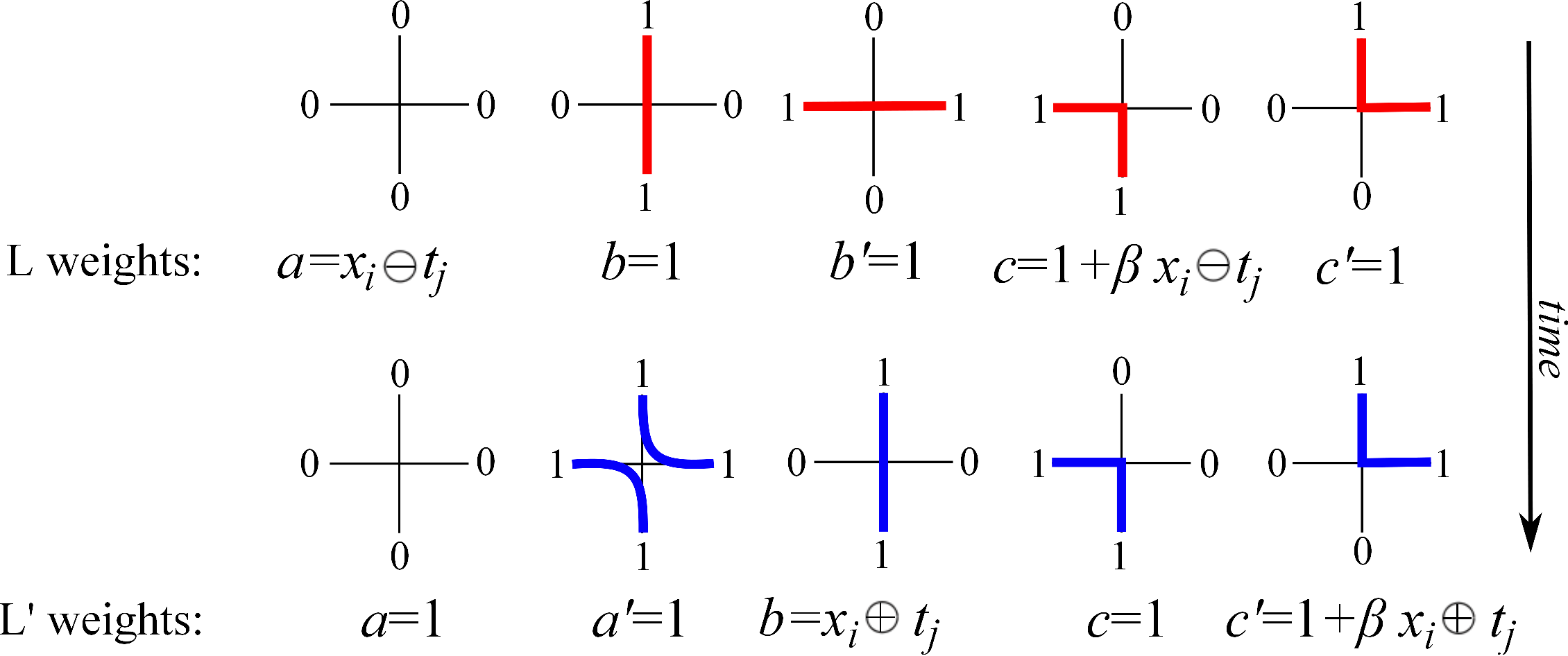}
\end{equation*}%
\caption{Graphical depiction of the matrix elements of the $L$-operators (%
\protect\ref{L}) and (\protect\ref{L'}) as weighted vertex configurations.}
\label{fig:5vmodels}
\end{figure}

In what follows we will consider multiple tensor products $V(x_{i_1})\otimes\cdots\otimes V(x_{i_r})\otimes V(t_{j_1})\otimes\cdots\otimes V(t_{j_s})$ and more complicated operators consisting of several $L$-operators which possibly depend on multiple 
spectral variables $x_{i_1},\ldots,x_{i_r}$ and/or multiple equivariant parameters $t_{j_1},\ldots,t_{j_s}$.  In this case it will always be understood that we are only considering endomorphisms which are invariant with respect to the appropriate ring $\mathcal{R}(x_{i_1},\ldots,x_{i_r},t_{j_1},\ldots,t_{j_s})$ but shall simply write $\operatorname{End}[
V(x_{i_1})\otimes\cdots\otimes V(x_{i_r})\otimes V(t_{j_1})\otimes\cdots\otimes V(t_{j_s})]$ in order to unburden the notation somewhat.

For example, consider the triple tensor product $V(x_i)\otimes V(x_{i\rq{}})\otimes V(t_j)$ and label the factors in it from left to right with 1, 2 and 3. The next proposition states identities in $\operatorname{End}[V(x_i)\otimes V(x_{i\rq{}})\otimes V(t_j)]$, known as Yang-Baxter equations, where we use the standard index notation such as $L_{13}(x_i|t_j)$ to define an element in $\operatorname{End}[V(x_i)\otimes V(x_{i\rq{}})\otimes V(t_j)]$ which acts trivially, as the identity, on the second space $V(x_{i\rq{}})$ and non-trivially on the first and third space as the operator $L(x_i|t_j)\in\operatorname{End}[V(x_i) \otimes V(t_j)]$ defined in (\ref{L}). Similarly, we define $L_{23}(x_i|t_j)\in \operatorname{End}[V(x_i)\otimes V(x_{i\rq{}})\otimes V(t_j)]$ as the element acting trivially on the first space $V(x_i)$ and non-trivially on the second and third space, etc. We shall use this index notation throughout this article. 

\begin{proposition}
\label{prop:ybe} There exist endomorphisms $R(x_i,x_{i\rq{}})\in\operatorname{End}[V(x_i)\otimes V(x_{i\rq{}})]$ and 
$r(t_j,t_{j\rq{}})\in\operatorname{End}[V(t_j)\otimes V(t_{j\rq{}})]$ obeying the following Yang-Baxter equations in 
$\operatorname{End}[V(x_i)\otimes V(x_{i\rq{}})\otimes V(t_j)]$,%
\begin{equation}
R_{12}(x_i, x_{i^{\prime}})L_{13}(x_i|t_j)L_{23}(x_{i^{\prime }}|t_j)=
L_{23}(x_{i^{\prime }}|t_j)L_{13}(x_i|t_j)R_{12}(x_i, x_{i^{\prime }}),
\label{YBE}
\end{equation}%
and in $\operatorname{End}[V(x_i)\otimes V(t_{j})\otimes V(t_{j\rq{}})]$,
\begin{equation}
r_{23}(t_j, t_{j^{\prime}})L_{12}(x_i|t_j)L_{13}(x_i|t_{j^{\prime}})=
L_{13}(x_i|t_{j^{\prime}})L_{12}(x_i|t_j)r_{23}(t_j, t_{j^{\prime}})\,.
\label{ybe}
\end{equation}%
Analogous identities are obtained for the $L^\prime$-operator (\ref{L'}) and we denote the corresponding endomorphisms by $R^{\prime}(x_i,x_{i\rq{}}),r^{\prime}(t_j,t_{j\rq{}})$. 

The $R,R\rq{}$ and $r,r\rq{}$-operators can be identified with $4\times 4$ matrices acting respectively in $%
V(x_i)\otimes V(x_{i^{\prime}})$ and $V(t_j)\otimes V(t_{j^{\prime}})$ by fixing the basis vectors $\{v_0\otimes v_0, v_0\otimes v_1,v_1\otimes
v_0,v_1\otimes v_1\}$. They are all of the general form%
\begin{equation}
\left( 
\begin{array}{cccc}
a & 0 & 0 & 0 \\ 
0 & b & c & 0 \\ 
0 & c^{\prime } & b^{\prime } & 0 \\ 
0 & 0 & 0 & a^{\prime }%
\end{array}%
\right) \;,  \label{R}
\end{equation}%
with the matrix entries given in the following table for each of the
respective cases:%
\begin{equation}
\begin{tabular}{|l||c|c|c|c|c|c|}
\hline
& $a$ & $b$ & $c$ & $c^{\prime }$ & $b^{\prime }$ & $a^{\prime }$ \\ 
\hline\hline
\multicolumn{1}{|c||}{$R(x_i,x_{i\rq{}})$} & $1$ & $0$ & $1$ & $1+\beta x_{i^{\prime
}}\ominus x_{i}$ & $x_{i^{\prime }}\ominus x_{i}$ & $1$ \\ \hline
\multicolumn{1}{|c||}{$R^{\prime}(x_i,x_{i\rq{}})$} & $1$ & $x_{i}\ominus x_{i^{\prime }}$ & 
$1$ & $1+\beta x_{i}\ominus x_{i^{\prime }}$ & $0$ & $1$ \\ \hline
\multicolumn{1}{|c||}{$r(t_j,t_{j\rq{}})$} & $1$ & $0$ & $1+\beta t_{j}\ominus
t_{j^{\prime }}$ & $1$ & $t_{j}\ominus t_{j^{\prime }}$ & $1$ \\ \hline
\multicolumn{1}{|c||}{$r^\prime(t_j,t_{j\rq{}})$} & $1$ & $0$ & $1+\beta t_{j}\ominus
t_{j^{\prime }}$ & $1$ & $t_{j}\ominus t_{j^{\prime }}$ & $1$ \\ \hline
\end{tabular}
\label{ybesolns}
\end{equation}
So, in particular, we have that $r^{\prime}(t_j,t_{j\rq{}})=r(t_j,t_{j\rq{}})$.
\end{proposition}
\begin{proof}
A straightforward but rather tedious and lengthy computation which we omit.
\end{proof}

\begin{remark}
The Lax operators (\ref{L}) and (\ref{L'}) are 5-vertex degenerations of the
asymmetric 6-vertex model which is used to model ferroelectrics in external
electromagnetic fields \cite{Baxter}. The solutions (\ref{L}), (\ref{L'})
and (\ref{ybesolns}) of the Yang-Baxter equation are special cases of this
more general model. It is known that solutions of the form (\ref{R}) exist
if the Boltzmann weights $(a,a^{\prime },b,b^{\prime },c,c^{\prime })$ for
each $R$-matrix in the Yang-Baxter equation yield constant values for the
following two ratios,%
\begin{equation}
\Delta =\frac{aa^{\prime }+bb^{\prime }-cc^{\prime }}{2ab},\qquad \Gamma =%
\frac{a^{\prime }b^{\prime }}{ab}\;.  \label{quadrics}
\end{equation}%
This statement is originally due to Baxter \cite{Baxter} but can also be
found in e.g. \cite{Bumpetal}. For (\ref{L'}) we find $\Delta =-\beta /2$
and $\Gamma =0$ and the same values apply also to (\ref{L}) after
\textquotedblleft spin-reversal\textquotedblright , i.e. exchanging 0 and
1-letters. The special point $\beta =0$ corresponds to the so-called free
fermion point, while $\beta =-1$ is the value where connections with the
alternating sign matrix conjecture and counting of plane partitions have
been made in the literature; see e.g. \cite{Bressoud} and \cite{PZJ} as well
as references therein.
\end{remark}

In what follows we concentrate on the solutions of (\ref{YBE}) but when
discussing Goresky-Kottwitz-MacPherson theory towards the end of this
article, the solutions of (\ref{ybe}) will become important.

\subsection{Monodromy matrices}

\begin{figure}[tbp]
\begin{equation*}
\includegraphics[scale=0.5]{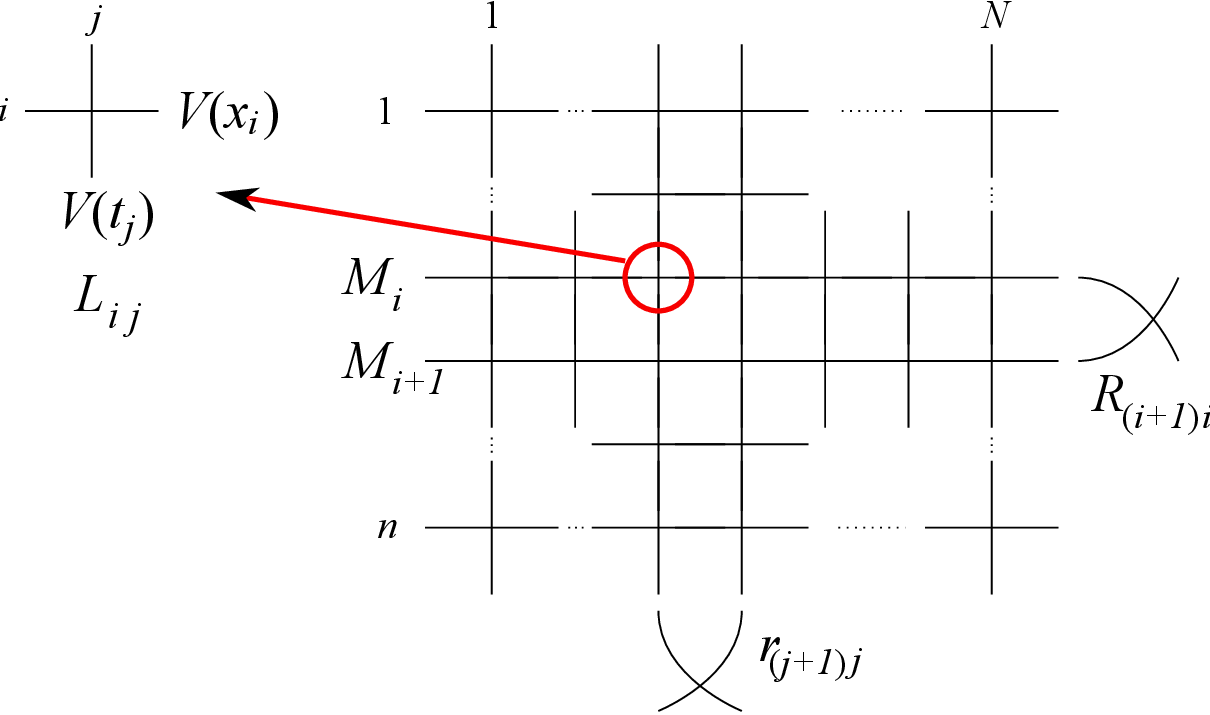}
\end{equation*}%
\caption{Graphical depiction of the $L$-operators and the monodromy
matrices. Each operator $L_{ij}$ is represented by a vertex in the $i$th row
and $j$th column. The square lattice on the right then represents the
operator (\protect\ref{rowmom}) over the tensor product $W_{n}\otimes 
\mathcal{V}$ obtained by reading out the lattice rows right to left, $%
M_{n}\cdots M_{2}M_{1}$. Braiding two lattice rows or two lattice columns
leads to the matrices $R_{i+1,i}$ and $r_{j+1,j}$, respectively.}
\label{fig:eqc_moms}
\end{figure}

We will now consider square lattices $\mathbb{L}_r$ where the number $N$ of 
columns is linked to the dimension of the ambient space $N=n+k$
of $\limfunc{Gr}_{n,N},$ and the number of rows $r$ to the dimension $n$ of the hyperplanes or their
co-dimension $k$. Namely, we consider the so-called \emph{auxiliary spaces }%
\begin{equation}
W_{r}=\tbigotimes_{i=1}^{r}V(x_{i})\cong \mathcal{R}(x_{1},\ldots
,x_{r})\otimes V^{\otimes r},\quad \quad r=n,k  \label{aux_space}
\end{equation}%
and associate the tensor product $W_{n}\otimes \mathcal{V}$ with a $n\times
N$ square lattice (for the vicious walker model) and $W_{k}\otimes \mathcal{V}$
with a $k\times N$ square lattice (for the osculating walker model); compare
with \cite{VicOsc}. 

Employing the same index notation as previously in the Yang-Baxter equation (\ref{YBE}), let 
$L_{ij}=L_{ij}(x_{i}|t_{j})\in\operatorname{End}(W_n\otimes\mathcal{V})$ denote the 
operators which act non-trivially only in the $i$th row and $j$th column of the lattice, i.e. the $i$%
th factor in the tensor product $W_{n}$ and the $j$th factor in $\mathcal{V}$, 
where their action coincides with that of $L(x_{i}|t_{j})\in\operatorname{End}[V(x_i)\otimes V(t_j)]$; see Figure \ref{fig:eqc_moms}. As explained in the previous section their matrix elements are the weights for the vertex configurations 
$\mathrm{v}_{ij}$.

In order to obtain the partition function (\ref{partition_function}) we now consider 
the following operator $\boldsymbol{Z}%
:W_{n}\otimes \mathcal{V}\rightarrow W_{n}\otimes \mathcal{V}$ 
\begin{equation}
\boldsymbol{Z}=M_{n}\cdots M_{2}M_{1},\qquad M_{i}=L_{iN}\cdots L_{i2}L_{i1}\;,
\label{rowmom}
\end{equation}%
where $M(x_{i}|t):V(x_{i})\otimes \mathcal{V}\rightarrow
V(x_{i})\otimes \mathcal{V}$ is called \emph{row-monodromy matrix} and the index notation 
$M_i=M_{i}(x_i|t)\in\operatorname{End}[W_n\otimes\mathcal{V}]$ indicates the factor (lattice row) 
of $W_n$ in which $M_i$ acts as $M(x_i|t)$; see Figure \ref{fig:eqc_moms}. Note that its action on each factor in the quantum space 
$\mathcal{V}$ is non-trivial, hence the second index is omitted.

\begin{corollary}
The row monodromy matrices obey the following Yang-Baxter equation in 
$\operatorname{End}[V(x_i)\otimes V(x_{i'})\otimes\mathcal{V}]$,
\begin{equation}
R_{12}(x_i, x_{i^{\prime }})M_{1}(x_i|t)M_{2}(x_{i^{\prime }}|t)=
M_{2}(x_{i^{\prime }}|t)M_{1}(x_i|t)R_{12}(x_i, x_{i^{\prime }})
\label{mom_ybe}
\end{equation}%
where $i,i^{\prime }=1,\ldots ,n$. The
analogous identity holds for $M^{\prime }$. 
\end{corollary}

\begin{proof}
The Yang-Baxter equations for the monodromy matrices are obtained by
repeatedly applying (\ref{YBE}) in the definition (\ref{rowmom}).
\end{proof}

The equation in (\ref{mom_ybe}) describes the exchange of two lattice rows but can be interpreted as definition of 
a subalgebra $\subset \operatorname{End}\mathcal{V}$. Namely, analogous to the case $N=1$, where $M(x_i|t)=L(x_i|t_1)$, 
we decompose the row monodromy matrix $M(x_i|t)\in\operatorname{End}[V(x_i)\otimes\mathcal{V}]$ defined 
in (\ref{rowmom}) over the auxiliary space $V(x_{i})$ as follows,%
\begin{equation}
M(x_i|t)=\sum_{a,b=0,1}e_{ab}\otimes M^{(ab)}(x_i|t),\quad (M^{(ab)}(x_i|t))_{a,b=0,1}=\left( 
\begin{array}{cc}
A(x_{i}|t) & B(x_{i}|t) \\ 
C(x_{i}|t) & D(x_{i}|t)%
\end{array}%
\right)  \label{row_yba}
\end{equation}%
where $e_{ab}$ are the $2\times 2$ unit matrices and the matrix entries $%
A(x_{i}|t)$, $B(x_{i}|t)$, $C(x_{i}|t)$, $D(x_{i}|t)$ are elements in $\mathcal{R}[x_i]\otimes\operatorname{End}\mathcal{V}$; see Figure \ref{fig:mom}. 
Before we define the Yang-Baxter algebra it is instructive to consider a simple example first.
\begin{example}\label{ex:ybalgebra}
Let us consider the case $N=2$ and $n=1$. To unburden the notation somewhat set $x=x_1$. Then $M_1(x|t)=L_{12}(x|t_2)L_{11}(x|t_1)$ and 
the monodromy matrix element $M^{(01)}(x|t)=B(x|t)$ is given by
$$B(x|t)=
\left( 
\begin{smallmatrix}
1 & 0 \\ 
0 & 0%
\end{smallmatrix}%
\right)\otimes
\left( 
\begin{smallmatrix}
0 & 0 \\ 
1+\beta x\ominus t_2 & 0%
\end{smallmatrix}%
\right)+
\left( 
\begin{smallmatrix}
0 & 0 \\ 
1+\beta x\ominus t_1 & 0%
\end{smallmatrix}%
\right)\otimes
\left( 
\begin{smallmatrix}
x\ominus t_2 & 0 \\ 
0 & 1%
\end{smallmatrix}%
\right).
$$
Expanding this expression into factorial powers $(x|\ominus t)^r=\prod_{j=1}^r(x\ominus t_j)$, 
\begin{equation}
B(x|t)=\sum_{r=0}^N B_r(x|\ominus t)^r
\end{equation}
we obtain operators $B_r\in\operatorname{End}\mathcal{V}$ (the lower index labels the power here and should not be confused with the earlier notation where the index indicated on which factor in the tensor product an operator acts),  which for $N=2$ are:
\begin{align*}
B_2 & = \beta
\left( 
\begin{smallmatrix}
0 & 0 \\ 
1 & 0%
\end{smallmatrix}%
\right)\otimes
\left( 
\begin{smallmatrix}
1 & 0 \\ 
0 & 0%
\end{smallmatrix}%
\right)\\
B_1 & = \beta
\left( 
\begin{smallmatrix}
0 & 0 \\ 
1 & 0%
\end{smallmatrix}%
\right)\otimes
\left( 
\begin{smallmatrix}
0 & 0 \\ 
0 & 1%
\end{smallmatrix}%
\right)+
\frac{1+\beta t_1}{1+\beta t_2}\,
\left( 
\begin{smallmatrix}
0 & 0 \\ 
1 & 0%
\end{smallmatrix}%
\right)\otimes
\left( 
\begin{smallmatrix}
1 & 0 \\ 
0 & 0%
\end{smallmatrix}%
\right)+
\beta\,\frac{1+\beta t_1}{1+\beta t_2}\,
\left( 
\begin{smallmatrix}
1 & 0 \\ 
0 & 0%
\end{smallmatrix}%
\right)\otimes
\left( 
\begin{smallmatrix}
0 & 0 \\ 
1 & 0%
\end{smallmatrix}%
\right)\\
B_0 & =
\left( 
\begin{smallmatrix}
1 & 0 \\ 
0 & 0%
\end{smallmatrix}%
\right)\otimes
\left( 
\begin{smallmatrix}
0 & 0 \\ 
1 & 0%
\end{smallmatrix}%
\right)+
\left( 
\begin{smallmatrix}
0 & 0 \\ 
1 & 0%
\end{smallmatrix}%
\right)\otimes
\left( 
\begin{smallmatrix}
0 & 0 \\ 
0 & 1%
\end{smallmatrix}%
\right)+
t_1\ominus t_2\left( 
\begin{smallmatrix}
0 & 0 \\ 
1 & 0%
\end{smallmatrix}%
\right)\otimes
\left( 
\begin{smallmatrix}
1 & 0 \\ 
0 & 0%
\end{smallmatrix}%
\right)+
t_1\ominus t_2\left( 
\begin{smallmatrix}
1 & 0 \\ 
0 & 0%
\end{smallmatrix}%
\right)\otimes
\left( 
\begin{smallmatrix}
0 & 0 \\ 
1 & 0%
\end{smallmatrix}%
\right)
\end{align*}
We claim that all these operators $B_r$ commute. 

Choose another spectral variable, say $y$, and let $u$ be any vector in $\mathcal{V}$, then (\ref{mom_ybe}) yields 
commutation relations for the elements of the monodromy matrix when acting with both sides of (\ref{mom_ybe}) on the same vector. 
For example, acting with the left hand side on $v_1\otimes v_1\otimes u$ we find:
\begin{align*}
& R_{12}(x,y) M_{1}(x|t)M_{2}(y|t)v_1\otimes v_1\otimes u =\\
& R_{12}(x,y) M_{1}(x|t)[v_1\otimes v_0\otimes B(y|t)u+
v_1\otimes v_1\otimes D(y|t)u] =\\
& R_{12}(x,y) \bigl[ v_0\otimes v_0\otimes B(x|t)B(y|t)u +
v_1\otimes v_0\otimes D(x|t)B(y|t)u\bigr.&\\
&\qquad\qquad\qquad\bigl.+\,v_0\otimes v_1\otimes B(x|t)D(y|t)u+v_1\otimes v_1\otimes D(x|t)D(y|t)u
\bigr] =\\
&\qquad\qquad v_0\otimes v_0\otimes B(x|t)B(y|t)u+v_0\otimes v_1\otimes D(x|t)B(y|t)u\\
&\qquad\qquad +\,v_1\otimes v_0\otimes \bigl[(y\ominus x)D(x|t)B(y|t)u+(1+\beta y\ominus x)B(x|t)D(y|t)u\bigr] \\
&\qquad\qquad +\,v_1\otimes v_1\otimes D(x|t)D(y|t)u
\end{align*}
In a similar fashion we find from the right hand side of (\ref{mom_ybe}) that
\begin{multline*}
M_{2}(y|t)M_{1}(x|t)R_{12}(x,y)v_1\otimes v_1\otimes u=\\
v_0\otimes v_0\otimes B(y|t)B(x|t)u+
v_0\otimes v_1\otimes D(y|t)B(x|t)u+\\
v_1\otimes v_0\otimes B(y|t)D(x|t)u+
v_1\otimes v_1\otimes D(y|t)D(x|t)u
\end{multline*}
Comparing terms we arrive at the relations
\begin{align}
& B(x|t)B(y|t) = B(y|t)B(x|t)\\
&D(x|t)B(y|t) = D(y|t)B(x|t)\\
& B(y|t)D(x|t)=(y\ominus x)D(x|t)B(y|t)+(1+\beta y\ominus x)B(x|t)D(y|t)\\
&D(x|t)D(y|t) = D(y|t)D(x|t) 
\end{align}
The first equality now implies that the $B_r$'s commute. Note that the derivation 
of these commutation relation does not depend on $N$, only (\ref{mom_ybe}) 
has been used. By changing vectors to $v_0\otimes v_1\otimes u$, $v_1\otimes v_0\otimes u$ and 
$v_0\otimes v_0\otimes$ we arrive at similar commutation relations for the remaining elements of 
the monodromy matrix.
\end{example}
\begin{definition} We define the \emph{row Yang-Baxter algebra }$\operatorname{YB_N}\subset \operatorname{End}\mathcal{V}$
as the subalgebra generated by the operators $O_r$ obtained when expanding $O(x|t)=A(x|t)$, $B(x|t)$, $C(x|t)$, $D(x|t)$ into factorial powers as explained in the previous example for $B(x|t)$. The row Yang-Baxter
algebra $\operatorname{YB}^\prime_N\subset \operatorname{End}\mathcal{V}$ of the monodromy matrix $M^{\prime }(x|t)$ associated with $L^{\prime }$ is defined analogously.
\end{definition}
As explained in the above example the commutation relations of the elements $O_r$ are given 
in terms of the Yang-Baxter equation (\ref{mom_ybe}) and it will prove convenient to continue to work 
with $O(x|t)$ instead as the commutation relations then take a simpler form.

\begin{figure}[tbp]
\begin{equation*}
\includegraphics[scale=0.9]{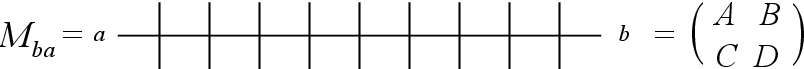}
\end{equation*}%
\caption{Graphical depiction of the monodromy matrix. Its matrix elements
are weighted sums over vertex configurations. To obtain the $A,B,C,D$-operators set respectively $(a,b)=(0,0)$, 
$(1,0)$, $(0,1)$ and $(1,1)$ for the values of the external horizontal edges. }
\label{fig:mom}
\end{figure}

\subsection{Commutation relations of the Yang-Baxter algebras}

We state a number of important commutation relations of the Yang-Baxter
algebra which we will employ in subsequent sections. The first will
be used when deriving the spectrum of the Bethe algebra and the Bethe ansatz
equations.

\begin{lemma}
\label{lem:yba1}Let $y=(y_{1},\ldots ,y_{n})$ be an $n$-tuple of pairwise
distinct variables and $x\neq y_{i}$ for all $i=1,\ldots ,n$. Then we have
the identities%
\begin{multline}
A(x|t)B(y_{1}|t)\cdots B(y_{n}|t)= \\
\frac{B(y_{1}|t)\cdots B(y_{n}|t)A(x|t)}{(x\ominus y_{1}|t)\cdots (x\ominus y_{n})}%
-\sum_{i=1}^{n}\frac{B(y_{1}|t)\cdots \overset{i}{B(x|t)}\cdots B(y_{n}|t)A(y_{i}|t)%
}{(x\ominus y_{i})\prod_{j\neq i}y_{i}\ominus y_{j}}
\end{multline}%
and%
\begin{multline}
D(x|t)B(y_{1}|t)\cdots B(y_{n}|t)= \\
\frac{B(y_{1}|t)\cdots B(y_{n}|t)D(x|t)}{(y_{1}\ominus x)\cdots (y_{n}\ominus x)}%
+\sum_{i=1}^{n}\frac{B(y_{1}|t)\cdots \overset{i}{B(x|t)}\cdots B(y_{n}|t)D(y_{i}|t)%
}{(x\ominus y_{i})\prod_{j\neq i}y_{j}\ominus y_{i}}
\end{multline}%
Analogous identities hold for the Yang-Baxter algebra elements $A^{\prime
},B^{\prime },D^{\prime }$.
\end{lemma}

\begin{proof}
Induction in $n$. The case $n=1$ follows from (\ref{YBE}), namely one
deduces via (\ref{row_yba}) the commutation relations (similar to the computation in Example \ref{ex:ybalgebra})%
\begin{eqnarray*}
(x\ominus y)A(x|t)B(y|t) &=&B(y|t)A(x|t)-B(x|t)A(y|t) \\
(y\ominus x)D(x|t)B(y|t) &=&B(y|t)D(x|t)-(1+\beta\, y\ominus x)~B(x|t)D(y|t)\;.
\end{eqnarray*}%
For the induction step use is made of the fact that $B(x|t)B(y|t)=B(y|t)B(x|t)$,
which again is a consequence of (\ref{mom_ybe}) and implies that the result is
symmetric in the $y_{i}$'s.
\end{proof}

The next lemma will be used to compute the bilinear form of our generalised
cohomology ring.

\begin{lemma}
\label{lem:yba2}Let $(x_{1},\ldots ,x_{n})$ and $(y_{1},\ldots ,y_{n})$ be
some mutually pairwise distinct sets of variables. Then%
\begin{multline}
C(x_{1}|t)\cdots C(x_{n}|t)B(y_{n}|t)\cdots B(y_{1}|t)= \\
\frac{1}{\Pi (x)}\sum_{w}w\left( \frac{\Pi (x)D(y_{1}|t)\cdots
D(y_{n}|t)A(x_{1}|t)\cdots A(x_{n}|t)}{\prod_{1\leq i,j\leq n}x_{i}\ominus y_{j}}%
\right)
\end{multline}%
where the sum runs over the minimal length coset representatives $w$ of $%
\mathbb{S}_{2n}/\mathbb{S}_{n}\times \mathbb{S}_{n}$ which act in the
obvious manner on the alphabet $\{x_{1},\ldots ,x_{n},y_{1},\ldots ,y_{n}\}$.
\end{lemma}

\begin{proof}
By induction in $n$. The case $n=1$ follows from the commutation relation%
\begin{equation*}
C(x|t)B(y|t)=\frac{D(y|t)A(x|t)-D(x|t)A(y|t)}{x\ominus y}=\frac{A(x|t)D(y|t)-A(y|t)D(x|t)}{%
x\ominus y}
\end{equation*}%
which is a direct consequence of (\ref{YBE}). For the induction step one
uses the commutation relations%
\begin{equation*}
O(x|t)O(y|t)=O(y|t)O(x|t),\qquad O=A,B,C,D
\end{equation*}%
and%
\begin{eqnarray*}
C(x|t)D(y|t) &=&\frac{D(y|t)C(x|t)-D(x|t)C(y|t)}{x\ominus y} \\
A(x|t)B(y|t) &=&\frac{B(y|t)A(x|t)-B(x|t)A(y|t)}{x\ominus y}
\end{eqnarray*}%
all of which follow once more from (\ref{mom_ybe}) along the same lines as 
demonstrated in Example \ref{ex:ybalgebra}. Note that these commutation
relations again imply that the result must be symmetric in the $x_{i}$'s and
symmetric in the $y_{i}$'s. This greatly simplifies the computation.
\end{proof}


Analogous commutation relations hold for the monodromy matrix $M^{\prime}$
and the generators $A^{\prime},B^{\prime}, C^{\prime}, D^{\prime}$. These
can be derived easily from the following result which relates both
Yang-Baxter algebras in a simple manner.

\begin{lemma}
\label{lem:yba3} Let $\Theta :\mathcal{V}\rightarrow \mathcal{V}$ be the
linear extension of the involution $v_{\lambda }\mapsto v_{\lambda ^{\prime
}}$. Set $\ominus t^{\prime }=(\ominus t_{N},\ldots ,\ominus t_{2},\ominus
t_{1})$, then we have the identity%
\begin{equation}
\Theta M^{(ab)}(x_i|t)=(M^{\prime })^{(ba)}(x_i|\ominus t^{\prime })\Theta\,.
\label{levelrankmom}
\end{equation}%
So, $\Theta A(x_i|t)=A\rq{}(x_i|\ominus t\rq{})\Theta$, $\Theta B(x_i|t)=
C\rq{}(x_i|\ominus t\rq{})\Theta$, etc. 
\end{lemma}

\begin{proof}
Recall the definition of the matrix elements of the $L$-operator (\ref{L})
via the expansion $L(x_i|t)v_{\varepsilon _{1}}\otimes v_{\varepsilon
_{2}}=\sum_{\varepsilon _{1}^{\prime },\varepsilon _{2}^{\prime
}=0,1}L(x_i|t)_{\varepsilon _{1}\varepsilon _{2}}^{\varepsilon _{1}^{\prime
}\varepsilon _{2}^{\prime }}v_{\varepsilon _{1}^{\prime }}\otimes
v_{\varepsilon _{2}^{\prime }}$ and similarly define $L^{\prime
}(x_i|t)_{\varepsilon _{1}\varepsilon _{2}}^{\varepsilon _{1}^{\prime }\varepsilon
_{2}^{\prime }}$. The matrix elements are the weights of the vertex
configurations given in Figure \ref{fig:5vmodels} where $\varepsilon
_{1},\varepsilon _{2},\varepsilon _{1}^{\prime },\varepsilon _{2}^{\prime }$
are the values of the W, N, E and S edge of the vertex as explained earlier.
Interchanging 0 with 1-letters attached to the vertical lines going through
the vertex configurations displayed in Figure \ref{fig:5vmodels} we find 
\begin{equation}
L(x_{i}|t_{j})_{\varepsilon _{1}\varepsilon _{2}}^{\varepsilon _{1}^{\prime
}\varepsilon _{2}^{\prime }}=L^{\prime }(x_{i}|\ominus t_{j})_{\varepsilon
_{1}^{\prime }(1-\varepsilon _{2})}^{\varepsilon _{1}(1-\varepsilon
_{2}^{\prime })}  \label{LRduality_L}
\end{equation}%
for all $\varepsilon _{i},\varepsilon _{i}^{\prime }=0,1$ with $i=1,2$. 
The assertion for the row monodromy matrix is now an immediate consequence
of the definition (\ref{rowmom}) and the identity (\ref{LRduality_L}).
\end{proof}

\subsection{Transposed Yang-Baxter algebras}

We will also need to consider the action of the Yang-Baxter algebra in the
dual quantum space. The transposed monodromy matrices can be explicitly
computed.

Define another pair of $L$-operators 
\begin{equation}
L^{\vee }(x_i|t_j)=\left( 
\begin{array}{cc}
\sigma^{+}\sigma^{-}+x_{i}\ominus t_{j}~\sigma^{-}\sigma^{+} & \sigma^{+} \\ 
(1+\beta x_{i}\ominus t_{j})\sigma^{-} & \sigma^{-}\sigma^{+}%
\end{array}%
\right)
\end{equation}%
and%
\begin{equation}
L^{\ast }(x_i|t_j)=\left( 
\begin{array}{cc}
\sigma^{-}\sigma^{+}+x_{i}\oplus t_{j}~\sigma^{+}\sigma^{-} & (1+\beta
x_{i}\oplus t_{j})\sigma^{+} \\ 
\sigma^{-} & \sigma^{+}\sigma^{-}%
\end{array}%
\right)\;.
\end{equation}%
Employing the latter we define \emph{dual row monodromy matrices} $%
M^\vee_i\in\operatorname{End}[W_n\otimes\mathcal{V}]$ in the same 
manner as before,%
\begin{equation}
M_{i}^{\vee }=L_{i1}^{\vee }L_{i2}^{\vee }\cdots L_{iN}^{\vee },
\label{dual_mom}
\end{equation}%
with $L^\vee_{ij}\in\operatorname{End}[W_n\otimes\mathcal{V}]$ being the operator which acts non-trivially only in the $i$th factor of $W_n$ and the $j$th factor of $\mathcal{V}$, where its action is that of $L^\vee(x_i|t_j)\in\operatorname{End}[V(x_i)\otimes V(t_j)]$. Similarly, we define $M_i^{\ast }$ where we use the $L^{\ast }$-operators instead. The dual monodromy matrices $M^\vee_i,M^\ast_i$ also obey a Yang-Baxter equation of the form (\ref{mom_ybe}), where the $R$-matrix elements are given by similar expressions as in (\ref%
{ybesolns}). As we shall not need their explicit form we omit them here.

\begin{lemma}
Recall the definitions of the row monodromy matrices $M_{i},M_{i}^{\prime
},M_{i}^{\vee },M_{i}^{\ast }$ as maps $W_r\otimes \mathcal{V}%
\rightarrow W_r\otimes \mathcal{V}$ for some $i$ and $r=n,k$. Then%
\begin{equation}
M_{i}^{1\otimes T}=(M_{i}^{\vee })^{T\otimes 1}\qquad \text{and}\qquad
(M_{i}^{\prime })^{1\otimes T}=(M_{i}^{\ast })^{T\otimes 1},
\label{trans_mom}
\end{equation}
where the upper indices $1\otimes T$ and $T\otimes 1$ indicate the transpose
in respectively the quantum space $\mathcal{V}$ and the auxiliary space $%
W_r$ with respect to the spin basis $\{v_\lambda\}$.
\end{lemma}

\begin{proof}
Recall the definition of $L(x_i|t_j):V(x_{i})\otimes V(t_{j})\rightarrow
V(x_{i})\otimes V(t_{j})$ and take the transpose in the second factor to
find that%
\begin{equation*}
L(x_i|t_j)^{1\otimes T}=\left( 
\begin{array}{cc}
\sigma _{j}^{-}\sigma _{j}^{+}+x_{i}\oplus t_{j}~\sigma _{j}^{+}\sigma
_{j}^{-} & (1+\beta x_{i}\oplus t_{j})\sigma _{j}^{-} \\ 
\sigma _{j}^{+} & \sigma _{j}^{+}\sigma _{j}^{-}%
\end{array}%
\right) =L^{\vee }(x_i|t_j)^{T\otimes 1}\;.
\end{equation*}%
Thus, we can deduce for the monodromy matrix $M_{i}:W_n\otimes \mathcal{%
V}\rightarrow W_n\otimes \mathcal{V}$%
\begin{eqnarray*}
M_{i}^{1\otimes T} &=&L_{iN}^{1\otimes T}\cdots L_{i2}^{1\otimes
T}L_{i1}^{1\otimes T} \\
&=&(L_{iN}^{\vee })^{T\otimes 1}\cdots (L_{i2}^{\vee })^{T\otimes
1}(L_{i1}^{\vee })^{T\otimes 1}=(M_{i}^{\vee })^{T\otimes 1}
\end{eqnarray*}%
The proof of the other identity is completely analogous.
\end{proof}

\subsection{Quantum deformation}

We discuss a slight generalisation of the previous results which will allow
us to introduce additional (invertible) \textquotedblleft quantum
parameters\textquotedblright\ $q_{1},\ldots ,q_{N}$ in the monodromy
matrices by considering the extension $\mathbb{Z}[\![q_{1},q_{1}^{-1},\ldots
,q_{n} ,q_{n}^{-1}]\!]\otimes \mathcal{V}$ as quantum space.

\begin{lemma}
We have the following $q$-deformed version of the Yang-Baxter equation (\ref{ybe}) in 
$\operatorname{End}[\mathbb{Z}[\![q_1,q_1^{-1}]\!]\otimes V(x_i)\otimes V(t_j)\otimes V(t_{j\rq{}})]$,%
\begin{equation}
r_{23}(q_1)L_{12}(x_{i};t_{j})\left( 
\begin{smallmatrix}
1 & 0 \\ 
0 & q_1%
\end{smallmatrix}%
\right) _{1}L_{13}(x_{i};t_{j^{\prime }})=L_{13}(x_{i};t_{j^{\prime
}})\left( 
\begin{smallmatrix}
1 & 0 \\ 
0 & q_1%
\end{smallmatrix}%
\right) _{1}L_{12}(x_{i};t_{j})r_{23}(q_1),  \label{qybe}
\end{equation}%
where $\left( 
\begin{smallmatrix}
1 & 0 \\ 
0 & q_1%
\end{smallmatrix}%
\right) _{1}$ acts in the first factor of $(\mathbb{Z}[\![q_1,q_{1}^{-1}]\!]\otimes 
V(x_i))\otimes V(t_j)\otimes V(t_{j\rq{}})$ and %
\begin{equation}
r(q_1)=\left( 
\begin{array}{cccc}
1 & 0 & 0 & 0 \\ 
0 & 0 & \frac{1+\beta t_{j}}{1+\beta t_{j^{\prime }}} & 0 \\ 
0 & 1 & q_1^{-1}\frac{t_{j}-t_{j^{\prime }}}{1+\beta t_{j^{\prime }}} & 0 \\ 
0 & 0 & 0 & 1%
\end{array}%
\right)\in\operatorname{End}[\mathbb{Z}[\![q_1,q_{1}^{-1}]\!]\otimes V(t_j)\otimes V(t_{j\rq{}})] \;  \label{qr}
\end{equation}
\end{lemma}

Using this result we can generalise our previous formulae for the monodromy
matrices by setting%
\begin{equation}
M_{i}(q_{1},\ldots ,q_{N}):=L_{iN}\left( 
\begin{smallmatrix}
1 & 0 \\ 
0 & q_{N}%
\end{smallmatrix}%
\right) _{i}\cdots L_{i2}\left( 
\begin{smallmatrix}
1 & 0 \\ 
0 & q_{2}%
\end{smallmatrix}%
\right) _{i}L_{i1}\left( 
\begin{smallmatrix}
1 & 0 \\ 
0 & q_{1}%
\end{smallmatrix}%
\right) _{i}\, ,  \label{qMom}
\end{equation}%
where the index notation is to be interpreted in the same manner as explained previously, the matrices 
$\left(\begin{smallmatrix}
1 & 0 \\ 
0 & q_{j}%
\end{smallmatrix}\right)_i$ act trivially except for the $i$th factor in the tensor product $W$. 
Employing the same type of arguments as in our previous discussion, one
shows that these deformed monodromy matrices satisfy the same type of
Yang-Baxter relations (\ref{qybe}) as the non-deformed ones, the only
difference lies in the braid matrix $r$ which is now replaced by $r(q)$. For
discussing the $q\,$-deformation of the cohomology and $K$-theory of the
Grassmannian we need to choose $q_{1}=q$ and $q_{2}=\cdots =q_{N}=1$. We
shall henceforth denote $\mathbb{Z}[\![q,q^{-1}]\!]\otimes \mathcal{V}$ by $%
\mathcal{V}^{q}$ and, similarly, $\mathbb{Z}[\![q,q^{-1}]\!]\otimes \mathcal{V}_{n}$
by $\mathcal{V}_{n}^{q}$, where $\mathcal{V}_n$ is the subspace defined in 
Section \ref{sec:quantumspace}.

\subsection{Row-to-row transfer matrices}

We now introduce periodic boundary conditions in the horizontal direction of
the lattice by taking the partial trace of the operator (\ref{rowmom}) over
the auxiliary space $V^{\otimes n}$. We obtain the following operator $Z_{n}:%
\mathcal{R}[x_{1},\ldots ,x_{n}]\otimes \mathcal{V}^{q}\rightarrow \mathcal{R%
}[x_{1},\ldots ,x_{n}]\otimes \mathcal{V}^{q}$,%
\begin{equation}
Z_{n}(x_1,\ldots,x_n|t_1,\ldots,t_N)=\limfunc{Tr}_{V^{\otimes n}}\left( 
\begin{smallmatrix}
1 & 0 \\ 
0 & q%
\end{smallmatrix}%
\right) _{n}M_{n}\cdots \left( 
\begin{smallmatrix}
1 & 0 \\ 
0 & q%
\end{smallmatrix}%
\right) _{2}M_{2}\left( 
\begin{smallmatrix}
1 & 0 \\ 
0 & q%
\end{smallmatrix}%
\right) _{1}M_{1}\;.  \label{Zdef}
\end{equation}%
We also define an operator $Z_{k}^{\prime }$ using
instead the $L^{\prime }$-operators and replacing $n\rightarrow k$
everywhere. The matrix elements of these operators with respect to the spin basis 
(\ref{spin basis}) give the partition functions (\ref{partition_function}) for lattice 
configurations with periodic boundary conditions imposed on the external horizontal lattice edges, 
that is partition functions for non-intersecting lattice paths on the cylinder.

Denote by 
\begin{equation}\label{defH}
H(x_1|t)=Z_{1}(x_1|t)=A(x_1|t)+qD(x_1|t)
\end{equation} 
and 
\begin{equation}\label{defE}
E(x_1|t)=Z_{1}^{\prime }(x_1|t)=A^{\prime }(x_1|t)+qD^{\prime }(x_1|t)\,
\end{equation}
the operators whose matrix elements give the partition functions of a single lattice row. The following lemma states that the partition functions on the cylinder with $n$ (in the case of $H$) or $k$ rows (in the case of $E$) can be obtained by taking matrix elements of  the following operator products:
\begin{lemma} We have the relations 
\begin{equation}
Z_{n}(x|t)=H(x_{n}|t)\cdots H(x_{2}|t)H(x_{1}|t)  \label{Z2H}
\end{equation}%
and%
\begin{equation}
Z_{k}^{\prime }(x|t)=E(x_{k}|t)\cdots E(x_{2}|t)E(x_{1}|t)\;.  \label{Z'2E}
\end{equation}%
The operators $H(x_i|t),E(x_i|t)$ are called the \emph{row-to-row transfer matrices}.
\end{lemma}

\begin{proof}
This is immediate from the definitions (\ref{rowmom}), (\ref{Zdef}) and the
fact that the $L$-operators $L_{ij},L_{i^{\prime }j^{\prime }}
\in\operatorname{End}[W_n\otimes\mathcal{V}]$ commute if $%
i\neq i^{\prime }$ and $j\neq j^{\prime }$.
\end{proof}

\begin{corollary}
\label{cor:levelrankEH}We have the following identity for the row-to-row
transfer matrices, $\Theta H(x_i|t)\Theta =E(x_i|\ominus t^{\prime })$.
\end{corollary}

\begin{proof}
Employ (\ref{levelrankmom}) and the defining relations (\ref{defH}), (\ref{defE}).
\end{proof}

The following statement shows that the transfer matrix generate a
commutative subalgebra - the so-called Bethe algebra - within the
Yang-Baxter algebra which we will identify with our generalised cohomology
ring. Because of the existence of this commutative subalgebra, which should
be thought of as the analogue of integrals of motion of a classical
integrable system described in terms of differential equations, the models
are called (quantum) \emph{integrable}.

\begin{proposition}[Integrability]
\label{prop:integrability} \emph{All} the row-to-row transfer matrices
commute, that is we have that 
\begin{equation}
H(x_{i}|t)H(x_{i^{\prime }}|t)=H(x_{i^{\prime }}|t)H(x_{i}|t),\qquad
E(x_{i}|t)E(x_{i^{\prime }}|t)=E(x_{i^{\prime }}|t)E(x_{i}|t)  \label{integrability1}
\end{equation}%
as well as%
\begin{equation}
H(x_{i}|t)E(x_{i^{\prime }}|t)=E(x_{i^{\prime }}|t)H(x_{i}|t)\;.
\label{integrability2}
\end{equation}%
In particular, the operators $Z_{n}(x|t)$, $Z_{k}^{\prime }(x|t)$ are symmetric in the 
$x$-variables.
\end{proposition}

\begin{proof}
The last assertion is a direct consequence of the Yang-Baxter equation (\ref%
{mom_ybe}):%
\begin{eqnarray*}
Z_{n}(x_{1},\ldots ,x_{n}|t) &=&\operatorname{Tr}_{V^{\otimes
n}}(R_{i,i+1}M_{n}\cdots M_{1}R_{i,i+1}^{-1}) \\
&=&\operatorname{Tr}_{V^{\otimes n}}(M_{n}\cdots R_{i,i+1}M_{i}M_{i+1}\cdots
M_{1}R_{i,i+1}^{-1}) \\
&=&\operatorname{Tr}_{V^{\otimes n}}(M_{n}\cdots M_{i+1}M_{i}R_{i,i+1}\cdots
M_{1}R_{i,i+1}^{-1}) \\
&=&\operatorname{Tr}_{V^{\otimes n}}(M_{n}\cdots M_{i+1}M_{i}\cdots M_{1}) \\
&=&Z_{n}(x_{1},\ldots ,x_{i+1},x_{i},\ldots ,x_{n}|t)
\end{eqnarray*}%
The proof for $Z_{k}^{\prime }$ follows along the same lines. Setting $n=k=2$
we obtain (\ref{integrability1}).

To prove (\ref{integrability2}) one establishes the existence of additional
solutions of the Yang-Baxter equation, 
\begin{equation}
R_{12}^{\prime \prime }(x_{i}, x_{i^{\prime
}})M_{1}(x_{i}|t)M_{2}^{\prime }(x_{i^{\prime }}|t)=M_{2}^{\prime
}(x_{i^{\prime }}|t)M_{1}(x_{i}|t)R_{12}^{\prime \prime }(x_{i},
x_{i^{\prime }})  \label{ybe2}
\end{equation}%
where $R^{\prime \prime }(x_{i}, x_{i^{\prime}})$ is again of the form (\ref{R}) with 
\begin{equation}
\begin{tabular}{|l||c|c|c|c|c|c|}
\hline
& $a$ & $b$ & $c$ & $c^{\prime }$ & $b^{\prime }$ & $a^{\prime }$ \\ 
\hline\hline
\multicolumn{1}{|c||}{$R^{\prime \prime }(x_{i}, x_{i^{\prime}})$} & $x_{i}\oplus x_{i^{\prime }}$
& $1$ & $1+\beta ~x_{i}\oplus x_{i^{\prime }}$ & $1$ & $1$ & $0$ \\ \hline
\end{tabular}%
\end{equation}%
Note that $R^{\prime \prime }$ is singular. However, from the Yang-Baxter
equations (\ref{ybe2}) one derives the commutation relations%
\begin{eqnarray*}
A(x_i|t)A^{\prime }(x_{i\rq{}}|t) &=&A^{\prime }(x_{i\rq{}}|t)A(x_i|t) \\
A(x_i|t)D^{\prime }(x_{i\rq{}}|t)-A^{\prime }(x_{i\rq{}}|t)D(x_i|t) &=&D^{\prime
}(x_{i\rq{}}|t)A(x_i|t)-D(x_i|t)A^{\prime }(x_{i\rq{}}|t)
\end{eqnarray*}%
for the row Yang-Baxter algebras. Employing the graphical calculus in terms
of the vertex configurations in Figure \ref{fig:5vmodels} one obtains the
additional relations 
\begin{equation*}
D(x_i|t)D^{\prime }(x_{i\rq{}}|t)=D^{\prime }(x_{i\rq{}}|t)D(x_i|t)=0\;.
\end{equation*}%
From these equalities we then easily deduce that $H(x_i|t)E(x_{i\rq{}}|t)=E(x_{i\rq{}}|t)H(x_i|t)$.
\end{proof}

\subsection{Combinatorial description of the transfer matrices}

We now describe the action of the row-to-row transfer matrices in the 
spin basis \eqref{spin basis}, $\{v_{\lambda}\}_{\lambda\subset (k^n)}\subset \mathcal{V}_{n}$ for $n\leq N/2$ 
using toric horizontal and vertical strips; see the earlier section on preliminaries. 
For $n>N/2$ the action can then be deduced by employing Cor \ref{cor:levelrankEH}.

We interpret partitions and their associated cylindric loops as subsets of $%
\mathbb{Z}^{2}$. Given a toric horizontal strip $\theta =\lambda /d/\mu $ of
degree $d$ denote by

\begin{itemize}
\item $\mathcal{R}_{\theta }$ the set which contains all squares $s=\langle
i,j\rangle \in \mathbb{Z}^{2}$, $1\leq i\leq n$ such that the square
immediately left to it, $s^{\prime }=\langle i,j-1\rangle $, is the
rightmost square in a row of $\lambda \lbrack d]$ intersecting $\theta $;

\item $\mathcal{\bar{C}}_{\theta }$ the set which contains all the bottom
squares $s=\langle i,j\rangle ,~1\leq j\leq k$ from each column of $\mu
\lbrack 0]$ which does not intersect $\theta $ as well as the squares $%
s=\langle 1,j\rangle $ in empty columns if $\lambda _{1}+n<j\leq N$ and $\mu
\subset \lambda $.
\end{itemize}

Likewise, given a toric vertical strip $\theta =\lambda /d/\mu $ denote by

\begin{itemize}
\item $\mathcal{\bar{R}}_{\theta }$ the set which contains the square $%
s=\langle i,j\rangle $ next to the rightmost square $s^{\prime }=\langle
i,j-1\rangle $ in each row of $\mu $ not intersecting $\theta $. This
includes squares $s=\langle i,1\rangle $ in empty rows for which $1\leq i<n$;

\item $\mathcal{C}_{\theta }$ the set which contains the bottom squares from
each column of $\lambda \lbrack d]$ which intersects $\theta $.
\end{itemize}

\begin{proposition}
\label{prop:comb_transfer}We have the following combinatorial action of the
transfer matrices on $\mathcal{V}_{n}^{q}$ in the spin basis $\{v_{\lambda }\}_{\lambda\subset (k^n)}$,%
\begin{eqnarray*}
H(x|t)v_{\mu } &=&\sum_{\substack{ \theta =\lambda /d/\mu \text{ }  \\ \text{%
hor strip}}}q^{d}\left( \prod_{s\in \mathcal{\bar{C}}_{\theta }}x\ominus
t_{n+c(s)}\right) \left( \prod_{s\in \mathcal{R}_{\theta }}(1+\beta x\ominus
t_{(n+c(s))\!\!\!\!\!\!\mod\!\! N}\right) v_{\lambda } \\
E(x|t)v_{\mu } &=&\sum_{\substack{ \theta =\lambda /d/\mu \text{ }  \\ \text{%
ver strip}}}q^{d}\left( \prod_{s\in \mathcal{\bar{R}}_{\theta }}x\oplus
t_{n+c(s)}\right) \left( \prod_{s\in \mathcal{C}_{\theta }}(1+\beta x\oplus
t_{n+c(s)}\right) v_{\lambda }
\end{eqnarray*}%
where the degree $d$ of the toric strips is either zero or one and $c(s)=j-i$
is the content of the square $s=\langle i,j\rangle $ in the Young diagram of 
$\lambda $ or $\mu $.
\end{proposition}

\begin{proof}
The proof of these formulae follows along similar lines as in \cite{VicOsc}
and we therefore only sketch the main argument. Consider a fixed matrix
element $\langle \tilde{v}_{\lambda}|H(x|t)v_{\mu }\rangle $ which is 
the partition function for a
single lattice row where the values of the upper and lower vertical edges
have been fixed in terms of the binary strings $b(\mu )$ and $b(\lambda )$,
respectively. We will discuss a simple example below; see Figure \ref{fig:gr13ex}. 
Using the bijection between binary strings $b$ and boxed
partitions $\lambda (b)$ from Section \ref{sec:partitions} one can translate
the various vertex configurations in Figure \ref{fig:5vmodels}, which
represent matrix elements of the $L$ and $L^{\prime }$-operators, into the
operation of adding boxes to the Young diagram of $\mu $. For example, the
first and second vertex configuration in the top row of Figure \ref%
{fig:5vmodels} leave the Young diagram of $\mu $ unchanged, the fourth and
fifth vertex configurations signal respectively the end and start of a
horizontal strip being added to $\mu $, while the third vertex in the top
row corresponds to two boxes being added in the same row. Similarly, the
first two vertex configurations in the bottom row of \ref{fig:5vmodels} do
not add any boxes to $\mu $, the fourth and fifth signal the start and end
of a vertical strip, while the third vertex in the bottom row indicates that
two boxes are added in the same column. Using these one-to-one maps between
horizontal (vertical) strips and lattice configurations, the above formulae follow 
from the weights fixed via the definitions (\ref{L}) and (\ref{L'}). 
\end{proof}

\begin{figure}[tbp]
\begin{equation*}
\includegraphics[scale=0.35]{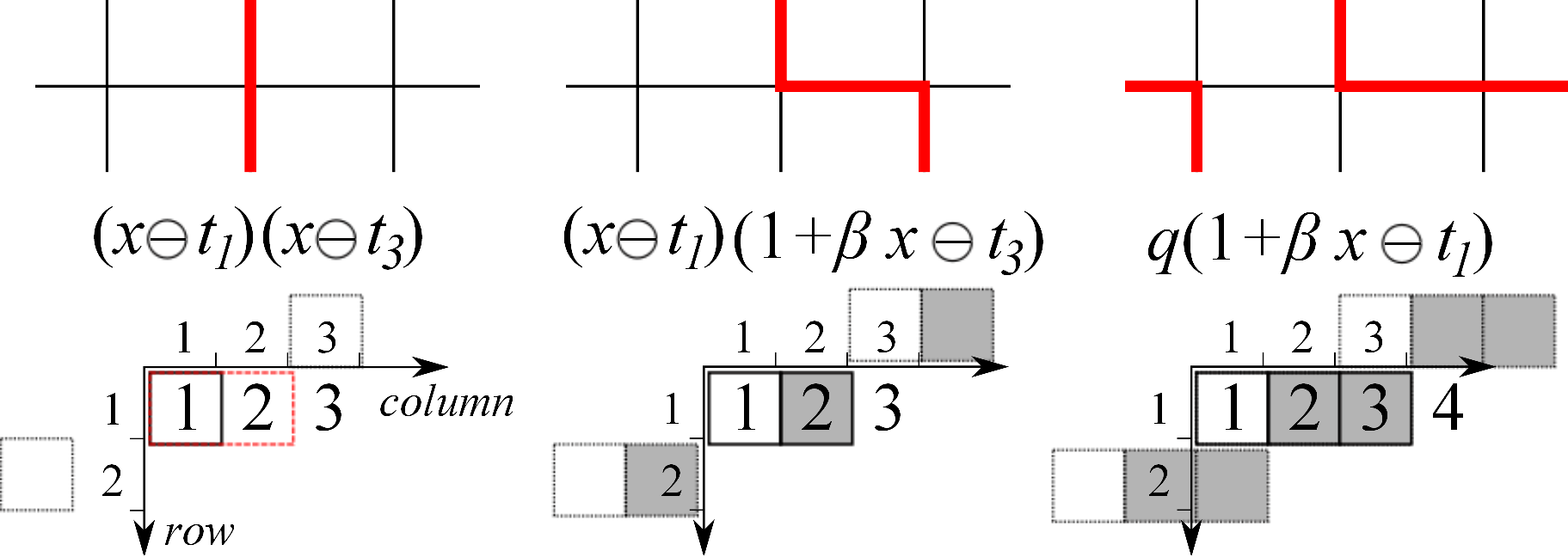}
\end{equation*}%
\caption{Lattice configurations for the projective space $\mathbb{P}^2$ and
their weights; see Figure \protect\ref{fig:5vmodels}. The values of the top
edges of the vertices are fixed by the binary string 010. Below the weights
are the corresponding toric skew diagrams; see Proposition \protect\ref%
{prop:comb_transfer} and Example \protect\ref{ex:1}. The dotted boxes are
the cylindric continuation (\protect\ref{cyl_loop}) of the solid Young
diagrams; see Section \protect\ref{sec:torictab}}
\label{fig:gr13ex}
\end{figure}

\begin{example}
\label{ex:1} Consider the simplest non-trivial case $\operatorname{Gr}_{1,3}=%
\mathbb{P}^{2}$, i.e. we set $N=3$ and $n=1$. In terms of binary strings $%
\mathcal{V}_{1}$ is spanned by $\{v_{100},v_{010},v_{001}\}$. We consider
the matrix elements of $H(x|t)$ in this basis, which can be visualised as a
sum over all the possible vertex configurations shown in Figure \ref%
{fig:5vmodels} occurring in a single lattice row of length $N=3$. Drawing
all these allowed lattice configurations with fixed binary strings $010$ and 
$001 $ on the top edges, we arrive at Figures \ref{fig:gr13ex} and \ref%
{fig:gr13exb} with the product of the respective vertex weights shown below.
We now convert the binary strings into partitions with bounding box $1\times
2$ to obtain toric horizontal strips; see Section \ref{sec:torictab}.

Starting from the left in Figure \ref{fig:gr13ex} the first lattice
configuration is the matrix element $\langle \tilde{v}_{010}|H(x|t)v_{010}%
\rangle $. The binary string $010$ is the partition with one square at
position $\langle 1,1\rangle $ and we have $\lambda =\mu =(1)$, that is an
empty horizontal strip where no box is added and $d=0$. Thus, $\mathcal{R}%
_{\lambda /\mu }=\emptyset $ and $\mathcal{\bar{C}}_{\lambda /\mu
}=\{\langle 1,1\rangle ,\langle 1,3\rangle \}$ where the last square in $%
\mathcal{\bar{C}}_{\lambda /\mu }$ belongs to an empty column with the
column number $j$ obeying the stated condition $1<j=3\leq N$. According to
Prop \ref{prop:comb_transfer} we arrive at the weight 
\begin{equation*}
\langle \tilde{v}_{010}|H(x|t)v_{010}\rangle =(x\ominus t_{1})(x\ominus
t_{3})\,.
\end{equation*}

The next lattice configuration is the matrix element $\langle \tilde{v}%
_{001}|H(x|t)v_{010}\rangle $ with $\lambda =(2),\mu =(1)$. Thus, we have the
horizontal strip $\theta =\lambda /\mu $ with one square at $\langle
1,2\rangle $ and $d=0$. The sets appearing in the formula of Prop \ref%
{prop:comb_transfer} are $\mathcal{R}_{\lambda /\mu }=\{\langle 1,3\rangle
\} $ and $\mathcal{\bar{C}}_{\lambda /\mu }=\{\langle 1,1\rangle \}$, since
the square $\langle 1,3\rangle $ is adjacent to the square $\langle
1,2\rangle $ which appears in a row intersecting $\lambda /\mu $ while the
square $\langle 1,1\rangle $ is the bottom square in a column not
intersecting $\lambda /\mu $. Hence, 
\begin{equation*}
\langle \tilde{v}_{001}|H(x|t)v_{010}\rangle =(x\ominus t_{1})(1+\beta
x\ominus t_{3})\,.
\end{equation*}

The last lattice configuration in the top row is the matrix element $\langle 
\tilde{v}_{100}|H(x|t)v_{010}\rangle $ with $\lambda =(0),\mu =(1)$. Now, we
have a toric strip with $d=1$, that is $\lambda /1/\mu =\{\langle 1,2\rangle
,\langle 1,3\rangle \}$. The first column with the square at $\langle
1,1\rangle $ now intersects $\lambda /1/\mu $, because the square at $%
\langle 2,1\rangle $ is in the cylindric loop $\lambda \lbrack 1]$.
Therefore, $\mathcal{R}_{\lambda /1/\mu }=\{\langle 1,4\rangle \}$, $%
\mathcal{\bar{C}}_{\lambda /1/\mu }=\emptyset $ and 
\begin{equation*}
\langle \tilde{v}_{100}|H(x|t)v_{010}\rangle =q(1+\beta x\ominus t_{1})\,.
\end{equation*}%
In summary, we have the action (compare with Prop \ref{prop:comb_transfer}), 
\begin{equation*}
H(x|t)v_{010}=(x\ominus t_{1})(x\ominus t_{3})v_{010}+(x\ominus t_{1})(1+\beta
x\ominus t_{3})v_{001}+q(1+\beta x\ominus t_{1})v_{100}\;.
\end{equation*}%
We leave the verification of the weights in Figure \ref{fig:gr13exb} to the
reader. 
\begin{figure}[tbp]
\begin{equation*}
\includegraphics[scale=0.35]{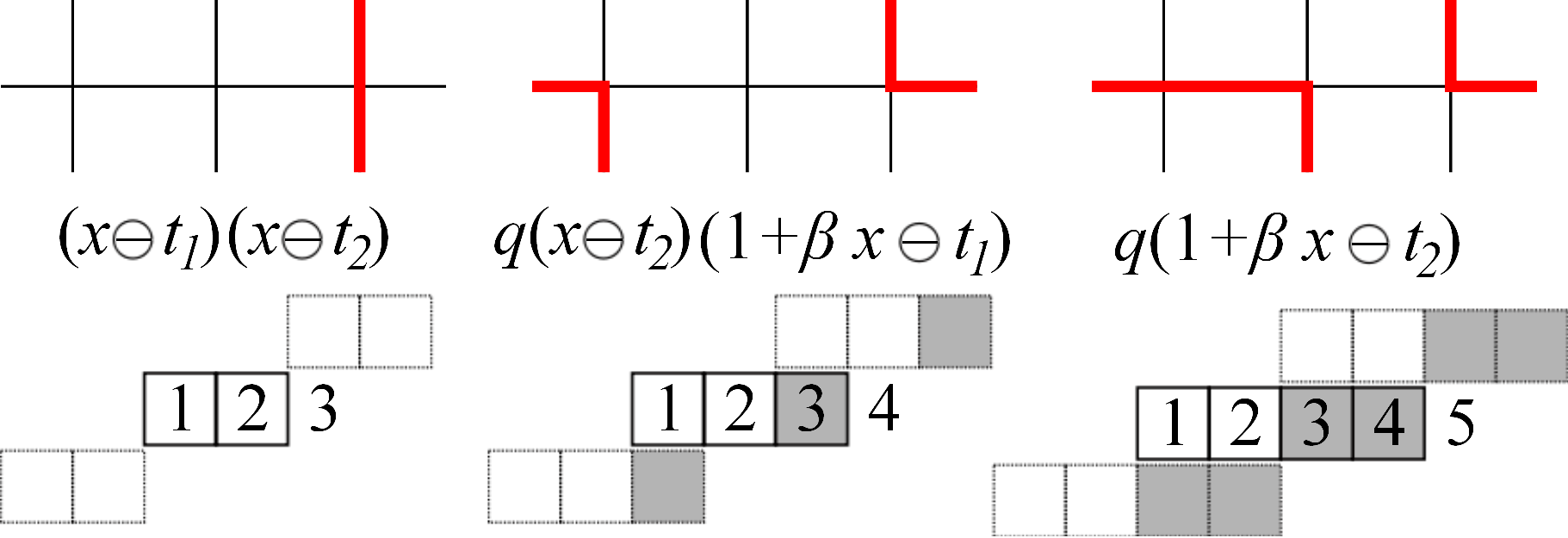}
\end{equation*}%
\caption{Lattice configurations for $\mathbb{P}^{2}$ and binary string $001$%
, their weights and corresponding toric strips; see Figure \protect\ref%
{fig:5vmodels}.}
\label{fig:gr13exb}
\end{figure}
\end{example}

Let $\ominus t^{\prime }=(\ominus t_{N},\ldots ,\ominus t_{2},\ominus t_{1})$ and on
each $\mathcal{V}_{n}^{q}$ define operators $\{H_{r}\}_{r=1}^{k}$ and $%
\{E_{r}\}_{r=1}^{n}$ through the expansions%
\begin{eqnarray}
H(x|t)|_{\mathcal{V}_{n}^{q}} &=&(x|\ominus t^{\prime })^{k}\cdot \mathbf{1}_{\mathcal{%
V}_{n}^{q}}+(1+\beta x)\sum_{r=1}^{k}H_{r}~\frac{(x|\ominus t^{\prime })^{k-r}}{%
1+\beta t_{n+r}},  \label{Hr} \\
E(x|t)|_{\mathcal{V}_{n}^{q}} &=&(x|t)^{n}\cdot \mathbf{1}_{\mathcal{V}%
_{n}^{q}}+(1+\beta x)\sum_{r=1}^{n}E_{r}~(1+\beta t_{n+1-r})(x|t)^{n-r}\;,
\label{Er}
\end{eqnarray}%
where $(x|t)^{r}=\prod_{j=1}^{r}(x\oplus t_{j})$ are the factorial powers (%
\ref{facpower}) with respect to the group law (\ref{group_law}). Below we
will relate the operator coefficients in these expansions to the Pieri rules
in $qh^\ast_n$. Setting $\beta=0$ they correspond to the generators in
Mihalcea's coordinate ring representation of equivariant quantum cohomology 
\cite[Thm 1.1]{Mihalcea}.

\begin{corollary}
The operators $\{E_{r}\}_{r=1}^{n}\cup \{H_{r}\}_{r=1}^{k}$ generate a
commutative subalgebra $\subset \operatorname{End}\mathcal{V}_{n}^{q}$ and we
have the formulae ($t_{j}^{\prime }=t_{N+1-j}$)%
\begin{eqnarray}
H_{k+1-i} &=&\sum_{j=1}^{i}\frac{H(t_{j}^{\prime }|t)}{\prod\nolimits_{1\leq
\ell \neq j\leq i}t_{j}^{\prime }\ominus t_{\ell }^{\prime }},\qquad
i=1,\ldots ,k  \label{Hr2} \\
E_{n+1-i} &=&\sum_{j=1}^{i}\frac{E(\ominus t_{j}|t)}{\prod_{1\leq \ell \neq
j\leq i}t_{\ell }\ominus t_{j}},\qquad i=1,\ldots ,n\;.  \label{Er2}
\end{eqnarray}%
In particular, for $i=1$ we have $H_{k}=H(t_{N}|t)$ and $E_{n}=E(t_{1}|t)$.
\end{corollary}

\begin{proof}
Setting $x=t_{i}$ in (\ref{Hr}) and $x=\ominus t_{i}$ in (\ref{Er}) we
obtain a linear system of equations expressing $H(t_{i}|t)$ and $E(\ominus
t_{i}|t)$ in terms of the (operator) coefficients $H_{r}$ and $E_{r}$
respectively. The corresponding matrices are lower triangular and therefore
can be easily inverted to produce the stated expressions.

It follows from Prop \ref{prop:integrability} that all these operators
commute.
\end{proof}

Together with Prop \ref{prop:comb_transfer} the last result allows one to
compute the action of $H_{r}$ and $E_{r}$ in the spin-basis $\{v_{\lambda
}\}_{\lambda\subset (k^n)}\subset \mathcal{V}_{n}$.

\begin{example}
\label{ex:2} We continue Example \ref{ex:1} with $\operatorname{Gr}_{1,3}=\mathbb{%
P}^{2}$. It follows from (\ref{Hr2}) that%
\begin{equation*}
H_{1}=\frac{H(t_{2}|t)}{t_{2}\ominus t_{3}}+\frac{H(t_{3}|t)}{t_{3}\ominus t_{2}}%
,\qquad H_{2}=H(t_{3}|t)\;.
\end{equation*}%
Employing the weights shown in Figure \ref{fig:gr13ex}, 
\begin{eqnarray*}
\langle \tilde{v}_{010}|H(x|t)v_{010}\rangle &=&(x\ominus t_{1})(x\ominus
t_{3}), \\
\langle \tilde{v}_{001}|H(x|t)v_{010}\rangle &=&(x\ominus t_{1})(1+\beta
x\ominus t_{3}) \\
\langle \tilde{v}_{100}|H(x|t)v_{010}\rangle &=&q(1+\beta x\ominus t_{1})
\end{eqnarray*}%
we arrive at the matrix elements 
\begin{eqnarray*}
\langle \tilde{v}_{010}|H_{1}v_{010}\rangle &=&\frac{(t_{2}\ominus
t_{1})(t_{2}\ominus t_{3})}{t_{2}\ominus t_{3}}+0=t_{2}\ominus t_{1} \\
\langle \tilde{v}_{001}|H_{1}v_{010}\rangle &=&\frac{(t_{2}\ominus
t_{1})(1+\beta t_{2}\ominus t_{3})}{t_{2}\ominus t_{3}}+\frac{t_{3}\ominus
t_{1}}{t_{3}\ominus t_{2}}=1+\beta ~t_{2}\ominus t_{1} \\
\langle \tilde{v}_{100}|H_{1}v_{010}\rangle &=&q\frac{(1+\beta t_{2}\ominus
t_{1})}{t_{2}\ominus t_{3}}+q\frac{(1+\beta t_{3}\ominus t_{1})}{%
t_{3}\ominus t_{2}}=0
\end{eqnarray*}%
From these we obtain,%
\begin{equation}
H_{1}v_{010}=t_{2}\ominus t_{1}v_{010}+(1+\beta t_{2}\ominus t_{1})v_{001}\;.
\label{Gr13Pieri1}
\end{equation}%
In an analogous fashion one finds,%
\begin{equation*}
H_{2}v_{010}=t_{3}\ominus t_{1}v_{001}+q(1+\beta t_{1}\ominus t_{1})v_{100}
\end{equation*}%
and using the weights in Figure \ref{fig:gr13exb} 
\begin{gather}
H_{1}v_{001}=t_{3}\ominus t_{1}v_{001}+q(1+\beta t_{3}\ominus t_{1})v_{100}  \label{Gr13Pieri2} \\
H_{2}v_{001}=(t_{3}\ominus t_{2})(t_{3}\ominus t_{1})v_{001}+q(t_{3}\ominus
t_{2})(1+\beta t_{3}\ominus t_{1})v_{100 }+q(1+\beta t_{3}\ominus
t_{2})v_{010}  \notag
\end{gather}%
Below we will define a product by $v_{r}\circledast v_{s}=H_{r}v_{s}$. Upon
setting $\beta =-1$ and $t_{4-i}=1-e^{\varepsilon _{i}}$ with $i=1,2,3$ the
above formulae then match the product expansions for quantum equivariant
K-theory of $\mathbb{P}^{2}$ stated by Buch and Mihalcea in \cite[Sec 5.5]%
{BM}.
\end{example}

\subsubsection{Functional relation \& quantum Pieri-Chevalley rule}

The coefficients (\ref{Hr2}) and (\ref{Er2}) of the transfer matrices are
algebraically dependent. We now derive the functional relation (\ref{ideal
def}) which allows one to deduce this dependence and as a byproduct of our
computation we give an explicit formula for the action of $H_{1}$ in the
spin basis \eqref{spin basis}.

Let $u_{j}=\sigma _{j}^{-}\sigma _{j+1}^{+}$ for $j=1,\ldots ,N-1$ and $%
u_{N}=q\sigma _{1}^{+}\sigma _{N}^{-}$. Define the following operator on $%
\mathcal{V}^{q},$%
\begin{equation}
\bar{H}_{1}=\sum_{j=1}^{N}u_{j}+\beta \sum_{|j_{1}-j_{2}|\operatorname{mod}%
N>1}u_{j_{1}}u_{j_{2}}+\beta ^{2}\sum_{|j_{a}-j_{b}|\operatorname{mod}%
N>1}u_{j_{1}}u_{j_{2}}u_{j_{3}}+\cdots  \label{h1}
\end{equation}%
as a formal power series in $\beta $. Note that the sums only run over
indices where $|j_{a}-j_{b}|\operatorname{mod}N>1$ which ensures that all the $u_{j}$%
's in each monomial commute. Obviously, only finitely many terms act
non-trivially for finite $N$ and the series therefore terminates.

\begin{lemma}
Acting with $\bar{H}_{1}$ on a spin basis vector $v_{\mu }\in \mathcal{V}%
_{n} $ one obtains%
\begin{equation}
\bar{H}_{1}v_{\mu }=\sum_{\substack{ \mu \rightrightarrows ^{\ast }\lambda
\lbrack d]  \\ d=0,1}}q^{d}\beta ^{|\lambda /d/\mu |-1}v_{\lambda },
\label{h1action}
\end{equation}%
where the sum runs over all boxed partitions $\lambda \subset (k^{n})$ such
that either $\lambda /0/\mu =\lambda /\mu $ or $\lambda /1/\mu $ are toric
diagrams which contain at most one box in each column and row and $\lambda
\neq \mu $.
\end{lemma}

\begin{proof}
Using the bijection between binary strings and partitions detailed in
Section \ref{sec:partitions} and the definition of cylindric loops in
Section \ref{sec:torictab}, one proves that either $u_{j}v_{\mu
}=q^{d}v_{\lambda }$ where one adds a box with coordinates $(x,y)$ and $%
j=n+y-x$ to obtain $\lambda $ (or $\lambda \lbrack 1]$ if $d=1$ and $j=N$)
or $u_{j}v_{\mu }=0$. The assertion then easily follows from the fact that
all $u_{j}$'s in each monomial term commute.
\end{proof}
\begin{remark}
If we identify the spin basis $v_\lambda$ defined in (\ref{spin basis}) with the Schubert structure sheaves 
$[\mathcal{O}_\lambda]$ in the K-theory ring of the Grassmannian, then the right hand side of (\ref{h1action}) 
coincides with the product expansion of the quantum Pieri rule in \cite[Lem 5.14]{BM}, that is, the operator 
$\bar H_1$ is the multiplication operator with $[\mathcal{O}_1]$. Note that in {\em loc. cit.} mainly the non-equivariant 
quantum K-theory of Grassmannians is discussed, except for Section 5.5. where also the equivariant quantum K-theory rings of $\mathbb{P}^1=\operatorname{Gr}_{1,2}$ and $\mathbb{P}^2=\operatorname{Gr}_{1,3}$ are presented. As we will see shortly in the equivariant case an additional factor depending on the equivariant parameters $t_j$ is appearing. We will discuss the limit $t_j\rightarrow 0$ which describes the non-equivariant quantum K-theory ring for Grassmannians in a separate section at the end.
\end{remark}
Define the following diagonal matrix in the spin basis (\ref{spin basis}),
\begin{equation}\label{Delta}
\Delta (x_i|t)v_\lambda=\left( \tprod\nolimits_{j\in I_{\lambda }}t_{j}\ominus x_i\right) \left(
\tprod\nolimits_{j\in I_{\lambda ^{\ast }}}x_i\ominus t_{j}\right)\,v_\lambda
\end{equation}
\begin{proposition}
The transfer matrices obey the following functional operator identity%
\begin{equation}
H(x_i|t)E(\ominus x_i|t)=(1+\beta \bar{H}_{1})\Delta (x_i|t)+q\cdot 1\;.  \label{func_eqn1}
\end{equation}%
In particular, we have that $H(t_{j}|t)E(\ominus t_{j}|t)=q\cdot 1$ for all $%
j=1,\ldots ,N$ which entails non-trivial identities between the
coefficients $\{H_{r}\}$ and $\{E_{r}\}$ defined in (\ref{Hr}), (\ref{Er}).
\end{proposition}

\begin{proof}
\begin{figure}[tbp]
\begin{equation*}
\includegraphics[scale=0.35]{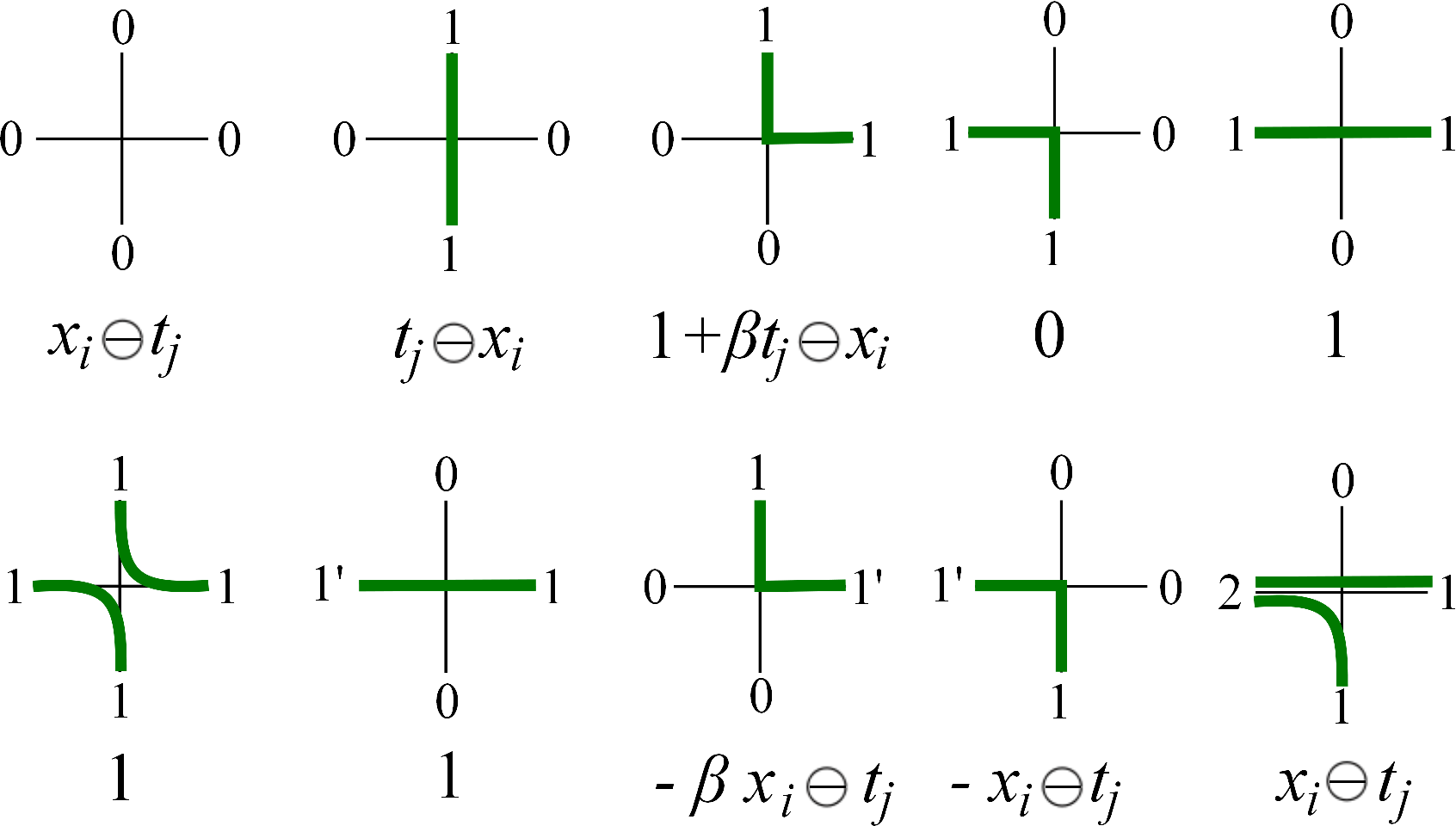}
\end{equation*}%
\caption{The vertex configurations corresponding to the operator $%
L_{i+1j}^{\prime }(\ominus x_{i}|t_j)L_{i,j}(x_{i}|t_j)$.}
\label{fig:func_eqn}
\end{figure}
A computation along similar lines as in \cite{VicOsc}. Since $i$ is arbitrary here, we set $i=1$. The idea is to
analyse the action of $\mathcal{L}_{(12)j}(x_1|t_j)=L_{1j}^{\prime }(\ominus
x_1|t_j)L_{2j}(x_1|t_j):W(x_1)\otimes V(t_{j})\rightarrow W(x_1)\otimes V(t_{j})$ where $%
W(x_1)=V(\ominus x_1)\otimes V(x_1)=\mathcal{R}(x_1)\otimes V^{\otimes 2}$ with
respect to the basis vectors%
\begin{eqnarray*}
w_{0} &=&v_{0}\otimes v_{0},\quad w_{1}=v_{0}\otimes v_{1}+v_{1}\otimes
v_{0}, \\
w_{1^{\prime }} &=&v_{0}\otimes v_{1},\quad w_{2}=v_{1}\otimes v_{1}~.
\end{eqnarray*}%
We find that  (suppressing the explicit dependence of the $\mathcal{L}$-operator on $x_1,t_j$ in the notation),%
\begin{eqnarray*}
\mathcal{L}_{(12)j}w_{0}\otimes v_{0} &=&x_1\ominus t_{j}~w_{0}\otimes v_{0} \\
\mathcal{L}_{(12)j}w_{0}\otimes v_{1} &=&t_{j}\ominus x_1~w_{0}\otimes
v_{1}+(1+\beta t_{j}\ominus x_1)w_{1}\otimes v_{0}-\beta t_{j}\ominus
x~w_{1^{\prime }}\otimes v_{0} \\
\mathcal{L}_{(12)j}w_{1}\otimes v_{0} &=&w_{1}\otimes v_{0} \\
\mathcal{L}_{(12)j}w_{1}\otimes v_{1} &=&w_{1}\otimes v_{1} \\
\mathcal{L}_{(12)j}w_{1^{\prime }}\otimes v_{0} &=&w_1\otimes v_0-x_1\ominus
t_{j}~w_{0}\otimes v_{1} \\
\mathcal{L}_{(12)j}w_{1^{\prime }}\otimes v_{1} &=&0 \\
\mathcal{L}_{(12)j}w_{2}\otimes v_{0} &=&(1+\beta x_1\ominus t_{j})w_{1}\otimes
v_{1}-\beta x_1\ominus t_{j}~w_{1^{\prime }}\otimes v_{1} \\
\mathcal{L}_{(12)j}w_{2}\otimes v_{1} &=&0
\end{eqnarray*}%
This action of $\mathcal{L}_{(12)j}(x_1|t_j)$ in the spin basis (\ref{spin basis})
can be encoded in terms of the vertex configurations shown in Figure \ref%
{fig:func_eqn} with labels $0,1,1^{\prime },2$ similarly as we can deduce
the action of $L$ and $L^{\prime }$ from the vertex configurations in Figure %
\ref{fig:5vmodels}. Thus, the operator product $H(x_1|t)E(\ominus x_1|t)$ can be
written as the partial trace%
\begin{equation*}
E(\ominus x_1|t)H(x_1|t)=\operatorname{Tr}_{V\otimes V}\left( 
\begin{smallmatrix}
1 & 0 & 0 & 0 \\ 
0 & q & 0 & 0 \\ 
0 & 0 & q & 0 \\ 
0 & 0 & 0 & q^{2}%
\end{smallmatrix}
\right) \mathcal{L}_{(12)N}(x_1|t_N)\cdots \mathcal{L}_{(12)1}(x_1|t_1)
\end{equation*}
and its matrix elements in the quantum space $\mathcal{V}_{n}^{q}$ are sums
over the possible vertex configurations of Figure \ref{fig:func_eqn} in a
single lattice row of length $N$. This lattice row is closed and forms a
circle of circumference $N$, since the partial trace together with the
matrix containing the deformation parameter $q$ imposes quasi-periodic
boundary conditions. Due to these periodicity conditions, one finds the
following constraints:

\begin{itemize}
\item the last vertex in the bottom row of Figure \ref{fig:func_eqn} cannot
occur;

\item the 2nd and 3rd vertex from the right in the top row always have to
come as a pair, but since one of them has weight zero their contribution can
be discarded;

\item configurations involving the second vertex from the left in the bottom
row do not contribute as they eventually lead to a vertex configuration
shown at the 2nd position from the right in the top row which has weight
zero;

\item the 2nd and 3rd vertex from the right in the bottom row always have to
come as an adjacent pair and it are these vertices which give rise to the
term involving $\beta \bar H_1$ as they correspond to shifting a 1-letter in
a binary string to the right.
\end{itemize}

From these conditions, which can be checked graphically, one then deduces
the asserted identity (\ref{func_eqn}) as only a very restricted number of
vertices in Figure \ref{fig:func_eqn} remain.
\end{proof}

\begin{corollary}[equivariant quantum Pieri-Chevalley rule]
We have the following explicit action of $H_{1}$ in terms of the basis $%
\{v_{\lambda }\}_{\lambda\subset (k^n)}\subset \mathcal{V}_{n},$%
\begin{equation}
(1+\beta H_{1})v_{\mu }=\frac{\Pi (t_{\mu })}{\Pi (t_{\emptyset })}\sum 
_{\substack{ \mu\rightrightarrows\lambda[d]  \\ d=0,1}}q^{d}\beta ^{|\lambda
/d/\mu |}v_{\lambda },  \label{H1}
\end{equation}%
where the sum runs over all $\lambda \subset (k^{n})$ such that either $%
\lambda /\mu $ or $\lambda /1/\mu $ is a skew diagram which contains at most
one box in each column or row. Moreover, the identity (\ref{func_eqn1}) can
be rewritten as 
\begin{equation}
H(x_i|t)E(\ominus x_i|t)=\prod_{j=1}^{n}(t_{j}\ominus
x_i)\prod_{j=n+1}^{N}(x_i\ominus t_{j})\,(1+\beta H_{1})+q\cdot 1\;.
\label{func_eqn}
\end{equation}
\end{corollary}

\begin{proof}
Acting with the first term on the right hand side of (\ref{func_eqn1}) on a
basis vector $v_{\lambda }$ we obtain%
\begin{eqnarray*}
(1+\beta \bar{H}_{1})\Delta(x_i|t)v_{\lambda }&=&
\left( \tprod\nolimits_{j\in I_{\lambda }}t_{j}\ominus x_i\right) \left(
\tprod\nolimits_{j\in I_{\lambda ^{\ast }}}x_i\ominus t_{j}\right) ~(1+\beta 
\bar{H}_{1})v_{\lambda } \\
&=&\left( \tprod\nolimits_{j=1}^{n}t_{j}\ominus x_i\right) \left(
\tprod\nolimits_{j=n+1}^{N}x_i\ominus t_{j}\right) ~\frac{\Pi (t_{\lambda })}{%
\Pi (t_{\emptyset })}(1+\beta \bar{H}_{1})v_{\lambda }
\end{eqnarray*}%
On the other hand using the expansions (\ref{Hr}) and (\ref{Er}) we see that
the coefficients of the leading factorial powers are%
\begin{eqnarray*}
H(x_i|t) &=&(x_i|\ominus t^{\prime })^{k}(1+\beta H_{1})+\cdots \\
E(x_i|t) &=&(x_i|t)^{n}(1+\beta E_{1})+\cdots
\end{eqnarray*}%
from which we deduce the desired identities with help of the left hand side
of (\ref{func_eqn1}) and (\ref{facpower_id}). Namely, we have%
\begin{eqnarray*}
(-1)^{n}\frac{(1+\beta x_i)^{n}}{\Pi (t_{\emptyset })}~E(\ominus x)
&=&(x_i|\ominus t)^{n}\sum_{r=0}^{n}(-1)^{r}\beta ^{r}(E_{r}+\beta
E_{r+1})+~\ldots \\
&=&(x_i|\ominus t)^{n}\cdot 1+~\ldots
\end{eqnarray*}%
where the omitted terms involve factorial powers $(x_i|\ominus t)^{p}$ with $%
p<n$ and we have set $E_{0}=1$, $E_{n+1}=0$. Thus, 
\begin{equation*}
(-1)^{n}\frac{(1+\beta x_i)^{n}}{\Pi (t_{\emptyset })}~E(\ominus
x_i|t)H(x_i|t)=(x_i|\ominus t)^{N}(1+\beta H_{1})+~\ldots
\end{equation*}%
and the assertion follows.
\end{proof}

\begin{remark}
After this article had been submitted, the work \cite{BCMP} appeared in which the 
equivariant quantum K-theory of cominiscule varieties is discussed. As part of this 
more general discussion the formula (\ref{H1}) has been proven to describe the 
multiplication with the Schubert structure sheaf $[\mathcal{O}_1]$ in the case of 
the Grassmannian; see Thm 3.9  for $X=Gr_{n,N}$ as well 
as Remark 5.13 in {\em loc. cit.} We note that our functional relation (\ref{func_eqn}) 
as well as Prop \ref{prop:comb_transfer} also encode the equivariant quantum Pieri rules describing the multiplication by structure 
sheaves  $[\mathcal{O}_r]$ and $[\mathcal{O}_{1^r}]$ with $r>1$.
\end{remark}
To facilitate the comparison between multiplication rules and the action of the transfer matrices we present a simple example.
\begin{example}
\label{ex:3} We consider once more the example $\operatorname{Gr}_{1,3}=\mathbb{P}%
^{2}$. It follows from (\ref{Gr13Pieri1}) and (\ref{Gr13Pieri2}) in Example %
\ref{ex:2} that%
\begin{eqnarray*}
(1+\beta H_{1})v_{010} &=&(1+\beta t_{2}\ominus t_{1})(v_{010}+v_{001}) \\
(1+\beta H_{1})v_{001} &=&(1+\beta t_{3}\ominus t_{1})(v_{001}+qv_{100})\;.
\end{eqnarray*}
We compare this against the quantum Pieri-Chevalley rule (\ref{H1}). The
binary string $010$ corresponds to the partition $\mu =(1)$ with a single
box and $001$ to the partition $\mu =(2)$. Thus, in the first case the only
partitions $\lambda $ for which $\lambda /\mu $ contains at most a single
box in each row and column are $\lambda =(1)$ and $\lambda =(2)$. For $\mu
=(2)$ we obtain $\lambda =(2)$ and $\lambda =\emptyset $, since in the
latter case the cylindric loop $\lambda \lbrack 1]$ contains 3 boxes and $%
\lambda /1/\mu $ has one box in one row.

Let us now consider the functional relation (\ref{func_eqn}). For $N=3$ and $%
n=1$ with $x=x_1$, expand the transfer matrices into factorial powers as follows%
\begin{gather*}
H(x|t)=(x\ominus t_{2})(x\ominus t_{3})(1+\beta H_{1})+(x\ominus
t_{3})(H_{1}+\beta H_{2})+H_{2} \\
-\frac{1+\beta x}{1+\beta t_{1}}E(\ominus x|t)=(x\ominus t_{1})(1+\beta E_{1})-%
\frac{1+\beta x}{1+\beta t_{1}}E_{2}=(x\ominus t_{1})-E_{1}
\end{gather*}%
The left hand side of (\ref{func_eqn}) yields%
\begin{multline*}
-\frac{1+\beta x}{1+\beta t_{1}}E(\ominus x|t)H(x|t)=(x|\ominus t)^{3}(1+\beta
H_{1}) \\
-(x\ominus t_{2})(x\ominus t_{3})[(1+\beta H_{1})E_{1}-(1+\beta t_{2}\ominus
t_{1})(H_{1}+\beta H_{2})] \\
+(x\ominus t_{3})[(1+\beta t_{3}\ominus t_{1})H_{2}+(t_{2}\ominus
t_{1}-E_{1})(H_{1}+\beta H_{2})]-E_{1}H_{2}(1-t_{3}\ominus t_{1})
\end{multline*}%
while the right hand side reads 
\begin{eqnarray*}
-\frac{1+\beta x}{1+\beta t_{1}}E(\ominus x|t)H(x|t)&=&(x|\ominus t)^{3}(1+\beta
H_{1})-\frac{1+\beta x}{1+\beta t_{1}}~q \\
&=&(x|\ominus t)^{3}(1+\beta H_{1})-(1+\beta t_{3}\ominus t_{1})q[1+\beta
(x\ominus t_{3})]
\end{eqnarray*}%
Comparing the coefficients of each factorial power we obtain the relations%
\begin{gather*}
E_{1}H_{2}(1-t_{3}\ominus t_{1})=(1+\beta t_{3}\ominus t_{1})q \\
(E_{1}-t_{2}\ominus t_{1})(H_{1}+\beta H_{2})-(1+\beta t_{3}\ominus
t_{1})H_{2}=\beta (1+\beta t_{3}\ominus t_{1})q \\
(1+\beta H_{1})E_{1}=(1+\beta t_{2}\ominus t_{1})(H_{1}+\beta H_{2})
\end{gather*}%
We will see in a subsequent section that there is an easier way to describe
the algebraic dependence between the $H_r$\rq{}s and $E_r$\rq{}s which 
avoids these rather complicated looking relations. However,
in the non-equivariant limit where $t_{j}=0$ for $j=1,2,3$, they simplify to%
\begin{equation*}
E_{1}=H_{1},\qquad E_{1}^{2}=H_{2},\qquad E_{1}^{3}=q\;.
\end{equation*}
\end{example}


\section{Bethe vectors as idempotents}

We now consider the eigenvalue problem of the transfer matrices \eqref{defH} and \eqref{defE}. Eigenvalues
and eigenvectors can be explicitly contructed using the Yang-Baxter algebra,
this general approach is known as \emph{quantum inverse scattering method }%
or \emph{algebraic Bethe ansatz}. Using the eigenvectors, called \emph{Bethe
vectors} in the quantum integrable systems literature, we then define for
each subspace $\mathcal{V}_{n}^{q}$ a generalised matrix ring $qh_{n}^{\ast
} $ by identifying appropriate renormalised versions of the Bethe vectors as
its idempotents.

\subsection{Bethe vectors \& factorial Grothendieck polynomials}

Let $y=(y_{1},\ldots ,y_{n})$ and $z=(z_{1},\ldots ,z_{k})$ be some
indeterminates. Recall McNamara's definition of factorial Grothendieck
polynomials from Section \ref{sec:grothendieck} and the definition of the
Yang-Baxter algebra (\ref{row_yba}).

\begin{proposition}
Let $\lambda \subset (k^{n})$. Then we have the following equalities for the 
$C$ and $B^{\prime }$-operators of the Yang-Baxter algebras constructed from the $L$ and $L\rq{}$-operators,%
\begin{eqnarray}
C(y_{1}|t)\cdots C(y_{n}|t)v_{\lambda } &=&G_{\lambda }(y|\ominus
t)~v_{0}\otimes \cdots \otimes v_{0}  \label{Caction} \\
B^{\prime }(z_{1}|t)\cdots B^{\prime }(z_{k}|t)v_{\lambda } &=&G_{\lambda
^{\prime }}(z|t^{\prime })~v_{1}\otimes \cdots \otimes v_{1}
\label{B'action}
\end{eqnarray}%
when acting on the basis vector $v_{\lambda }$ in $\mathcal{V}_{n}$.
\end{proposition}

\begin{proof}
We only sketch the proof leaving technical details to the reader to verify.
Since $\lambda \subset (k^{n})$ the corresponding binary string $b(\lambda )$
contains $n$ 1-letters. From the definition (\ref{row_yba}) it follows that $%
C(y_i|t):\mathbb{Z}[y_i]\otimes \mathcal{V}_{n}\rightarrow \mathbb{Z}[y_i]\otimes 
\mathcal{V}_{n-1}$ and that $B^{\prime }(z_i|t):\mathbb{Z}[z_i]\otimes \mathcal{V}%
_{n}\rightarrow \mathbb{Z}[z_i]\otimes \mathcal{V}_{n+1}$. This implies that $%
C(y_{1}|t)\cdots C(y_{n}|t)v_{\lambda }$ is a multiple of $v_{0}\otimes \cdots
\otimes v_{0}$ and $B^{\prime }(z_{1}|t)\cdots B^{\prime }(z_{k}|t)v_{\lambda }$
a multiple of $v_{1}\otimes \cdots \otimes v_{1}$. Denote the
proportionality factors, i.e. the respective matrix elements, by $\langle
\tilde v_{0\cdots 0}|C(y_{n}|t)\cdots C(y_{1}|t)v_\lambda \rangle $ and $\langle\tilde v_{1\cdots 1}|B^{\prime
}(z_{k}|t)\cdots B^{\prime }(z_{1}|t)v_\lambda \rangle $. Each can be identified
with a weighted sum $\sum_{\mathcal{C} }\operatorname{wt}(\mathcal{C} )$ over vertex
configurations $\mathcal{C} $ on a lattice with certain fixed boundary
conditions; see Figure \ref{fig:dualB} and \ref{fig:dualC} for simple
examples with $N=4$, $n=k=2$. Here $\operatorname{wt}(\mathcal{C} )=\prod_{\mathrm{v}%
\in \mathcal{C} }\operatorname{wt}(\mathrm{v})$ with $\mathrm{v}$ being one of the
vertex configurations in Figure \ref{fig:5vmodels} and the respective weight 
$\operatorname{wt}(\mathrm{v})$ takes the values $a,a^{\prime },b,b^{\prime
},c,c^{\prime }$ as specified in Figure \ref{fig:5vmodels} or zero if it is none of the
allowed vertices. We now identify lattice configurations $\mathcal{C} $ with
certain \emph{sets} of set-valued tableaux.

Define a surjection $\operatorname{SVT}(\lambda )\twoheadrightarrow \operatorname{SST}%
(\lambda )$ which assigns to each set-valued tableau $\mathcal{T}$ the
semi-standard tableau $T:\lambda \rightarrow \lbrack n]$ with $T(i,j)=\min 
\mathcal{T}(i,j)$. Given a semistandard tableau $T$ with entries $\leq n$,
there exists a unique \textquotedblleft maximal\textquotedblright\ set
valued tableau $\mathcal{T}^{\max }$ in its pre-image that has the maximum
number of entries $\leq n$, i.e. $|\mathcal{T}^{\max }|\geq |\mathcal{T}|$
for all $\mathcal{T}$ in the pre-image of $T$. The lattice path
configurations are in one-to-one correspondence with these maximal
set-valued tableaux (and therefore semi-standard tableaux) of shape $\lambda 
$ and $\lambda ^{\prime }$. We state the bijections:

\emph{Vicious walkers}. Starting from the bottom, place for each
vertex labelled with a \emph{bullet} in lattice row $i$ a box labelled with $%
i$ in the $j$th row of the Young tableau where $j$ is the total number of
paths crossing the row to the right of the vertex. Vertices with a \emph{%
square} in row $i$ mean that an entry $i$ is placed within an existing box
of the $j$th row of the Young tableau where $j$ is again the total number of
paths crossing the row to the right of the vertex. The resulting set-valued
tableau has shape $\lambda $.

\begin{figure}[tbp]
\begin{equation*}
\includegraphics[scale=0.5]{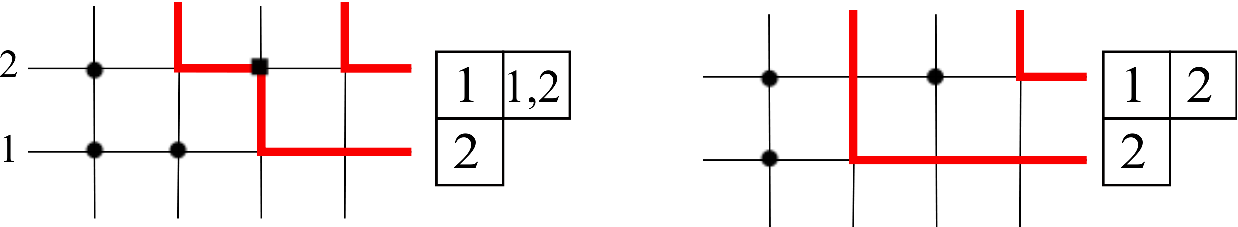}
\end{equation*}%
\caption{The lattice configurations corresponding to the $C$-operator.}
\label{fig:dualB}
\end{figure}

\emph{Osculating walkers}. Consider the rightmost path and add in the
first column (counting from left to right) of the Young diagram of $\lambda
^{\prime }$ a box with the lattice row number where a vertex with a bullet
occurs. If a vertex with a square occurs in row $i$ then place an $i$ into
an existing box in this column. Then do the same for the next path writing
the lattice row numbers now in the second column of the Young tableau etc.
If there are no vertices with a bullet leave the column empty. The resulting
set-valued tableau has shape $\lambda ^{\prime }$.%
\begin{figure}[tbp]
\begin{equation*}
\includegraphics[scale=0.5]{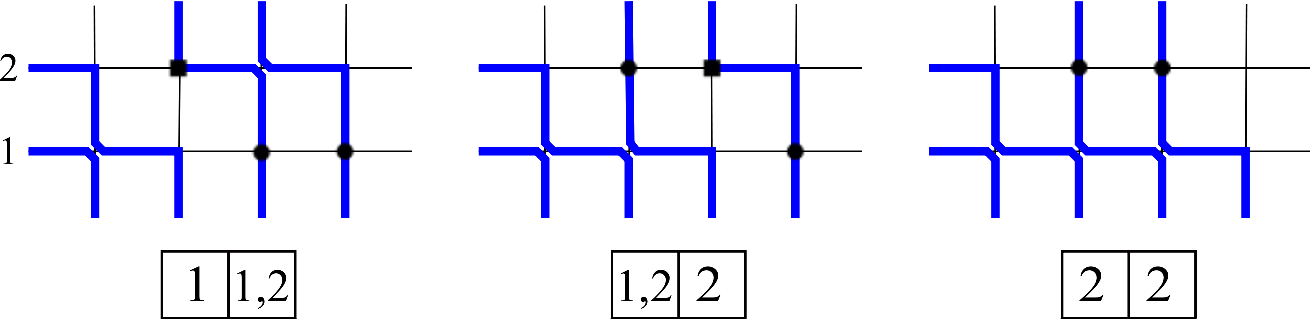}
\end{equation*}%
\caption{The lattice configurations corresponding to the $B^{\prime}$%
-operator.}
\label{fig:dualC}
\end{figure}

Let $\mathcal{C} $ be a lattice configuration of the $C$-operator ($B^{\prime }$%
-operator) and denote by $T_{\mathcal{C} }$ the corresponding semi-standard
tableau under the surjection $\operatorname{SVT}(\lambda )\twoheadrightarrow 
\operatorname{SST}(\lambda )$. Note that each vertex with a bullet contributes a
factor $y_{i}\ominus t_{j}$ and each vertex with a square a factor $(1+\beta
y_{i}\ominus t_{j})$; see Figure \ref{fig:5vmodels}. Here$\ i,j$ are the
lattice row and column numbers where the vertex occurs and we number lattice
rows decreasingly from top to bottom and lattice rows increasingly from left
to right. This allows us to deduce the following result.

\begin{lemma}
Denote by $\mathcal{C}$ a lattice configuration of the $C$-operator and by $\mathcal{C}\rq{}$ 
a lattice configuration of the $B\rq{}$-operator. Then we have the following identities%
\begin{eqnarray}
\operatorname{wt}(\mathcal{C}) &=&\prod_{\substack{ \langle i,j\rangle \in
\lambda  \\ r=T_{\mathcal{C} }(i,j)}}\left( y_{r}\ominus t_{r+j-i}\right) ~\prod 
_{\substack{ \langle i,j\rangle \in \lambda  \\ r\in \mathcal{T}_{\mathcal{C}
}^{\max }(i,j)\backslash T_{\mathcal{C} }(i,j)}}(1+\beta y_{r}\ominus t_{r+j-i}) 
\notag \\
&=&\sum_{\mathcal{T}}\beta ^{|\mathcal{T}|-|\lambda |}\prod_{\substack{ %
\langle i,j\rangle \in \lambda ^{\vee }  \\ r\in \mathcal{T}(i,j)}}%
y_{r}\ominus t_{r+j-i}
\end{eqnarray}%
and%
\begin{eqnarray}
\limfunc{wt}(\mathcal{C}\rq{}) &=&\prod_{\substack{ \langle i,j\rangle
\in \lambda ^{\prime }  \\ r=T_{\mathcal{C}\rq{} }(i,j)}}\left( z_{r}\oplus
t_{r+j-i}\right) ~\prod_{\substack{ \langle i,j\rangle \in \lambda ^{\prime
}  \\ r\in \mathcal{T}_{\mathcal{C}\rq{} }^{\max }(i,j)\backslash T_{\mathcal{C}\rq{}}(i,j)}}%
(1+\beta y_{r}\oplus t_{r+j-i})  \notag \\
&=&\sum_{\mathcal{T}}\beta ^{|\mathcal{T}|-|\lambda |}\prod_{\substack{ %
\langle i,j\rangle \in \lambda ^{\ast }  \\ r\in \mathcal{T}(i,j)}}%
z_{r}\oplus t_{r+j-i}^{\prime },
\end{eqnarray}%
where the sums run over all set-valued tableaux $\mathcal{T}$ of shape $%
\lambda $ and $\lambda ^{\prime }$ which obey the condition that 
$\min \mathcal{T}(i,j)=T_{\mathcal{C} }(i,j)$ and $\min \mathcal{T}(i,j)=T_{\mathcal{C}\rq{} }(i,j)$ , respectively.
\end{lemma}
Thus, it follows that 
\begin{eqnarray*}
\langle \tilde v_{0\cdots 0}|C(y_{n}|t)\cdots C(y_{1}|t)v_\lambda \rangle &=&\sum_{\mathcal{C}}%
\limfunc{wt}(\mathcal{C})=\sum_{\mathcal{T}\in \operatorname{SVT}(\lambda )}\beta
^{|\mathcal{T}|-|\lambda ^{\vee }|}\prod_{\substack{ \langle i,j\rangle \in
\lambda  \\ r\in \mathcal{T}(i,j)}}y_{r}\ominus t_{r+j-i} \\
\langle \tilde v_{1\cdots 1}|B^{\prime }(z_{k}|t)\cdots B^{\prime }(z_{1}|t)v_\lambda \rangle
&=&\sum_{\mathcal{C}\rq{}}\operatorname{wt}(\mathcal{C}\rq{})=\sum_{%
\mathcal{T}\in \operatorname{SVT}(\lambda ^{\prime })}\beta ^{|\mathcal{T}%
|-|\lambda ^{\vee }|}\prod_{\substack{ \langle i,j\rangle \in \lambda
^{\prime }  \\ r\in \mathcal{T}(i,j)}}z_{r}\oplus t_{r+j-i}^{\prime }
\end{eqnarray*}%
which proves the assertion as the last two equations are McNamara's
definition (\ref{facG}) of factorial Grothendieck polynomials.
\end{proof}

Introduce the so-called off-shell Bethe vector in $\mathbb{Z}[y_{1},\ldots
,y_{n}]\otimes \mathcal{V}_{n}$ and its dual%
\begin{eqnarray}
|y_{1},\ldots ,y_{n}\rangle &=&B(y_{n}|t)\cdots B(y_{1}|t)v_{0\cdots 0}\label{Bethev} \\
\langle y_{1},\ldots ,y_{n}| &=&[C^{\vee }(y_{n}|t)\cdots C^{\vee
}(y_{1}|t) ]^T \tilde v_{0\cdots 0}\label{leftBethev}
\end{eqnarray}%
where the upper index $T$ denotes the transpose. Similarly, we define for a $k$%
-tuple $z=(z_{1},\ldots ,z_{k})$ the complementary Bethe vector in $\mathbb{Z%
}[z_{1},\ldots ,z_{k}]\otimes \mathcal{V}_{n}$ and its dual as%
\begin{eqnarray}
|z_{1},\ldots ,z_{k}\rangle &=&C^{\prime }(z_{k}|t)\cdots C^{\prime
}(z_{1}|t)v_{1\cdots 1}  \label{Bethev'} \\
\langle z_{1},\ldots ,z_{k}| &=&[B^{\ast }(z_{k}|t)\cdots B^{\ast
}(z_{1}|t)]^T\tilde v_{1\cdots 1}\;.  \label{leftBethev'}
\end{eqnarray}%
From (\ref{YBE}) one deduces that
the $B,B^{\ast },C^{\prime },C^{\vee }$-operators each commute for different
values of the spectral variables $y_i$ and $z_i$. Hence, we can conclude that the vectors (%
\ref{Bethev}), (\ref{Bethev'}) as well as their dual versions are symmetric
in the $y$ and $z$-variables.
We now identify the coefficients of the off-shell Bethe vectors with
factorial Grothendieck polynomials.

\begin{proposition}
Recall the definitions of $\lambda ^{\vee }$, $\lambda ^{\ast }$ from Sec %
\ref{sec:partitions} and set once more $\ominus t^{\prime }=(\ominus
t_{N+1},\ldots ,\ominus t_{2},\ominus t_{1})$. Then we have the identities%
\begin{eqnarray}
|y_{1},\ldots ,y_{n}\rangle &=&\Pi (y)\sum_{\lambda \in (k^{n})}\frac{%
G_{\lambda ^{\vee }}(y_{1},\ldots ,y_{n}|\ominus t^{\prime })}{\Pi
(t_{\lambda })}~v_{\lambda }  \label{rightBethe} \\
|z_{1},\ldots ,z_{k}\rangle &=&\Pi (z)\sum_{\lambda \in (k^{n})}G_{\lambda
^{\ast }}(z_{1},\ldots ,z_{k}|t)\Pi (t_{\lambda ^{\ast }})v_{\lambda },
\label{rightBethe'}
\end{eqnarray}%
For the dual vectors we obtain instead%
\begin{eqnarray}
\langle y_{1},\ldots ,y_{n}| &=&\sum_{\lambda \in (k^{n})}G_{\lambda
}(y_{1},\ldots ,y_{n}|\ominus t)~\tilde{v}_{\lambda }  \label{leftBethe} \\
\langle z_{1},\ldots ,z_{k}| &=&\sum_{\lambda \in (k^{n})}G_{\lambda
^{\prime }}(z_{1},\ldots ,z_{k}|t^{\prime })~\tilde{v}_{\lambda }\;.
\label{leftBethe'}
\end{eqnarray}
\end{proposition}

\begin{proof}
The proof is very similar to the one of the previous identities with some
minor changes in the bijections between lattice configurations of
non-intersecting paths and maximal set-valued tableaux.

As before $\langle \tilde v_\lambda |B(y_{n}|t)\cdots
B(y_{1}|t)v_{0\cdots 0}\rangle $ and $\langle\tilde v_\lambda |C^{\prime }(z_{k}|t)\cdots C^{\prime
}(z_{1}|t)v_{1\cdots1}\rangle $ can each be identified with a weighted sum $\sum_{\mathcal{C}
}\operatorname{wt}(\mathcal{C})$ over lattice configurations $\mathcal{C}$ with
certain fixed boundary conditions; examples are provided in Figure \ref%
{fig:B} for $B$ with $N=9$, $n=5$ and Figure \ref{fig:C} for $C^{\prime }$
with $N=9,$ $k=5$. Also these configurations are in one-to-one
correspondence with certain maximal set-valued tableaux (as defined
previously) and are respectively mapped onto semistandard tableaux of shape $%
\lambda ^{\vee }$ and $(\lambda ^{\vee })^{\prime }$ when taking the
smallest entry in each box. The bijections are now as follows:

\emph{Vicious walkers}. Starting now from the top, place for each
vertex labelled with a bullet in lattice row $i$ a box labelled with $i$ in
the $j$th row of the Young tableau where $j$ is the total number of paths
crossing the row to the left of the vertex. For a vertex with a square in
row $i$ place an additional entry $i$ into an already existing box. The
resulting tableau will now have shape $\lambda ^{\vee }$. 
\begin{figure}[tbp]
\begin{equation*}
\includegraphics[scale=0.5]{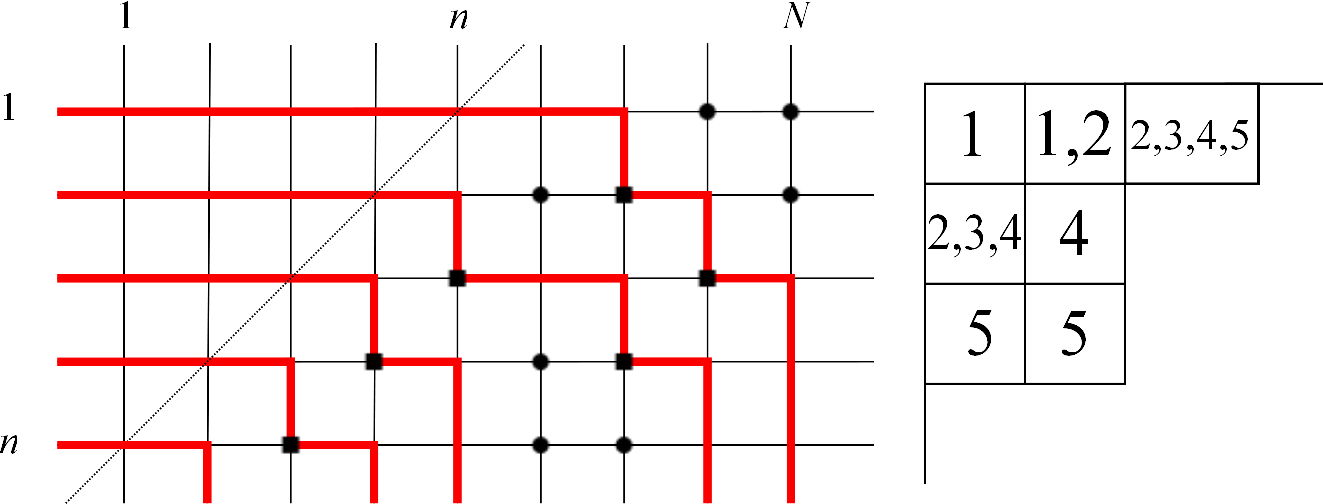}
\end{equation*}%
\caption{A lattice configuration corresponding to the $B$-operator. The
vertex configurations above the dotted line are ``frozen'', i.e. there is no
other choice possible which would yield a nonzero weight.}
\label{fig:B}
\end{figure}

\emph{Osculating walkers}. Consider the leftmost path and write in the
first column (counting from left to right) of the Young diagram of $(\lambda
^{\vee })^{\prime }$ the lattice row numbers where a vertex with a bullet
occurs. If there is a vertex with a square in row $i$ place an $i$ into the
last added box in the same column. Then do the same for the next path
writing the lattice row numbers now in the second column etc. If there are
no vertices with a bullet leave the column empty.%
\begin{figure}[tbp]
\begin{equation*}
\includegraphics[scale=0.5]{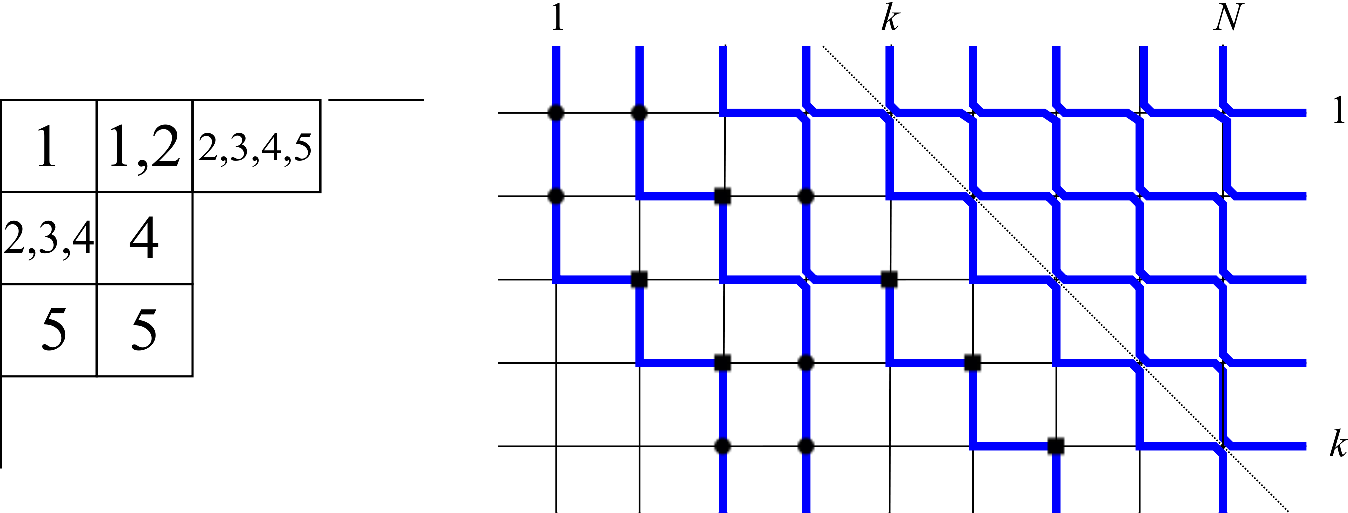}
\end{equation*}%
\caption{A lattice configuration corresponding to the $C^{\prime}$-operator.
The vertex configuration above the dotted line are ``frozen''.}
\label{fig:C}
\end{figure}

Let $\mathcal{C}$ be a lattice configuration of the $B$-operator ($C^{\prime }$%
-operator) and denote by $T_{\mathcal{C}}$ the corresponding semistandard
tableau. If we multiply the matrix element $\langle\tilde v_\lambda |B(y_{n}|t)\cdots
B(y_{1}|t)v_{0\cdots0}\rangle $ with $\Pi (t_{\lambda })/\Pi (y)$ then according to
Figure \ref{fig:5vmodels} each vertex with a bullet contributes a factor $%
y_{i}\ominus t_{j}^{\prime }$ and each vertex with a square a factor $%
(1+\beta y_{i}\ominus t_{j}^{\prime })$ in the vicious walker case. In the
osculating walker case we divide $\langle\tilde v_\lambda |C^{\prime
}(z_{k}^{-1}|t)\cdots C^{\prime }(z_{1}^{-1}|t)|N\rangle $ with $\Pi (z)\Pi
(t_{\lambda ^{\ast }})$ to obtain respectively the factors $z_{i}\oplus
t_{j} $ and $(1+\beta z_{i}\oplus t_{j})$. Here$\ i,j$ are the lattice row
and column numbers where the vertex occurs. As before this implies the
following summation identities for the lattice weights of configuration $\mathcal{C}$ for the $B$-operator,
\begin{equation*}
\frac{\Pi (t_{\lambda })}{\Pi (y)}\operatorname{wt}(\mathcal{C})=\sum_{\mathcal{%
T}}\beta ^{|\mathcal{T}|-|\lambda ^{\vee }|}\prod_{\substack{ \langle
i,j\rangle \in \lambda ^{\vee }  \\ r\in \mathcal{T}(i,j)}}y_{r}\ominus
t_{r+j-i}^{\prime }
\end{equation*}
and a configuration $\mathcal{C\rq{}}$ for the $C\rq{}$-operator,
\begin{equation*}
\frac{\operatorname{wt}(\mathcal{C}\rq{})}{\Pi (z)\Pi (t_{\lambda ^{\ast
}})}=\sum_{\mathcal{T}}\beta ^{|\mathcal{T}|-|\lambda ^{\vee }|}\prod 
_{\substack{ \langle i,j\rangle \in \lambda ^{\ast }  \\ r\in \mathcal{T}%
(i,j) }}z_{r}\oplus t_{r+j-i},
\end{equation*}%
where the sums run over all set-valued tableaux $\mathcal{T}$ of shape $%
\lambda ^{\vee }$ ($\lambda ^{\ast }=(\lambda ^{\vee })^{\prime }$) which
obey the condition that $\min \mathcal{T}(i,j)=T_{\gamma }(i,j)$. The final
step then uses again that the map $\operatorname{SVT}(\lambda ^{\vee
})\twoheadrightarrow \operatorname{SST}(\lambda ^{\vee })$ which assigns to each
set-valued tableau $\mathcal{T}$ the SST $T:\lambda ^{\vee }\rightarrow
\lbrack n]$ with $T(i,j)=\min \mathcal{T}(i,j)$ is a surjection. Thus, it
follows that 
\[
\langle\tilde v_\lambda |B(y_{n}|t)\cdots B(y_{1}|t)v_{0\cdots 0}\rangle =\frac{\Pi (y)}{\Pi
(t_{\lambda })}\sum_{\mathcal{T}\in \operatorname{SVT}(\lambda ^{\vee })}\beta ^{|%
\mathcal{T}|-|\lambda ^{\vee }|}\prod_{\substack{ \langle i,j\rangle \in
\lambda ^{\vee }  \\ r\in \mathcal{T}(i,j)}}y_{r}\ominus t_{r+j-i}^{\prime }
\]
and
\begin{multline*}
\langle\tilde v_\lambda |C^{\prime }(z_{k}^{-1}|t)\cdots C^{\prime
}(z_{1}^{-1}|t)v_{1\cdots 1}\rangle =\\
\Pi (z)\Pi (t_{\lambda ^{\ast }})\sum_{\mathcal{T}%
\in \operatorname{SVT}(\lambda ^{\ast })}\beta ^{|\mathcal{T}|-|\lambda ^{\vee
}|}\prod_{\substack{ \langle i,j\rangle \in \lambda ^{\ast }  \\ r\in 
\mathcal{T}(i,j)}}z_{r}\oplus t_{r+j-i}
\end{multline*}%
which are the expressions for the stated factorial Grothendieck polynomials in our assertion 
in terms of set-valued tableaux.

The identities for the dual Bethe vectors (\ref{leftBethe}), (\ref%
{leftBethe'}) are obtained by a very similar argument noting from the
definition (\ref{dual_mom}) that the matrix elements of the transposed
operators are obtained by reversing binary strings and swapping $%
t_j\leftrightarrow t^{\prime }_j$.
\end{proof}

We are now in the position to prove a generalised Cauchy identity for
factorial Grothendieck polynomials; compare with \cite[Thm 5.3]{MotegiSakai}
and \cite[Thm 9]{LN} for the non-factorial case which we obtain as a special
case.

\begin{corollary}
Setting $e(x,y)=\langle\tilde v_{0\cdots 0}|C(x_{1}|t)\cdots C(x_{n}|t)B(y_{n}|t)\cdots
B(y_{1}|t)v_{0\cdots 0}\rangle $ we have%
\begin{eqnarray} \label{norm}
e(x,y)&=&\Pi (y)\sum_{\lambda \subset (k^{n})}\frac{G_{\lambda }(x|\ominus
t)G_{\lambda ^{\vee }}(y|\ominus t^{\prime })}{\Pi (t_{\lambda })} \\
&=&\frac{1}{\Pi (x)}\sum_{w}w\left( \Pi (x)\frac{\prod_{i=1}^{n}%
\prod_{j=1}^{N}x_{i}\ominus t_{j}}{\prod_{1\leq i,j\leq n}x_{i}\ominus y_{j}}%
\right),\notag 
\end{eqnarray}%
where - as in Lemma \ref{lem:yba2} - the sum runs over the minimal length
coset representatives $w$ of $\mathbb{S}_{2n}/\mathbb{S}_{n}\times \mathbb{S}%
_{n}$ which act on $(x,y)$ in the obvious manner.
\end{corollary}

\begin{proof}
Noting that 
\begin{equation}
A(x)v_{0\cdots 0} =\left( \prod\limits_{j=1}^{N}x\ominus t_{j}\right) v_{0\cdots 0}
\qquad \text{and\qquad }D(x)v_{0\cdots 0} =v_{0\cdots 0} \label{A0_and_D0}
\end{equation}%
the assertion is immediate from Lemma \ref{lem:yba2} and the formulae (\ref%
{Caction}), (\ref{rightBethe}).
\end{proof}

Note that the limit $\lim_{x_{i}\rightarrow y_{i}}e(x,y)$ is well-defined as
can be seen from the definition of $e(x,y)$ as the matrix element 
$\langle\tilde v_{0\cdots 0}|C(x_{1}|t)\cdots C(x_{n}|t)B(y_{n}|t)\cdots
B(y_{1}|t)v_{0\cdots 0}\rangle $ and (\ref{norm}%
). The poles in the last line of Equation (\ref{norm}) cancel against the zeroes
in the numerator as $x_{i}\rightarrow y_{i}$ after the sum over the $w$'s is
taken.

\begin{corollary}
\label{cor:cauchy0}Setting $y=t_{\mu }$ in (\ref{norm}) we obtain%
\begin{equation}
\prod_{i=1}^{n}\prod_{j\in I_{\mu ^{\ast }}}(x_{i}\ominus
t_{j})=\sum_{\lambda \subset (k^{n})}\frac{\Pi (t_{\mu })}{\Pi (t_{\lambda })%
}~G_{\lambda ^{\vee }}(t_{\mu }|\ominus t^{\prime })G_{\lambda }(x|\ominus
t)\;.
\end{equation}%
This proves in particular (\ref{Cauchy0}) and, thus, we obtain after setting
also $x=t_{\mu }$,%
\begin{equation}
e(t_{\mu },t_{\mu })=\prod_{i\in I_{\mu }}\prod_{j\in I_{\mu ^{\ast
}}}(t_{i}\ominus t_{j})\;.  \label{norm0}
\end{equation}
\end{corollary}

\begin{proof}
Specialising $y=t_{\mu }$ in (\ref{norm}) one easily sees that only the term
with $w$ being the identity survives in the last sum.
\end{proof}

\subsection{The Bethe ansatz equations}

We call the Bethe vectors (\ref{Bethev}), (\ref{Bethev'}) \textquotedblleft
on-shell\textquotedblright\ if the variables $y=(y_{1},\ldots ,y_{n})$
are \emph{pairwise distinct} solutions to the following set of equations with $\Pi (y)$ defined in (\ref{Pi}),%
\begin{equation}
(-1)^{n}\frac{\Pi (y)}{\left( 1+\beta y_{i}\right) ^{n}}\prod_{j=1}^{N}y_{i}%
\ominus t_{j}+q=0,\qquad i=1,\ldots ,n\;.  \label{BAE}
\end{equation}%
This set of equations is known as {\em Bethe ansatz equations} in the literature on integrable systems. Note that these equations are interdependent. We define a second set of equations for the variables $z=(z_{1},\ldots
,z_{k})$ in (\ref{Bethev'}),%
\begin{equation}
(-1)^{k}\frac{\Pi (z)}{\left( 1+\beta z_{i}\right) ^{k}}\prod_{j=1}^{N}z_{i}%
\oplus t_{j}+q=0,\qquad i=1,\ldots ,k\;.  \label{BAE'}
\end{equation}%
The origin of these equations is Lemma \ref{lem:yba1} from which we deduce
that if (\ref{BAE}) holds the Bethe vector (\ref{Bethev}) is an eigenvector
of $H(x_i|t)=A(x_i|t)+qD(x_i|t)$; see below. By the same argument one shows that (\ref{Bethev'})
is an eigenvector of $E(x_i|t)=A^{\prime }(x_i|t)+qD^{\prime }(x_i|t)$ if (\ref{BAE'}) hold.
Obviously, one set of equations transforms into the other under the
substitution $t=(t_{1},\ldots ,t_{N})\rightarrow \ominus t^{\prime
}=(\ominus t_{N},\ldots ,\ominus t_{1})$ and exchanging $n$ with $k$. This
substitution is related to level-rank duality (\ref{levelrankmom}) which we
will use below to relate the Bethe vectors (\ref{Bethev}) with the vectors (%
\ref{Bethev'}). We shall therefore focus on the equations (\ref{BAE}) only.

\begin{lemma}
\label{lem:BAE}The set of equations (\ref{BAE}) has $\binom{N}{n}$ \emph{%
pairwise distinct }solutions, called {\em Bethe roots}, 
\begin{equation}
y_{\lambda }=(y_{\lambda _{n}+1},\ldots y_{\lambda _{2}+n-1},y_{\lambda
_{1}+n})\in \mathbb{Z}[\![q]\!]\otimes \mathbb{Z}(t_{1},\ldots ,t_{N}),
\label{ylambda}
\end{equation}%
where $\lambda \subset (k^{n})$ and up to first order in $q$ we have%
\begin{equation}
y_{i}=t_{i}+q~(-1)^{n-1}\frac{(1+\beta t_{i})^{n+1}}{\Pi (t_{\lambda
})\prod_{j\neq i}t_{i}\ominus t_{j}}+O(q^{2})\;.  \label{Brootexpansion}
\end{equation}
\end{lemma}

\begin{proof}
Make the ansatz $y_{i}=y_{i}^{(0)}+y_{i}^{(1)}q+y_{i}^{(2)}q^{2}+\cdots $
and set $q=0$ in (\ref{BAE}). Since all equations need to be solved simultaneously, they are 
interdependent, setting $%
y_{i^{\prime }}^{(0)}=-\beta ^{-1}$ with $i^{\prime }\neq i$ in the factor
in front of the product in (\ref{BAE}) is not a valid solution, since it
would imply a singular term in the equations with $i$ replaced by $i^{\prime
}$. Therefore, the only possible solution is $y_{i}^{(0)}=t_{i}$. Here the
labelling in terms of the index $i$ is a matter of choice but it will prove
convenient later on. Differentiating the equations with respect to $q$,%
\begin{equation*}
\frac{d}{dq}\left( \frac{(-1)^{n-1}\Pi (y)}{\left( 1+\beta y_{i}\right) ^{n}}%
\right) \prod_{j=1}^{N}y_{i}\ominus t_{j}+\frac{(-1)^{n-1}\Pi (y)}{\left(
1+\beta y_{i}\right) ^{n}}\frac{d}{dq}\prod_{j=1}^{N}y_{i}\ominus t_{j}=1,
\end{equation*}%
and setting $q=0$ afterwards we find%
\begin{equation*}
\frac{d}{dq}\left. \prod_{j=1}^{N}y_{i}\ominus t_{j}\right\vert _{q=0}=\frac{%
y_{i}^{(1)}}{1-\beta t_{i}}\prod_{j\neq i}t_{i}\ominus t_{j}
\end{equation*}%
and the formula (\ref{Brootexpansion}) follows. Continuing in the same
manner by taking the $r$th derivative we see that the coefficient $\Pi
(t_{\lambda })/(1+\beta t_{i})^{n}$ in front of the term 
\begin{equation*}
\frac{d^{r}}{dq^{r}}\left. \prod_{j=1}^{N}y_{i}\ominus t_{j}\right\vert
_{q=0}
\end{equation*}%
is alway nonzero and, hence, that the equations yield a rational solution in
the $t_{j}$'s for any $y_{i}^{(r)}$.
\end{proof}

\begin{lemma}
Given a partition $\lambda =(\lambda _{1},\ldots ,\lambda _{n})$ with $%
\lambda _{1}>k$ and a solution $y=y_{\mu }$ of the Bethe ansatz equations (%
\ref{BAE}), one has the identity%
\begin{multline}
G_{\lambda }(y|\ominus t)=  \label{Greduction} \\
q\sum_{r=0}^{\lambda _{1}-1-k}h_{\lambda _{1}-1-k-r}(t_{1},\ldots
,t_{r+1})G_{(\lambda _{2}-1,\ldots ,\lambda _{n}-1,r)}(y|\ominus
t)\prod_{i=1}^{r}(1+\beta t_{i}),
\end{multline}%
where the $h_{r}$'s denote the complete symmetric functions and the
factorial Grothendieck polynomial on the right hand side is defined via (\ref%
{Gstraight}).
\end{lemma}

\begin{proof}
Recall the determinant formula (\ref{facGdet}) for factorial Grothendieck
polynomials. Writing out the determinant in the numerator we find%
\begin{multline*}
a_{\lambda }=\left\vert 
\begin{array}{ccc}
(y_{1}|\ominus t)^{n+\lambda _{1}-1} & \cdots & (y_{n}|\ominus t)^{n+\lambda
_{1}-1} \\ 
(y_{1}|\ominus t)^{n+\lambda _{2}-2}(1+\beta y_{1}) &  & (y_{n}|\ominus
t)^{n+\lambda _{2}-2}(1+\beta y_{n}) \\ 
\vdots &  & \vdots \\ 
(y_{1}|\ominus t)^{\lambda _{n}}(1+\beta y_{1})^{n-1} & \cdots & 
(y_{n}|\ominus t)^{\lambda _{n}}(1+\beta y_{n})^{n-1}%
\end{array}%
\right\vert \\
=\frac{q}{\Pi (y)}\left\vert 
\begin{array}{ccc}
(y_{1}|\ominus t)^{n+\lambda _{2}-2}(1+\beta y_{1}) & \cdots & 
(y_{n}|\ominus t)^{n+\lambda _{2}-2}(1+\beta y_{n}) \\ 
\vdots &  & \vdots \\ 
(y_{1}|\ominus t)^{\lambda _{n}}(1+\beta y_{1})^{n-1} &  & (y_{n}|\ominus
t)^{\lambda _{n}}(1+\beta y_{n})^{n-1} \\ 
y_{1}^{\lambda _{1}-1-k}(1+\beta y_{1})^{n} & \cdots & y_{n}^{\lambda
_{1}-1-k}(1+\beta y_{n})^{n}%
\end{array}%
\right\vert
\end{multline*}%
Here we have made use of (\ref{BAE}), exchanged the first row with the last
row in the determinant and used row linearity of the determinant to pull out
the common factor in front. Note that $t_{j}=0$ for $j>N$, whence the powers
in the bottom row are not factorial. To rewrite them as factorial powers we
use the equality%
\begin{equation*}
x^{m}=\sum_{r=0}^{m}(x|\ominus t)^{m-r}h_{r}(t_{1},\ldots
,t_{m+1-r})\prod_{i=1}^{m-r}(1+\beta t_{i})
\end{equation*}%
which is easily proved via induction using the known recursion relation%
\begin{equation*}
h_{r+1}(t_{1},\ldots ,t_{m+1-r})=h_{r}(t_{1},\ldots
,t_{m+1-r})t_{m+1-r}+h_{r+1}(t_{1},\ldots ,t_{m-r})
\end{equation*}%
of the complete symmetric functions. We leave this step to the reader.

Thus, after employing the above identity and column/row linearity of the
determinant we arrive at%
\begin{equation*}
a_{\lambda }=q\sum_{r=0}^{\lambda _{1}-1-k}a_{(\lambda _{2}-1,\ldots
,\lambda _{n}-1,\lambda _{1}-1-k-r)}h_{r}(t_{1},\ldots ,t_{\lambda
_{1}-k-r})\prod_{i=1}^{\lambda _{1}-1-k-r}(1+\beta t_{i})
\end{equation*}%
which is the asserted identity (\ref{Greduction}) after dividing by the
Vandermonde determinant.
\end{proof}

\begin{theorem}
The on-shell Bethe vectors (\ref{Bethev}), (\ref{Bethev'}) and (\ref%
{leftBethev}), (\ref{leftBethev'}) form respectively right and left
eigenbases of the transfer matrices $H(x_i|t)$ and $E(x_i|t)$ in each subspace $\mathcal{V}%
_{n}^{q}$ with eigenvalue equations%
\begin{equation}
H(x_i|t)|y_{\mu }\rangle =\left( \frac{\prod\limits_{j=1}^{N}x_i\ominus
t_{j}+(-1)^{n}q\prod\limits_{j\in I_{\mu }}\left( 1+\beta x_i\ominus
y_{j}\right) }{\prod\limits_{j\in I_{\mu }}x_i\ominus y_{j}}\right) ~|y_{\mu
}\rangle  \label{specH}
\end{equation}%
and%
\begin{equation}
E(x_i|t)|z_{\mu }\rangle =\left( \frac{\prod\limits_{j=1}^{N}x\oplus
t_{j}+(-1)^{n}q\prod\limits_{j\in I_{\mu ^{\ast }}}\left( 1+\beta x_i\ominus
z_{j}\right) }{\prod\limits_{j\in I_{\mu ^{\ast }}}x_i\ominus z_{j}}\right)
~|z_{\mu }\rangle \;.  \label{specE}
\end{equation}
\end{theorem}

\begin{proof}
Here we have used the commutation relations of the Yang-Baxter algebra as per
Lemma \ref{lem:yba1} and (\ref{A0_and_D0}) from which we deduce that if (\ref%
{BAE}) holds the Bethe vector (\ref{Bethev}) is an eigenvector of $H(x_i|t)=A(x_i|t)+qD(x_i|t)$.
The computation follows along the same lines for (\ref{Bethev'}) and the
left eigenvectors (\ref{leftBethev}, \ref{leftBethev'}).

One deduces that the eigenvalues must separate points and, hence, $\langle
y_{\lambda }|y_{\mu }\rangle =\langle z_{\lambda }|z_{\mu }\rangle =0$ for $%
\lambda \neq \mu $. That these eigenvectors form a basis then follows from
the fact that there exist $\dim \mathcal{V}_{n}=\binom{N}{n}$ solutions to
the equations (\ref{BAE}); see Lemma \ref{lem:BAE}.
\end{proof}

Note that the above formulae simplify if $q=0$. Then the on-shell Bethe
vectors with $y_{\mu }=t_{\mu }$ are given by%
\begin{equation}
|t_{\mu }\rangle =\sum_{\lambda \subset (k^{n})}G_{\lambda ^{\vee }}(t_{\mu
}|\ominus t^{\prime })\frac{\Pi (t_{\mu })}{\Pi (t_{\lambda })}~v_{\lambda }
\label{Bethev0}
\end{equation}%
and form an eigenbasis of the transfer matrices with eigenvalues,%
\begin{eqnarray}
H(x_i|t)|t_{\mu }\rangle &=&\left( \tprod\limits_{j\in I_{\mu ^{\ast
}}}x_i\ominus t_{j}\right) |t_{\mu }\rangle  \label{specH0} \\
E(x_i|t)|t_{\mu }\rangle &=&\left( \tprod\limits_{j\in I_{\mu }}x_i\oplus
t_{j}\right) |t_{\mu }\rangle  \label{specE0}
\end{eqnarray}%
As we will discuss below this special case describes generalised equivariant
cohomology theory, $h_{n}^{\ast }=qh_{n}^{\ast }/\langle q\rangle $ and we
show below that $h_{n}^{\ast }/\langle \beta +1\rangle \cong K_{\mathbb{T}}(%
\operatorname{Gr}_{n,N})$.

\begin{proposition}
\label{prop:Gduality}The eigenvectors of $H(x_i|t)$ and $E(x_i|t)$ coincide under the
substitution $z_{\lambda ^{\prime }}=\ominus y_{\lambda ^{\vee }}$ and,
thus, we have the equality%
\begin{equation}
G_{\lambda ^{\prime }}(\ominus y_{\mu ^{\ast }}|t)=G_{\lambda }(y_{\mu
}|\ominus t^{\prime })  \label{Gduality}
\end{equation}%
for each solution $y_{\mu }$ of (\ref{BAE}). In particular, for $q=0$ we
have $G_{\lambda ^{\prime }}(\ominus t_{\mu ^{\ast }}|t)=G_{\lambda }(t_{\mu
}|\ominus t^{\prime })$.
\end{proposition}

\begin{proof}
Using the identity (\ref{levelrankmom}) when acting on the Bethe vectors we
find%
\begin{eqnarray*}
\Theta H(x_i|t)|y_{\lambda }\rangle &=&E(x_i|\ominus t^{\prime })\Theta
|y_{\lambda }\rangle \\
\Theta E(x_i|t)|z_{\lambda }\rangle &=&H(x_i|\ominus t^{\prime })\Theta
|z_{\lambda }\rangle
\end{eqnarray*}%
and%
\begin{eqnarray*}
\Theta |y_{\lambda }\rangle &=&C^{\prime }(y_{1}|\ominus t^{\prime })\cdots
C^{\prime }(y_{n}|\ominus t^{\prime })|N\rangle \\
\Theta |z_{\lambda }\rangle &=&B(z_{1}|\ominus t^{\prime })\cdots
B(z_{k}|\ominus t^{\prime })|0\rangle
\end{eqnarray*}%
These identities together with the expansion (\ref{Brootexpansion}) allows
us to identify $z_{\lambda ^{\prime }}=\ominus y_{\lambda ^{\vee }}$.
\end{proof}

\begin{corollary}
The eigenvalue equation (\ref{specE}) of the $E$-transfer matrix simplifies
in the Bethe roots (\ref{ylambda}) to 
\begin{equation}
E(x_i|t)|y_{\mu }\rangle =\prod\limits_{j\in I_{\mu }}\left( x_i\oplus
y_{j}\right) ~|y_{\mu }\rangle \;.  \label{SpecE}
\end{equation}
\end{corollary}

\begin{proof}
Replacing $x_i\rightarrow t_{j}$ the functional equation (\ref{func_eqn}) together with (%
\ref{specH}) implies%
\begin{equation*}
q|y_{\mu }\rangle =E(\ominus t_{j}|t)H(t_{j}|t)|y_{\mu }\rangle =\frac{q}{%
\prod\limits_{i=1}^{n}\left( t_{j}\ominus y_{i}\right) }E(\ominus
t_{j}|t)|y_{\mu }\rangle
\end{equation*}%
for all $j=1,2,\ldots ,N$. Since the $t_{j}$'s are arbitrary variables and
the Bethe vectors form an eigenbasis the assertion follows.
\end{proof}

Since the Bethe vectors (\ref{Bethev}) and (\ref{leftBethev}) form each an
eigenbasis they give rise to a resolution of the identity $\boldsymbol{1}%
=\sum_{\alpha \in (n,k)}|y_{\alpha }\rangle \langle y_{\alpha }|$ where $%
|y_{\alpha }\rangle \langle y_{\alpha }|$ denotes the orthogonal projector
onto the eigenspace spanned by $|y_{\alpha }\rangle $. This elementary fact
of linear algebra translates into the following non-trivial identities for
factorial Grothendieck polynomials evaluated at solutions of the Bethe
ansatz equations (\ref{BAE}).

\begin{corollary}[orthogonality \& completeness]
For all $\lambda ,\mu \subset (k^{n})$ we have the identities%
\begin{equation}
\sum_{\alpha \subset (k^{n})}\frac{\Pi (y_{\lambda })}{\Pi (t_{\alpha })}%
\frac{G_{\alpha ^{\vee }}(y_{\lambda }|\ominus t^{\prime })G_{\alpha
}(y_{\mu }|\ominus t)}{e(y_{\lambda },y_{\lambda })}=\delta _{\lambda \mu }
\label{Cauchy2}
\end{equation}%
and%
\begin{equation}
\sum_{\alpha \subset (k^{n})}\frac{\Pi (y_{\alpha })}{\Pi (t_{\lambda })}%
\frac{G_{\lambda ^{\vee }}(y_{\alpha }|\ominus t^{\prime })G_{\mu
}(y_{\alpha }|\ominus t)}{e(y_{\alpha },y_{\alpha })}=\delta _{\lambda \mu
}\;,  \label{res of 1}
\end{equation}%
where $\delta _{\lambda \mu }$ denotes the Kronecker delta with $\delta
_{\lambda \mu }=1$ if $\lambda =\mu $ and $0$ otherwise.
\end{corollary}

\subsection{Generalised matrix algebras and Frobenius structures}

Following the suggested construction in \cite[Section 7]{Korffproc} we now
introduce a ring structure on each $\mathcal{V}_{n}^{q}=\mathbb{Z}%
[\![q]\!]\otimes \mathcal{V}_{n}$ by interpreting the on-shell Bethe vectors
(\ref{Bethev}) as central orthogonal idempotents of a semisimple algebra:
for each $n=0,1,\ldots ,N$ define $qh_{n}^{\ast }=(\mathcal{V}%
_{n}^{q},\circledast )$ by fixing the product $\circledast $ as follows,%
\begin{equation}
Y_{\lambda }\circledast Y_{\mu }=\delta _{\lambda \mu }Y_{\mu }~,\qquad
Y_{\lambda }=e(y_{\lambda },y_{\lambda })^{-1}|y_{\lambda }\rangle \;,
\label{idempotent_def}
\end{equation}%
where $e(y_{\lambda },y_{\lambda })$ is the matrix element defined in (\ref%
{norm}). Note that $e(y_{\lambda },y_{\lambda })$ is a power series in $q$
with nonzero constant term (\ref{norm0}) according to (\ref{Brootexpansion}%
). The unit element is given by 
\begin{equation}
v_{\emptyset }=\sum_{\lambda \subset (k^{n})}Y_{\lambda }\;.  \label{unit}
\end{equation}%
This determines $qh_{n}^{\ast }$ via its Peirce decomposition \cite{Peirce}.
We turn $qh_{n}^{\ast }$ into a (generalised) Frobenius algebra by introducing in addition
the following symmetric bilinear form $\mathcal{V}_{n}^{q}\times \mathcal{V}%
_{n}^{q}\rightarrow \mathcal{R}(\mathbb{T},q)$,%
\begin{equation}
(Y_{\lambda },Y_{\mu })=e(y_{\lambda },y_{\lambda })^{-1}\delta _{\lambda
\mu }\;.  \label{bilinear_form}
\end{equation}%
By definition this bilinear form is invariant with respect to the product (%
\ref{idempotent_def}) and non-degenerate, since the Bethe vectors form a
basis.

\subsection{A residue formula for the structure constants}

We now describe the resulting generalised matrix algebra $qh_{n}^{\ast }$ in
the spin basis $\{v_{\lambda }\}_{\lambda \subset (k^{n})}$. Introduce a
family of operators $\{\boldsymbol{G}_{\lambda }\}_{\lambda \subset
(k^{n})}\subset \operatorname{End}\mathcal{V}_{n}^{q}$ via the following
eigenvalue equation 
\begin{equation}
\boldsymbol{G}_{\lambda }Y_{\mu }=G_{\lambda }(y_{\mu }|\ominus t)Y_{\mu }\;.
\label{BigG}
\end{equation}%
This \emph{defines} the operators $\boldsymbol{G}_{\lambda }$, since the
Bethe vectors form an eigenbasis and the eigenvalues separate points. Recall
from Section \ref{sec:grothendieck} that the factorial Grothendieck
polynomials form a basis in the ring of symmetric polynomials \cite[Thm 4.6]{McNamara}. Below we give an
explicit, basis independent construction of $\boldsymbol{G}_{\lambda }$ in
terms of the transfer matrix $H(x_i|t)$.

\begin{corollary}
In the spin basis (\ref{spin basis}) the product (\ref{idempotent_def}) is
given by 
\begin{equation}
v_{\lambda }\circledast v_{\mu }=\boldsymbol{G}_{\lambda }v_{\mu }=\sum_{\nu
\subset (k^{n})}C_{\lambda \mu }^{\nu }(t,q)v_{\nu },  \label{combproduct}
\end{equation}%
where the structure constants $C_{\lambda \mu }^{\nu }(t,q)=\langle\tilde v_\nu |%
\boldsymbol{G}_{\lambda }v_\mu \rangle $ are obtained in terms of the
Bethe roots (\ref{ylambda}) via the residue formula%
\begin{equation}
C_{\lambda \mu }^{\nu }(t,q)=\sum_{\alpha \subset (k^{n})}\frac{\Pi
(y_{\alpha })}{\Pi (t_{\nu })}\frac{G_{\lambda }(y_{\alpha }|\ominus
t)G_{\mu }(y_{\alpha }|\ominus t)G_{\nu ^{\ast }}(\ominus y_{\alpha ^{\ast
}}|t)}{e(y_{\alpha },y_{\alpha })}\;.  \label{Verlinde}
\end{equation}%
Similarly, the bilinear form (\ref{bilinear_form}) can be expressed as%
\begin{equation}
(v_{\lambda },v_{\mu })=\sum_{\alpha \subset (k^{n})}\frac{G_{\lambda
}(y_{\alpha }|\ominus t)G_{\mu }(y_{\alpha }|\ominus t)}{e(y_{\alpha
},y_{\alpha })}\;.  \label{bilinear_form_spin}
\end{equation}
\end{corollary}

\begin{remark}
Our residue formula (\ref{Verlinde}) is a generalisation of the
Bertram-Vafa-Intriligator formula for Gromov-Witten invariants; see \cite[Sec. 5]{Bertram} and references therein. It holds
also true for $q=0$, where the Bethe roots are explicitly known, $%
y_{i}=t_{i} $,%
\begin{equation}
c_{\lambda \mu }^{\nu }(t)=C_{\lambda \mu }^{\nu }(t,0)=\sum_{\alpha \subset
(k^{n})}\frac{\Pi (t_{\alpha })}{\Pi (t_{\nu })}\frac{G_{\lambda }(t_{\alpha
}|\ominus t)G_{\mu }(t_{\alpha }|\ominus t)G_{\nu ^{\ast }}(\ominus
t_{\alpha ^{\ast }}|t)}{\prod_{i\in I_{\alpha },j\in I_{\alpha ^{\ast
}}}t_{i}\ominus t_{j}}\;.  \label{Verlinde0}
\end{equation}%
The bilinear form (\ref{bilinear_form_spin}) for $q=0$ reads%
\begin{equation}
(v_{\lambda },v_{\mu })=\sum_{\alpha \subset (k^{n})}\frac{G_{\lambda
}(t_{\alpha }|\ominus t)G_{\mu }(t_{\alpha }|\ominus t)}{\prod_{i\in
I_{\alpha },j\in I_{\alpha ^{\ast }}}t_{i}\ominus t_{j}}\;.
\label{bilinear_form0}
\end{equation}
\end{remark}

\begin{proof}
According to (\ref{rightBethe}) and (\ref{res of 1}) we have the inverse
basis transformation%
\begin{equation}
v_{\lambda }=\sum_{\mu \subset (k^{n})}G_{\lambda }(y_{\mu }|\ominus
t)Y_{\mu }\;.  \label{Bethe2spin}
\end{equation}%
which allows us to compute%
\begin{eqnarray*}
v_{\lambda }\circledast v_{\mu } &=&\sum_{\rho ,\sigma }G_{\lambda }(y_{\rho
}|\ominus t)G_{\mu }(y_{\sigma }|\ominus t)Y_{\rho }\circledast Y_{\sigma }
\\
&=&\sum_{\rho }G_{\lambda }(y_{\rho }|\ominus t)G_{\mu }(y_{\rho }|\ominus
t)Y_{\rho }=\boldsymbol{G}_{\lambda }v_{\mu }=\boldsymbol{G}_{\mu
}v_{\lambda }\;.
\end{eqnarray*}%
This proves the first assertion. Continuing the computation from the second
line employing (\ref{rightBethe}) we arrive at (\ref{Verlinde}).

The expression (\ref{bilinear_form_spin}) is also an immediate consequence
of (\ref{Bethe2spin}). Insert the latter and use the definition (\ref%
{bilinear_form}) to find the asserted identity (\ref{bilinear_form_spin}).
\end{proof}

As is to be expected from our previous results (\ref{LRduality_L}) and (\ref%
{levelrankmom}), the rings related by exchanging the dimension $n$ with the
codimension $k$ of the hyperplanes in the Grassmannian are closely related.

\begin{corollary}[level-rank duality]
The involution $qh_{n}^{\ast }\rightarrow qh_{k}^{\ast }$ given by $%
f(t,q)v_{\lambda }\mapsto f(\ominus t^{\prime },q)v_{\lambda ^{\prime }}$ is
a ring isomorphism over $\mathcal{R}\otimes \mathbb{Z}[\![q]\!]$. That is,%
\begin{equation}
C_{\mu \nu }^{\lambda }(t,q)=C_{\mu ^{\prime }\nu ^{\prime }}^{\lambda
^{\prime }}(\ominus t^{\prime },q)\;.  \label{levelrank}
\end{equation}
\end{corollary}

\begin{proof}
First we note that (\ref{G1}) and (\ref{Gduality}) imply the identity%
\begin{eqnarray*}
\frac{\Pi (y_{\lambda })}{\Pi (t_{\mu })} &=&\frac{\Pi (t_{\emptyset })}{\Pi
(t_{\mu })}~\left( 1+\beta G_{1}(y_{\lambda }|\ominus t)\right) \\
&=&\frac{\Pi (\ominus t_{\emptyset }^{\prime })}{\Pi (\ominus t_{\mu
^{\prime }}^{\prime })}~\left( 1+\beta G_{1}(\ominus y_{\lambda ^{\ast
}}|t^{\prime })\right) =\frac{\Pi (\ominus y_{\lambda ^{\ast }})}{\Pi
(\ominus t_{\mu ^{\prime }}^{\prime })}\;.
\end{eqnarray*}%
Note further that according to (\ref{Brootexpansion}) the $k$-tuple $\ominus
y_{\lambda ^{\ast }}$ is obtained from solutions $y_{i}$ by replacing $%
t=(t_{1},\ldots ,t_{N})$ with $\ominus t^{\prime }=(\ominus t_{N},\ldots
,\ominus t_{1})$, i.e. the constant terms of the components of the solution $%
\ominus y_{\lambda ^{\ast }}$ are $\ominus t_{\lambda ^{\prime }}^{\prime }$
which identifies the solution uniquely. Using the residue formula (\ref%
{Verlinde}) and (\ref{Gduality}) we compute%
\begin{multline*}
C_{\mu \nu }^{\lambda }(t,q)=\sum_{\alpha \subset (k^{n})}\frac{\Pi
(y_{\alpha })}{\Pi (t_{\nu })}\frac{G_{\lambda }(y_{\alpha }|\ominus
t)G_{\mu }(y_{\alpha }|\ominus t)G_{\nu ^{\vee }}(y_{\alpha }|\ominus
t^{\prime })}{e(y_{\alpha },y_{\alpha })}= \\
\sum_{\alpha }\frac{\Pi (\ominus y_{\alpha ^{\ast }})}{\Pi (\ominus t_{\nu
^{\prime }}^{\prime })}\frac{G_{\lambda ^{\prime }}(\ominus y_{\alpha ^{\ast
}}|t^{\prime })G_{\mu ^{\prime }}(\ominus y_{\alpha ^{\ast }}|t^{\prime
})G_{\nu ^{\ast }}(\ominus y_{\alpha ^{\ast }}|t)}{e(y_{\alpha },y_{\alpha })%
}=C_{\mu ^{\prime }\nu ^{\prime }}^{\lambda ^{\prime }}(\ominus t^{\prime
},q),
\end{multline*}%
where in the last step we have used the definition (\ref{norm}) to show that%
\begin{eqnarray*}
e(y_{\alpha },y_{\alpha }) &=&\sum_{\lambda \subset (k^{n})}\frac{\Pi
(y_{\alpha })}{\Pi (t_{\lambda })}G_{\lambda }(y_{\alpha }|\ominus
t)G_{\lambda ^{\vee }}(y_{\alpha }|\ominus t^{\prime }) \\
&=&\sum_{\lambda \subset (k^{n})}\frac{\Pi (\ominus y_{\alpha ^{\ast }})}{%
\Pi (\ominus t_{\lambda ^{\prime }}^{\prime })}G_{\lambda ^{\prime
}}(\ominus y_{\alpha ^{\ast }}|t^{\prime })G_{\lambda ^{\ast }}(\ominus
y_{\alpha ^{\ast }}|t)=e(\ominus y_{\alpha ^{\ast }},\ominus y_{\alpha
^{\ast }})\;.
\end{eqnarray*}
\end{proof}

\subsection{A recurrence formula}

We now return to the result (\ref{H1}) and show that the latter formula
describes the multiplication with the class of the Schubert divisor, i.e.
that (\ref{H1}) describes indeed the equivariant quantum Pieri-Chevalley
rule for the generalised cohomology ring $qh_{n}^{\ast }$.

\begin{corollary}
Let $\lambda =(1,0,\ldots ,0)$ then%
\begin{equation}
\boldsymbol{G}_{1}=H_{1}  \label{G1=H1}
\end{equation}%
and the product $v_{1}\circledast v_{\lambda }=H_{1}v_{\lambda }$ in the
spin basis is given explicitly via (\ref{H1}).
\end{corollary}

\begin{proof}
Employing the functional equation (\ref{func_eqn}) and (\ref{specH}), (\ref%
{SpecE}) we obtain%
\begin{multline*}
\prod_{j=1}^{n}(t_{j}\ominus x)\prod_{j=n+1}^{N}(x\ominus t_{j})(1+\beta
H_{1})Y_{\mu }=(H(x|t)E(\ominus x|t)-q\cdot 1)Y_{\mu } \\
=(-1)^{n}\frac{\Pi (y_{\mu })}{(1+\beta x)^{n}}\prod_{j=1}^{N}(x\ominus
t_{j})~Y_{\mu } \\
=\frac{\Pi (y_{\mu })}{\Pi (t_{\emptyset })}\prod_{j=1}^{n}(t_{j}\ominus
x)\prod_{j=n+1}^{N}(x\ominus t_{j})~Y_{\mu }
\end{multline*}%
Thus, according to (\ref{G1Pieri}), (\ref{BigG}) we have%
\begin{equation*}
(1+\beta H_{1})Y_{\mu }=\frac{\Pi (y_{\mu })}{\Pi (t_{\emptyset })}~Y_{\mu
}=(1+\beta \boldsymbol{G}_{1})Y_{\mu }
\end{equation*}%
and the assertion follows from the fact that the Bethe vectors form a basis.
\end{proof}

Analogous to the case of equivariant (quantum) cohomology one derives from
the quantum Pieri-Chevalley rule (\ref{H1}) the following recurrence
relation for the structure constants.

\begin{corollary}[Recurrence relation]
We have the identity%
\begin{equation}
\left( \Pi (t_{\nu })-\Pi (t_{\lambda })\right) C_{\lambda \mu }^{\nu
}=\sum_{\tilde{\lambda}/d^{\prime }/\lambda }\beta ^{|\tilde{\lambda}%
/d/\lambda |}C_{\tilde{\lambda}\mu }^{\nu }-\sum_{\nu /d^{\prime \prime }/%
\tilde{\nu}}\beta ^{|\nu /d^{\prime \prime }/\tilde{\nu}|}\Pi (t_{\tilde{\nu}%
})C_{\lambda \mu }^{\tilde{\nu}},  \label{recurr}
\end{equation}%
where the sums run over all partitions $\tilde{\lambda}\neq \lambda ,\tilde{%
\nu}\neq \nu $ such that respectively $\tilde{\lambda}/d^{\prime }/\lambda $
and $\nu /d^{\prime \prime }/\nu $ are toric skew-diagrams with $d^{\prime
},d^{\prime \prime }$ either $0$ or $1$ and where each row and column
contains at most one box.
\end{corollary}

\begin{proof}
The derivation follows the same idea as in ordinary (quantum) cohomology;
see e.g. \cite{KnutsonTao}. Since the product $\circledast $ by definition
is associative we have in light of (\ref{G1=H1}) that%
\begin{equation*}
\lbrack (1+\beta H_{1})v_{\lambda }]\circledast v_{\mu }=(1+\beta
H_{1})(v_{\lambda }\circledast v_{\mu })\;.
\end{equation*}%
Applying the Pieri-Chevalley rule (\ref{H1}) on both sides of the equality
sign and comparing coefficients the assertion follows.
\end{proof}

\begin{example}
\label{ex:4}Consider once more the simplest non-trivial case $\operatorname{Gr}%
_{1,3}=\mathbb{P}^{2}$. Let $\lambda =\mu =(2)$ and $\nu =\emptyset $. Then $%
\Pi (t_{\nu })=1+\beta t_{1}$, $\Pi (t_{\lambda })=1+\beta t_{3}$ and $%
\tilde{\lambda}=\emptyset ,$ $\tilde{\nu}=(2)$ with $d^{\prime }=d^{\prime
\prime }=1$ are the only boxed partitions which give rise to allowed
cylindric skew diagrams. Therefore, we arrive at the relation%
\begin{equation*}
\beta (t_{1}-t_{3})C_{22}^{\emptyset }=q\beta C_{\emptyset 2}^{\emptyset
}-q\beta (1+\beta t_{3})C_{22}^{2}=-q\beta (1+\beta t_{3})C_{22}^{2},
\end{equation*}%
where we have used that $v_{\emptyset }$ is the unit and we therefore must
have $C_{\emptyset 2}^{\emptyset }=0$. Similarly, setting $\nu =1$ we obtain%
\begin{equation*}
\beta (t_{2}-t_{3})C_{22}^{1}=q\beta C_{\emptyset 2}^{1}-\beta (1+\beta
t_{1})C_{22}^{\emptyset }\;.
\end{equation*}%
Thus, we end up with the recursion%
\begin{equation*}
C_{22}^{\emptyset }=q\frac{1+\beta t_{3}}{t_{3}-t_{1}}~C_{22}^{2},\qquad
C_{22}^{1}=\frac{1+\beta t_{1}}{t_{3}-t_{2}}~C_{22}^{\emptyset }
\end{equation*}%
with $C_{22}^{2}=(t_{3}\ominus t_{2})(t_{3}\ominus t_{1})$. Thus,%
\begin{equation*}
C_{22}^{\emptyset }=q(t_{3}\ominus t_{2})\frac{1+\beta t_{3}}{1+\beta t_{1}}%
,\qquad C_{22}^{1}=q\frac{1+\beta t_{3}}{1+\beta t_{2}}
\end{equation*}%
which is in agreement with our earlier computation and the product expansion
in \cite[Sec 5.5]{BM} upon setting $t_{i}=1-e^{\varepsilon _{4-i}}$ and $%
\beta =-1$.
\end{example}

\section{Localised Schubert classes and GKM theory}
An important result in (ordinary) equivariant quantum cohomology and
equivariant K-theory is that the respective rings have a purely algebraic
realisation by restricting Schubert classes to the fixed points under the
torus action. This monomorphism becomes a ring isomorphism with respect to
pointwise multiplication if one imposes the Goresky-Kottwitz-MacPherson
(GKM) conditions \cite[Thm 1.2.2]{GKM}; see \cite[Thm 3.13]{KostantKumar2}
for the analogous statement in K-theory. We now show that this algebraic
realisation naturally emerges from our lattice model approach for our
generalised cohomology theories $qh_{n}^{\ast }$.

\subsection{Generalised difference operators and Iwahori-Hecke algebras}

We recall that the ring $\mathcal{R}(\mathbb{T})=\mathcal{R}(t_{1},\ldots
,t_{N})$ is naturally endowed with an $\mathbb{S}_{N}$-action by permuting
the equivariant parameters. By abuse of notation we will identify
permutations $w\in \mathbb{S}_{N}$ with their operators acting on $\mathcal{R%
}(\mathbb{T})$. This $\mathbb{S}_{N}$-action can be used to define a
representation of a generalised (affine) Hecke or Iwahori algebra $\mathbb{H}%
_{N}(\beta )$; compare with \cite[Def. 2.2]{FK} and references therein.

\begin{definition}
Denote by $\mathbb{H}_{N}(\beta )$ the associative unital algebra with the
following generators and relations%
\begin{equation}
\pi _{i}^{2}=\beta \pi _{i}\qquad \text{and}\qquad \left\{ 
\begin{array}{cc}
\pi _{i}\pi _{j}=\pi _{j}\pi _{i}, & (i-j)\operatorname{mod}N\neq \pm 1 \\ 
\pi _{i}\pi _{i+1}\pi _{i}=\pi _{i+1}\pi _{i}\pi _{i+1}, & \text{else}%
\end{array}%
\right.  \label{Iwahori_def}
\end{equation}%
where all indices are understood modulo $N$. Denote by $\mathbb{H}_{N}^{%
\operatorname{fin}}(\beta )$ the subalgebra generated by $\{\pi _{1},\ldots ,\pi
_{N-1}\}$.
\end{definition}

The subring $\mathcal{R}[t_{1},\ldots ,t_{N}]\subset \mathcal{R}(\mathbb{T})$
and $\mathcal{R}(\mathbb{T})$ itself are both $\mathbb{H}_{N}^{\operatorname{fin}%
}(\beta )$-modules with respect to the following action in terms of isobaric
divided difference operators%
\begin{equation}
\partial _{j}=(1+\beta t_{j})~\frac{1-s_{j}}{t_{j}-t_{j+1}}\;,
\label{divided_difference}
\end{equation}%
where $s_{j}$ is the simple transposition interchanging $t_{j}$ and $t_{j+1}$%
. Note that setting $\beta =0$ we obtain a representation of the nil-Coxeter
algebra $\mathbb{A}_{N}=\mathbb{H}_{N}(0)$ and when setting $\beta =-1$ a
representation of the nil-Hecke algebra $\mathbb{H}_{N}=\mathbb{H}_{N}(-1)$.

\begin{proposition}[braid matrices]
Let $p_{j}:V_{n}\rightarrow V_{n}$ be the operator which permutes vectors in
the $j$th and $(j+1)$th factor and acts everywhere else trivially, i.e. $%
p_{j}v_{b}=v_{s_{j}b}$. Then the matrices $\{\hat{r}%
_{j}(t_{j},t_{j+1})=p_{j}r_{j+1,j}(t_{j+1}\ominus t_{j})\}_{j=1}^{N}$\ act
on the spin basis $\{v_{b}\}_{|b|=n}$ in $\mathcal{V}_n^q$ via 
\begin{equation}
\hat{r}_{j}(t_{j},t_{j+1})v_{b}=\left\{ 
\begin{array}{cc}
(1+\beta t_{j+1}\ominus t_{j})v_{b}+q^{-\delta _{j,N}}t_{j+1}\ominus
t_{j}~v_{s_{j}b}, & b_{j}<b_{j+1} \\ 
v_{b}, & \text{else}%
\end{array}%
\right.  \label{braid_action}
\end{equation}%
Moreover, the $\hat{r}_{j}$'s obey the relations%
\begin{multline}
\hat{r}_{j}(t_{j+1},t_{j+2})\hat{r}_{j+1}(t_{j},t_{j+2})\hat{r}%
_{j}(t_{j},t_{j+1})=  \label{braid_relation} \\
\hat{r}_{j+1}(t_{j},t_{j+1})\hat{r}_{j}(t_{j},t_{j+2})\hat{r}%
_{j+1}(t_{j+1},t_{j+2})
\end{multline}%
and%
\begin{gather}
\hat{r}_{j}^{2}-(2+\beta t_{j+1}\ominus t_{j})~\hat{r}_{j}+(1+\beta
t_{j+1}\ominus t_{j})~1=0  \label{quad_relation} \\
(s_{j}\otimes 1)\hat{r}_{j}=\hat{r}_{j}^{-1}(s_{j}\otimes 1)\;.
\label{inv_relation}
\end{gather}%
Here all indices are understood modulo $N$.
\end{proposition}

\begin{proof}
If we fix the basis $\{v_{0}\otimes v_{0},v_{0}\otimes v_{1},v_{1}\otimes
v_{0},v_{1}\otimes v_{1}\}$ in $V_{j}\otimes V_{j+1}$ then $\hat{r}_{j}$
reads as a matrix,%
\begin{equation}
\hat{r}_{j}(t_{j},t_{j+1})=\left( 
\begin{array}{cccc}
1 & 0 & 0 & 0 \\ 
0 & 1+\beta t_{j+1}\ominus t_{j} & 0 & 0 \\ 
0 & q^{-\delta _{j,N}}t_{j+1}\ominus t_{j} & 1 & 0 \\ 
0 & 0 & 0 & 1%
\end{array}%
\right) _{j,j+1}
\end{equation}%
Using this matrix form one now verifies easily the various assertions.
\end{proof}

\begin{corollary}[symmetric group action]
The operators $\boldsymbol{s}_{j}=(s_{j}\otimes 1)\hat{r}_{j}$ for $%
j=1,\ldots ,N-1$ define an action of the symmetric group $\mathbb{S}_{N}$ on
the space $\mathcal{V}_{n}$. For $q$ invertible, we have an action of the
affine symmetric group with $\boldsymbol{s}_{N}=(s_{N}\otimes 1)\hat{r}%
_{N}(q^{-1})$ on $\mathcal{V}_n^q=\mathbb{Z}[q^{\pm 1}]\otimes \mathcal{V}_{n}$, where $%
s_{N} $ is the affine reflection in the level-zero representation on $%
\mathcal{R}(\mathbb{T})$. Explicitly, one has in the spin-basis%
\begin{equation}
\boldsymbol{s}_{j}v_{b}=\left\{ 
\begin{array}{cc}
(1+\beta t_{j}\ominus t_{j+1})v_{b}+q^{-\delta _{j,N}}t_{j}\ominus
t_{j+1}~v_{s_{j}b}, & b_{j}<b_{j+1} \\ 
v_{b}, & \text{else}%
\end{array}%
\right. \;.  \label{SNaction}
\end{equation}%
Note that the $\mathbb{S}_{N}$-action does not commute with the
multiplicative action of $\mathcal{R}(\mathbb{T})$ on $\mathcal{V}_{n}$.
\end{corollary}

\begin{proof}
That the $\boldsymbol{s}_{j}$ yield a representation of $\mathbb{S}_{N}$
follows easily from our previous findings (\ref{braid_relation}), (\ref%
{quad_relation}), and (\ref{inv_relation}).
\end{proof}

The next result shows that the Yang-Baxter algebra (\ref{row_yba}) commutes
with the action \eqref{SNaction} of the symmetric group. For the transfer matrices this
extends to the action including the affine reflection depending on the
deformation parameter $q$; compare with \ref{qybe}.

\begin{corollary}
The action \eqref{SNaction} of $\mathbb{S}_N$ on $\mathbb{Z}[x]\otimes \mathcal{V}$ commutes with the action of
the row Yang-Baxter algebras, i.e.%
\begin{eqnarray}
(1\otimes \boldsymbol{s}_{j})M(x_i|t) &=&M(x_i|t)(1\otimes \boldsymbol{s}_{j})
\label{SWD} \\
(1\otimes \boldsymbol{s}_{j})M^{\prime }(x_i|t) &=&M^{\prime }(x_i|t)(1\otimes 
\boldsymbol{s}_{j}),\qquad j=1,2,\ldots ,N-1,  \label{SWD'}
\end{eqnarray}%
where $M,M^{\prime }$ are the monodromy matrices in (\ref{rowmom}) and (\ref%
{row_yba}) for $L$ and $L^{\prime }$, respectively. In case of the transfer
matrices we have the additional relations%
\begin{equation}
\boldsymbol{s}_{N}H(x_i|t)=H(x_i|t)\boldsymbol{s}_{N}\quad \;\text{and\quad\ }%
\boldsymbol{s}_{N}E(x_i|t)=E(x_|t)\boldsymbol{s}_{N}
\end{equation}%
with $\boldsymbol{s}_{N}=(s_{N}\otimes 1)\hat{r}_{N}$ and $%
s_{N}=s_{1}s_{2}\cdots s_{N-2}s_{N-1}s_{N-2}\cdots s_{2}s_{1}$.
\end{corollary}

\begin{remark}
The commutation of the symmetric group action (\ref{SNaction}) with the
action of the Yang-Baxter algebra (\ref{row_yba}) is reminiscent of
Schur-Weyl duality and we will explore this connection in a forthcoming
publication.
\end{remark}

\begin{proof}
The commutation relations with the monodromy and transfer matrices follow
from (\ref{ybe}) and (\ref{qybe}).
\end{proof}

\begin{proposition}[generalised divided difference operators]
The matrices%
\begin{equation}
\delta _{j}=\frac{1-\hat{r}_{j}}{t_{j}-t_{j+1}}~(1+\beta t_{j})=\left( 
\begin{array}{cccc}
0 & 0 & 0 & 0 \\ 
0 & \beta & 0 & 0 \\ 
0 & q^{-\delta _{jN}} & 0 & 0 \\ 
0 & 0 & 0 & 0%
\end{array}%
\right) _{j,j+1}\;,  \label{Iwahori_action}
\end{equation}%
define an action of $\mathbb{H}_{N}(\beta )$ on the space $\mathbb{Z}[q^{\pm
1}]\otimes V^{\otimes N}$.
\end{proposition}

\begin{proof}
A straightforward computation using the explicit matrix representation given
which follows from (\ref{braid_action}).
\end{proof}

Note that the action (\ref{Iwahori_action}) commutes with the multiplicative
action of $\mathcal{R}(\mathbb{T})$ on $\mathcal{V}$.

\subsection{Localised Schubert classes}
Recall that each boxed partition $\mu \subset (k^{n})$ can be identified
with the 01-word of length $N$ which has one-letters at positions $I_{\mu }$%
. Recall the natural $\mathbb{S}_{N}$-action on 01-words, i.e. write $%
s_{j}\mu $ for the partition obtained by exchanging the $j$th and $(j+1)$th
letter in the corresponding 01-word for $\mu $.

Define a sequence $\mathcal{G}_{\lambda }=((\mathcal{G}_{\lambda})_{\mu })_{\mu \subset (k^{n})}$ 
in $\mathcal{R}(\mathbb{T},q)^{\binom{N}{n}%
}$ with $(\mathcal{G}_{\lambda })_{\mu }:=G_{\lambda }(y_{\mu }|\ominus t)$. Note that the elements in 
this sequence describe the basis change from the Bethe vectors (idempotents) to the spin basis 
(Schubert classes) and depend on the solutions to the Bethe ansatz equations (\ref{BAE}); see (\ref{Bethe2spin}). 
The following theorem shows that these sequences obey 
generalised Goresky-Kottwitz-MacPherson conditions which specialise to the known conditions for the 
special cases of equivariant cohomology \cite[Thm 1.2.2]{GKM} if $q=0$ and $\beta=0$, and K-theory 
\cite[Thm 3.13]{KostantKumar2}  (see \cite[Appendix A]{Rosu} for an explicit formula) if $q=0$ and $\beta=-1$. 

\begin{theorem}[localised Schubert classes]
\label{thm:GKM}The sequences $\{\mathcal{G}_\lambda\}_{\lambda\subset(k^n)}$ 
obey the following generalised Goresky-Kottwitz-MacPherson condition%
\begin{equation}
\boldsymbol{s}_{j}\mathcal{G}_{\lambda }-\mathcal{G}_{\lambda
}=(t_{j}\ominus t_{j+1})\delta _{j}^{\ast }\mathcal{G}_{\lambda },\;
\label{GKM}
\end{equation}%
where $\boldsymbol{s}_{j}$ denotes the $\mathbb{S}_{N}$-action given by $(%
\boldsymbol{s}_{j}\mathcal{G}_{\lambda })_{\mu }=s_{j}G_{\lambda
}(y_{s_{j}\mu }|\ominus t)$ and 
\begin{equation}
\delta _{j}^{\ast }\mathcal{G}_{\lambda }=\left\{ 
\begin{array}{cc}
\beta \mathcal{G}_{\lambda }+\mathcal{G}_{s_{j}\lambda }, & \text{%
if }j\notin I_{\lambda }\text{ and }(j+1)\in I_{\lambda } \\ 
0, & \text{else}%
\end{array}%
\right. \;.  \label{gendiffop}
\end{equation}
In particular, for $q=0$ and $\beta=0,-1$ we have that $(\mathcal{G}_\lambda)_\mu=[\mathcal{O}_\lambda]_\mu$, where 
$[\mathcal{O}_\lambda]_\mu$ denotes a localised Schubert class in $H_{\mathbb{T}}^\ast(\operatorname{Gr}_{n,N})$ and $K_{\mathbb{T}}(\operatorname{Gr}_{n,N})$, respectively.
\end{theorem}

To prove the theorem we require the following result first.

\begin{lemma}
Let $\boldsymbol{s}_{j}=(s_{j}\otimes 1)\hat{r}_{j}$ be the $\mathbb{S}_{N}$%
-action (\ref{SNaction}). Then 
\begin{equation}
\boldsymbol{s}_{j}Y_{b}=Y_{s_{j}b},\qquad j=1,2,\ldots ,N-1\;.
\label{diagSNaction}
\end{equation}%
In other words, in the basis of Bethe vectors the action (\ref{SNaction}) is the natural diagonal $\mathbb{%
S}_{N}$-action on $\mathcal{V}_{n}$ in the basis of Bethe vectors.
\end{lemma}

\begin{proof}
Consider the action of $\hat{r}_{j}$ on an off-shell Bethe vector. According
to (\ref{SWD}) we have%
\begin{equation*}
\hat{r}_{j}B(x_{1}|t)\cdots B(x_{n}|t)v_{0\cdots 0} =B(x_{1}|s_{j}t)\cdots
B(x_{n}|s_{j}t)v_{0\cdots 0} \;.
\end{equation*}%
According to Lemma \ref{lem:BAE} the Bethe roots are uniquely determined by
the constant term, $y_{\lambda }=t_{\lambda }+O(q)$, thus, we have%
\begin{equation*}
\boldsymbol{s}_{j}|y_{\mu }\rangle =s_{j}(B(y_{\mu _{n}+1}|s_{j}t)\cdots
B(y_{\mu _{1}+n}|s_{j}t))v_{0\cdots 0} =|s_{j}y_{\mu }\rangle =|y_{s_{j}\mu
}\rangle \;.
\end{equation*}%
An analogous argument shows that 
\begin{equation*}
s_{j}e(y_{\mu },y_{\mu })=\langle
\tilde v_{0\cdots 0}|\tprod_{i=1}^{n}C(y_{i}|t)\tprod_{i=1}^{n}B(y_{i}|t)v_{0\cdots 0}\rangle =e(y_{s_{j}\mu
},y_{s_{j}\mu })~.
\end{equation*}
\end{proof}

We now prove the generalised GKM conditions (\ref{GKM}).

\begin{proof}[Proof of Theorem \protect\ref{thm:GKM}]
Employ the expansion (\ref{Bethe2spin}) and apply $\boldsymbol{s}_{j}$ on
both sides of the equation. Then using (\ref{SNaction}) on the left hand
side and (\ref{diagSNaction}) on the right hand side of the equality, we
obtain%
\begin{equation*}
v_{\lambda }+(t_{j}\ominus t_{j+1})\delta _{j}v_{\lambda }=\sum_{\mu \subset
(k^{n})}(s_{j}G_{\lambda }(y_{s_{j}\mu }|\ominus t))Y_{\mu }\;.
\end{equation*}%
Comparing coefficients with respect to the basis of the Bethe vectors yields
(\ref{GKM}).
\end{proof}

The next result states a generating formula for localised Schubert classes
using the representation (\ref{gendiffop}) of the Iwahori-Hecke algebra. For 
$q=0$ and $\beta =-1$ this statement is originally due to Kostant and Kumar 
\cite{KostantKumar2}; see also \cite[Appendix A]{Rosu} for an explicit formula. 

Employing McNamara's Vanishing Theorem we easily find for $q=0$ that%
\begin{equation}
\lbrack \mathcal{O}_{(k^{n})}]_{\lambda }=G_{(k^{n})}(t_{\lambda }|\ominus
t)=\left\{ 
\begin{array}{cc}
\prod_{i=1}^{k}\prod_{j=k+1}^{N}t_{j}\ominus t_{i}, & \lambda =(k^{n}) \\ 
0, & \text{else}%
\end{array}%
\right.  \label{topclass}
\end{equation}%
which gives us an explicit description for the top (localised) Schubert
class. For the quantum case with $q\neq 0$ we have instead 
\begin{equation}
(\mathcal{G}_{(k^{n})})_{\lambda }
=G_{(k^{n})}(y_{\lambda }|\ominus
t)=\prod_{j=1}^{k}\prod\limits_{i\in I_{\lambda }}y_{i}\ominus t_{j}
\label{qtopclass}
\end{equation}%
where $y_{\lambda }$ is the solution (\ref{ylambda}) of (\ref{BAE}) and the
values 
$(\mathcal{G}_{(k^{n})})_{\lambda }$ 
at fixed points $y_{\lambda }$ with $\lambda \neq (k^{n})$ are in general nonzero.

\begin{corollary}
Any generalised Schubert class 
$\mathcal{G}_{\lambda }$ can be obtained by successive
action of the generalised difference operators $\delta _{j_{1}}^{\ast
},\delta _{j_{2}}^{\ast },\ldots ,\delta _{j_{r}}^{\ast }$ on the top class $%
\mathcal{G}_{(k^{n})}$ for some $j_{1},\ldots ,j_{r}\in \lbrack N]$ such
that $w=s_{j_{1}}\cdots s_{j_{r}}$ is a reduced word with $w(k^{n})=\lambda $
in terms of the natural $\mathbb{S}_{N}$-action on 01-words.
\end{corollary}

\begin{proof}
A direct consequence of (\ref{gendiffop}) and the $\mathbb{S}_{N}$-action on
binary strings.
\end{proof}

\begin{corollary}
The ring $qh_{n}^{\ast }/\langle q,\beta +1,t_{j}-1+e^{\varepsilon
_{N+1-i}}\rangle $ is isomorphic to $K_{\mathbb{T}}(\operatorname{Gr}_{n,N})$,
while the ring $qh_{n}^{\ast }/\langle q,\beta \rangle $ is isomorphic to $%
H_{\mathbb{T}}^{\ast }(\operatorname{Gr}_{n,N})$. In both cases the isomorphism
is given by $v_{\lambda }\mapsto \lbrack \mathcal{O}_{\lambda }]$, that is
the spin basis (\ref{spin basis}) is mapped onto Schubert classes.
\end{corollary}

\begin{proof}
Working in the basis of Bethe vectors we employ once more (\ref{Bethe2spin})
for $q=0$ to find%
\begin{equation}
v_{\lambda }=\sum_{\mu }G_{\lambda }(t_{\mu }|\ominus t)Y_{\mu }\;.
\label{Bethe2spin0}
\end{equation}%
In other words each Schubert class $[\mathcal{O}_{\lambda }]$ is identified
with the (finite) sequence $\{G_{\lambda }(t_{\mu }|\ominus t)\}_{\mu
\subset (k^{n})}$ where each boxed partition $\mu $ labels a fixed point
under the torus action. The definition (\ref{idempotent_def}) of $%
\circledast $ corresponds to pointwise multiplication of these sequences
which satisfy the conditions (\ref{GKM}) and can be successively generated
from the top class (\ref{topclass}). The assertion then follows from \cite[%
Thm 1.2.2]{GKM} for $\beta =0$ and from \cite[Cor. A.5]{Rosu} for $%
\beta =-1$.
\end{proof}

\begin{corollary}
The ring $qh_{n}^{\ast }/\langle \beta \rangle $ is isomorphic to
equivariant quantum cohomology $QH_{\mathbb{T}}^{\ast }(\operatorname{Gr}_{n,N})$.
\end{corollary}

\begin{proof}
Consider the equivariant quantum Pieri-Chevalley rule (\ref{H1}). Rewriting
it as 
\begin{equation*}
H_{1}v_{\mu }=\beta ^{-1}\left( \frac{\Pi (t_{\mu })}{\Pi (t_{\emptyset })}%
-1\right) v_{\mu }+\frac{\Pi (t_{\mu })}{\Pi (t_{\emptyset })}\sum 
_{\substack{ \mu \rightrightarrows ^{\ast }\lambda \lbrack d]  \\ d=0,1}}%
q^{d}v_{\lambda }
\end{equation*}%
where the sum runs over all $\lambda \subset (k^{n})$ such that $\lambda
\neq \mu $ and either $\lambda /\mu $ or $\lambda /1/\mu $ is a skew diagram
which contains at most\emph{\ }one box in each column or row. Setting $\beta
=0$ this simplifies to%
\begin{equation*}
H_{1}v_{\mu }=\left( \tsum_{i\in I_{\mu }}t_{i}-\tsum_{i=1}^{n}t_{i}\right)
v_{\mu }+\sum_{\substack{ \lambda /d/\mu =(1)  \\ d=0,1}}q^{d}v_{\lambda
}=v_{1}\circledast _{\beta =0}v_{\mu }
\end{equation*}%
where the sum now runs over all $\lambda \subset (k^{n})$ such that $\lambda
\neq \mu $ and either $\lambda /\mu $ or $\lambda /1/\mu $ is a skew diagram
which contains \emph{exactly} one box. This is Mihalcea's equivariant
quantum Pieri-Chevalley rule for $QH_{\mathbb{T}}^{\ast }(\operatorname{Gr}%
_{n,N}) $ which together with the usual grading, $v_\lambda$ has degree $%
|\lambda|$ and $q$ has degree $N$, fixes the ring up to isomorphism; see 
\cite[Cor 7.1]{MihalceaAIM}. An alternative proof which exploits the
presentation of $QH_{\mathbb{T}}^{\ast }(\operatorname{Gr}_{n,N})$ as Jacobi
algebra can be found in \cite{GoKo}.
\end{proof}

\subsection{Equivariant quantum Pieri rules and Giambelli formula}

According to its definition (\ref{BigG}) the operator $\boldsymbol{G}%
_{\lambda }$ is the multiplication operator which multiplies with a
localised Schubert class. The following corollary states that for $\lambda $
being a single row or column this operator is given by the transfer matrices
(\ref{Hr}), (\ref{Er}) in the spin-basis.

\begin{corollary}
The operators $\{H_{r}\}_{r=1}^{k}$ and $\{E_{r}\}_{r=1}^{n}$ defined
respectively in (\ref{Hr}) and (\ref{Er}) act on the Bethe vectors $|y_{\mu
}\rangle $ by multiplication with $G_{r}(y_{\mu }|\ominus t)$ and $%
G_{1^{r}}(y_{\mu }|\ominus t)$, respectively. That is,%
\begin{equation}
\boldsymbol{G}_{r}=H_{r}\qquad \text{and\qquad }\boldsymbol{G}%
_{1^{r}}=E_{r}\;.  \label{Pieri0}
\end{equation}
\end{corollary}

Note in particular, that this implies for $q=0$ that the matrix elements $%
\langle v_\lambda, H_{r}v_\mu \rangle $, $\langle v_\lambda, E_{r} v_\mu \rangle $
in the basis $\{v_{\lambda }\}_{\lambda\subset (k^n)}$ give the coefficients in the equivariant
Pieri rules for $H_{\mathbb{T}}^{\ast }(\operatorname{Gr}_{n,n+k})$ if $\beta =0$
and for $K_{\mathbb{T}}(\operatorname{Gr}_{n,n+k})$ if $\beta =-1$.

\begin{proof}
Using (\ref{SpecE}) and the expansions (\ref{e}), (\ref{Er}) we deduce that%
\begin{equation*}
E_{r}|y_{\mu }\rangle =G_{1^{r}}(y|\ominus t)|y_{\mu }\rangle \;.
\end{equation*}%
But then (\ref{levelrankmom}) together with (\ref{Gduality}) gives%
\begin{eqnarray*}
H_{r}(t)|y_{\mu }\rangle &=&H_{r}(t)\Theta |y_{\mu ^{\ast }}\rangle =\Theta
E_{r}(\ominus t^{\prime })|y_{\mu ^{\ast }}\rangle \\
&=&G_{1^{r}}(\ominus y_{\mu ^{\ast }}|t^{\prime })|y_{\mu }\rangle
=G_{r}(y_{\mu }|\ominus t)|y_{\mu }\rangle \;.
\end{eqnarray*}
\end{proof}

In light of the expansion (\ref{G2F}) and (\ref{JTG}), the last result
allows us to express the operator (\ref{BigG}) which corresponds to
multiplication with a Schubert class, in terms of the transfer matrix
coefficients $H_{r}$ from (\ref{Hr}). The latter, as we have just seen,
correspond to multiplication with a Chern class. Such a formula expressing a
general Schubert class in terms of Chern classes, is often called \emph{%
Giambelli formula} in the literature on cohomology.

\begin{corollary}[equivariant quantum Giambelli formula]
For $\lambda \subset (k^{n})$ a boxed partition define the operators $%
\boldsymbol{F}_{\lambda }=\det (\tau ^{1-j}H_{\lambda _{i}-i+j})$ where $%
\tau $ is the shift operator (\ref{shifted power}) and $\tau ^{p}H_{r}$ is
the coefficient in (\ref{Hr}) with respect to the shifted factorial powers $%
(x|\tau ^{p}\,\ominus t^{\prime })^{r}$. Then%
\begin{equation}
\boldsymbol{G}_{\lambda }=\sum_{\alpha }\beta ^{|\alpha |}\phi _{\alpha
}(\lambda )\boldsymbol{F}_{\lambda +\alpha }  \label{BigG2F}
\end{equation}%
with the same conventions for $\alpha $ and $\phi _{\alpha }(\lambda )$ as
in Prop \ref{prop:Gexp} and $\boldsymbol{F}_{\lambda +\alpha }$ is defined
in terms of the straightening rule (\ref{Fstraight}).
\end{corollary}

\begin{example}
Recall the formula (\ref{G2Fex}) for $n=2$. Then%
\begin{equation*}
\boldsymbol{G}_{\lambda _{1},\lambda _{2}}=\frac{1+\beta t_{\lambda _{2}+1}}{%
1+\beta t_{1}}~\left( \boldsymbol{F}_{\lambda _{1},\lambda _{2}}+\boldsymbol{%
F}_{\lambda _{1},\lambda _{2}+1}\right)
\end{equation*}%
where%
\begin{equation*}
\boldsymbol{F}_{\lambda _{1},\lambda _{2}}=\left\vert 
\begin{array}{cc}
H_{\lambda _{1}} & \tau ^{-1}H_{\lambda _{1}-1} \\ 
H_{\lambda _{2}-1} & \tau ^{-1}H_{\lambda _{2}}%
\end{array}%
\right\vert
\end{equation*}%
and $H_{r}$ is given by (\ref{Hr2}) while the negative shifted factorial
power is defined as%
\begin{equation*}
\tau ^{-1}H_{k+1-i}=\sum_{j=1}^{i}\frac{H(t_{N-j}|t)}{\prod\nolimits_{1\leq
\ell \neq j\leq i}t_{N-j}\ominus t_{N-\ell }}\;.
\end{equation*}
\end{example}

For certain choices the formula (\ref{BigG2F}) considerably simplifies. We
already saw that for $\lambda $ a single row or column we obtain $%
\boldsymbol{G}_{r}=H_{r}$ and $\boldsymbol{G}_{1^{r}}=E_{r}$. Setting $%
\lambda =(k^{n})$ we find from (\ref{SpecE}) and (\ref{qtopclass}) that 
\begin{equation*}
\prod_{i=1}^{k}E(\ominus t_{i}|t)~Y_{\mu }=G_{(k^{n})}(y_{\mu }|\ominus
t)Y_{\mu }
\end{equation*}%
and, hence, that $\boldsymbol{G}_{(k^{n})}=\prod_{i=1}^{k}E(\ominus t_{i}|t)$.

\begin{corollary}[Fusion matrices]
\label{cor:fusion_matrices}The matrices $\{\boldsymbol{G}_{\lambda
}\}_{\lambda \subset (k^{n})}$ yield a faithful representation of $%
qh_{n}^{\ast }$, that is%
\begin{equation}
\boldsymbol{G}_{\lambda }\boldsymbol{G}_{\mu }=\sum_{\nu \subset
(k^{n})}C_{\lambda \mu }^{\nu }(t,q)\boldsymbol{G}_{\nu }\;.
\label{fusion_matrices}
\end{equation}
\end{corollary}

\begin{proof}
This is a direct consequence of $v_{\lambda }=v_{\lambda }\circledast
v_{\emptyset }=\boldsymbol{G}_{\lambda }v_{\emptyset }$ and the fact that
the $v_{\lambda }$'s are linearly independent. Namely, assume $%
0=\sum_{\lambda \subset (k^{n})}c_{\lambda }\boldsymbol{G}_{\lambda }$ for
some coefficients $c_{\lambda }$. Then $0=\sum_{\lambda \subset
(k^{n})}c_{\lambda }\boldsymbol{G}_{\lambda }v_{\emptyset }=\sum_{\lambda
\subset (k^{n})}c_{\lambda }v_{\lambda }$ and, thus, we must have $%
c_{\lambda }=0$ for all $\lambda \subset (k^{n})$. The product expansion
follows from (\ref{combproduct}), $v_{\lambda }\circledast v_{\mu }=%
\boldsymbol{G}_{\lambda }\boldsymbol{G}_{\mu }v_{\emptyset }$.
\end{proof}

\subsection{Coordinate ring presentation}\label{sec:coordring}

We now prove the presentation of $qh_{n}^{\ast }$ stated in the
introduction. Consider the polynomial algebra $\mathcal{A}_{n}$ generated by 
$\{e_{r}\}_{r=1}^{n}\cup \{h_{r}\}_{r=1}^{k}$ over $\mathcal{R}(\mathbb{T}%
,q) $ subject to the relations given by (\ref{ideal def}) with $e(x)$ and $%
h(x)$ as in (\ref{H def}) and (\ref{E def}). Define $\{g_{\lambda
}\}_{\lambda \subset (k^{n})}\subset \mathcal{A}_{n}$ as just explained for
the $\boldsymbol{G}_{\lambda }$'s: set $f_{\lambda }=\det (\tau
^{1-j}h_{\lambda _{i}-i+j})$, where the \textquotedblleft shifted
generators\textquotedblright\ $\tau ^{p}h_{r}$ are obtained by expanding $%
h(x)$ into shifted factorial powers $(x|\tau ^{p}t^{\ast })^{r}$, and then
introduce $g_{\lambda }$ through the analogous expansion as in Prop \ref%
{prop:Gexp} and (\ref{BigG2F}).

\begin{theorem}
\label{thm:coord_ring}The map $g_{\lambda }\mapsto v_{\lambda }$ constitutes
an algebra isomorphism $\mathcal{A}_{n}\cong qh_{n}^{\ast }$.
\end{theorem}

\begin{proof}
Introduce auxiliary variables $\xi =(\xi _{1},\ldots ,\xi _{n})$ by setting%
\begin{equation*}
e(x)=\prod_{i=1}^{n}x\oplus \xi _{i}\;.
\end{equation*}%
Dividing by $e(x)$ in (\ref{ideal def}) one obtains $h(x)$ as a rational
function in $x$, but as $h(x)$ -- by definition -- is polynomial in $x$ the
residues at the poles must vanish. This implies that the $\xi _{i}$'s obey
the Bethe ansatz equations (\ref{BAE}). Moreover, one deduces in a similar
manner as we did before that $e_{r}=G_{1^{r}}(\xi |\ominus t)$ and $%
h_{r}=G_{r}(\xi |\ominus t)$. Thus, $g_{\lambda }=G_{\lambda }(\xi |\ominus
t)$ according to Prop \ref{prop:Gexp}. This then implies that the map $%
g_{\lambda }\mapsto v_{\lambda }$ is an algebra homomorphism and it is also
surjective. It remains to show that the dimension of $\mathcal{A}_{n}$
equals the dimension of $qh_{n}^{\ast }$. Recall from Section \ref%
{sec:grothendieck} that each $G_{\lambda }$ can be expressed via (\ref{JTG}%
), (\ref{Fdet}) and (\ref{G2F}) in terms of $G_{r}$'s and that the factorial
Grothendieck polynomials $\{G_{\lambda }\}$ with $\lambda $ having at most $%
n $ parts form a basis of $\mathcal{R}(\mathbb{T},q)[\xi _{1},\ldots ,\xi
_{n}]^{\mathbb{S}_{n}}$, hence $\mathcal{R}(\mathbb{T},q)[\xi _{1},\ldots
,\xi _{n}]^{\mathbb{S}_{n}}\cong \mathcal{R}(\mathbb{T},q)[G_{1},G_{2},%
\ldots ]$. Therefore, we only have to show that each $G_{\lambda }(\xi
|\ominus t)$ with $\lambda \nsubseteq (k^{n})$ can be expressed as a linear
combination of the $\{g_{\mu }\}_{\mu \subset (k^{n})}$. But since the $\xi $%
's obey (\ref{BAE}), we can deduce that each $G_{\lambda }(\xi |\ominus t)$
with $\lambda \nsubseteq (k^{n})$ can be \textquotedblleft
reduced\textquotedblright\ using multiple times (\ref{Greduction}) until it
is indexed by a composition where no part is greater than $k$. Then one
applies repeatedly the straightening rule (\ref{Gstraight}) to rewrite the
result as a linear combination of the $g_{\mu }$'s with $\mu \subset (k^{n})$%
.
\end{proof}

\subsubsection{A generalised rim-hook algorithm\label{sec:rimhook}}

Our proof of the last theorem contains an algorithm for the successive
computation of the structure constants $C_{\lambda \mu }^{\nu }(t,q)$
without making use of the explicit solutions of the Bethe ansatz equations (%
\ref{BAE}) and the residue formula (\ref{Verlinde}). Namely, starting from
the Pieri rule (\ref{G1Pieri}) for $G_{1}$, one can use (\ref{Greduction})
and (\ref{Gstraight}) to define a generalised version of the rim-hook
algorithm at $\beta =0$ \cite{BCF}; see \cite{BBT} for a recent extension to
the equivariant case with $\beta =0$. We shall demonstrate this only on a
simple example.

\begin{example}
Set $G_{\lambda }=G_{\lambda }(\xi |\ominus t)$ and consider the following
product expansions which follow from (\ref{G1}) and (\ref{G1Pieri}),%
\begin{eqnarray*}
G_{1,0}\cdot G_{1,0} &=&t_{3}\ominus t_{2}~G_{1,0}+\frac{1+\beta t_{3}}{%
1+\beta t_{2}}\left( G_{2,0}+G_{1,1}+\beta G_{2,1}\right) \\
G_{1,0}\cdot G_{1,1} &=&t_{3}\ominus t_{1}~G_{1,1}+\frac{1+\beta t_{3}}{%
1+\beta t_{1}}~G_{2,1}\;.
\end{eqnarray*}%
For $N=3$ and $n=2$ employ (\ref{Greduction}) and (\ref{Gstraight}) to find%
\begin{equation*}
G_{2,0}=qG_{-1,0}=-q\beta G_{0,0}=-q\beta \qquad \text{and}\qquad
G_{2,1}=qG_{0,0}=q\;.
\end{equation*}%
This yields the following product expansion in $qh_{2}^{\ast }$,%
\begin{eqnarray*}
g_{1,0}\cdot g_{1,0} &=&t_{3}\ominus t_{2}~g_{1,0}+\left( 1+\beta
t_{3}\ominus t_{2}\right) g_{1,1} \\
g_{1,0}\cdot g_{1,1} &=&t_{3}\ominus t_{1}~g_{1,1}+\left( 1+\beta
t_{3}\ominus t_{1}\right) q
\end{eqnarray*}%
which because of $qh_{2}^{\ast }\cong qh_{1}^{\ast }$ -- see (\ref{levelrank}%
) -- are equivalent to the products $g_{1}\cdot g_{1}$ and $g_{1}\cdot g_{2}$
in $qh_{1}^{\ast }$ which we computed in Example \ref{ex:3}.
\end{example}

\subsection{Partition functions and Richardson varieties}

We provide another concrete example where a natural link between our lattice
model approach and geometry occurs. We show that the partition functions (\ref{partition_function}) 
represented in terms of matrix elements of the operators (\ref{Z2H}), (\ref{Z'2E}) for $q=0$ provide 
generating functions for the K-theoretic Littlewood-Richardson coefficients (\ref{Verlinde0}). 

First, we need to introduce another basis: define the \textquotedblleft opposite spin basis\textquotedblright\ $%
\{v^{\lambda }\}$ by setting%
\begin{equation}
v^{\lambda }=\sum_{\mu
\subset (k^{n})}G_{\lambda ^{\vee }}(y_{\mu }|\ominus t^{\prime })Y_{\mu }\;.
\label{op_spin_basis}
\end{equation}%
Comparison with \eqref{Bethe2spin} shows that the spin basis \eqref{spin basis} and opposite spin basis 
are related by $\lambda\rightarrow\lambda^\vee$ {\em and} replacing simultaneously the equivariant parameters $t=(t_1,\ldots,t_N)$ 
with $t'=(t_N,\ldots,t_2,t_1)$. In particular, if $q=0$ it follows from the trivial identity $t_\mu=t\rq{}_{\mu^\vee}$ that the opposite 
spin basis can be rewritten as
\begin{equation}
v^{\lambda }=\sum_{\mu
\subset (k^{n})}G_{\lambda ^{\vee }}(t\rq{}_{\mu^\vee }|\ominus t^{\prime })Y_{\mu }\;.\label{op_spin_basis0}
\end{equation}
The following proposition establishes the relationship between the opposite spin basis and the dual spin basis with respect to the inner product (\ref{bilinear_form}); compare with the geometric  definition of the dual Schubert basis \eqref{equiKdualbasis}.

\begin{proposition}\label{prop:dual_basis}
We have the relation%
\begin{equation}
(v_{\lambda },w^\mu)=\delta _{\lambda \mu },\qquad w^\mu=\frac{(1+\beta H_{1})v^{\mu }}{1+\beta G_1(t_{\mu }|\ominus t)}
\label{dual_spin_basis}
\end{equation}%
and the product expansion %
\begin{equation}
v_{\mu }\circledast v^{\lambda }=\sum_{\nu \subset (k^{n})}\frac{\Pi(t_\lambda)}{\Pi(t_\nu)}C_{\mu \nu
}^{\lambda }(t,q)v^{\nu }\;.  \label{dual_product}
\end{equation}
\end{proposition}

\begin{proof}
From the definition (\ref{op_spin_basis}), the identity (\ref{G1}),
\[
\frac{\Pi (t_{\emptyset })}{\Pi (t_{\lambda })}=\frac{1}{1+\beta G_1(t_{\lambda }|\ominus t)}\,,
\]
and (\ref{G1=H1}) it follows that%
\begin{equation*}
w^\lambda=\frac{(1+\beta H_{1})v^{\lambda }}{1+\beta G_1(t_{\lambda }|\ominus t)}=\sum_{\alpha \subset (k^{n})}\frac{\Pi (y_{\alpha })%
}{\Pi (t_{\lambda })}~G_{\lambda ^{\vee }}(y_{\alpha }|\ominus t^{\prime
})Y_{\alpha }\;.
\end{equation*}%
The assertion (\ref{dual_spin_basis}) then follows from the definition (\ref{bilinear_form}) and (\ref{res of 1}).

To find the product expansion we make once more use of (\ref{Bethe2spin})
and (\ref{op_spin_basis}) to find%
\begin{equation*}
v_{\mu }\circledast v^{\lambda }=\sum_{\alpha \subset (k^{n})}
G_{\mu }(y_{\alpha }|\ominus
t)G_{\lambda ^{\vee }}(y_{\alpha }|\ominus t^{\prime })Y_{\alpha }\;.
\end{equation*}%
Using (\ref{Cauchy2}) we compute the expansion%
\begin{equation*}
Y_{\alpha }=\sum_{\lambda
\subset (k^{n})}\frac{\Pi (y_{\alpha })}{\Pi (t_{\lambda })}~\frac{G_{\lambda }(y_{\alpha }|\ominus t)}{e(y_{\alpha
},y_{\alpha })}~v^{\lambda }\;.
\end{equation*}%
Inserting the latter into the previous equation we arrive at (\ref%
{dual_product}) by making use of (\ref{Verlinde}).
\end{proof}

Recall the definition of Richardson
varieties and the expansion (\ref{Richardson}) for $\beta =0$. The following 
result, which holds for generic $\beta$ and $q=0$, links in the special case of $\beta=0$ 
classes of Richardson varieties to the partition functions (\ref{partition_function}) 
of the lattice models on the finite strip.

\begin{corollary}\label{cor:Zexpansion}
The following matrix elements of the operators (\ref{Z2H}), (\ref{Z'2E}) for $q=0$ have the expansions%
\begin{eqnarray}
(v_\lambda, Z_{n}(x|t)v^\mu) &=&\sum_{\nu \subset
(k^{n})}c_{\lambda \nu }^{\mu }(t)\, \frac{\Pi (t_{\mu })}{\Pi (t_{\nu})}~G_{\nu ^{\vee }}(x|\ominus t^{\prime })\\
( v_\lambda, Z_{k}^{\prime }(x|t)v^\mu) &=&\sum_{\nu \subset
(k^{n})}c_{\lambda \nu }^{\mu }(t)\, \frac{\Pi (t_{\mu })}{\Pi (t_{\nu})}~G_{\nu ^{\ast }}(x|t)\,,
\end{eqnarray}%
where the coefficients $c^\mu_{\lambda\nu}(t)$ are the generalised K-theoretic Littlewood-Richardson coefficients 
(\ref{Verlinde0}).
\end{corollary}

\begin{proof}
Employing the result (\ref{SpecE}) from the Bethe ansatz we find for $q=0$ when acting on a Bethe vector,%
\[
Z^{\prime }(x_{1},\ldots ,x_{k}|t)Y_\alpha=
E(x_{1}|t)\cdots E(x_{k}|t)Y_\alpha =\prod_{i=1}^{k}\prod_{j\in I_{\alpha }}(x_{i}\oplus t_{j})\,Y_\alpha
\]
Making use of (\ref{Cauchy0}) with $\lambda\rightarrow\lambda^\ast,t\rightarrow \ominus t$ and (\ref{G1}) 
we have
\begin{eqnarray*}
\prod_{i=1}^{k}\prod_{j%
\in I_{\alpha }}(x_{i}\oplus t_{j})&=&\sum_\nu G_{\nu^\ast}(x|t)G_\nu(t_\alpha|\ominus t)
\frac{\Pi(t_\alpha)}{\Pi(t_\nu)}\\
&=&\sum_\nu G_{\nu^\ast}(x|t)G_\nu(t_\alpha|\ominus t)
\frac{\Pi(t_\emptyset)}{\Pi(t_\nu)} (1+\beta G_1(t_\alpha|\ominus t))\,.
\end{eqnarray*}
Since the Bethe vectors form a basis we thus have arrived for $q=0$ at the operator identity
\begin{equation}\label{NCCauchy}
Z^{\prime }(x_{1},\ldots ,x_{k}|t)=\sum_{\nu \subset(k^{n})}\frac{\Pi(t_\emptyset)}{\Pi(t_\nu)} G_{\nu ^{\ast }}(x|t)
(1+\beta H_1)\boldsymbol{G}_\nu
\end{equation}
which yields the asserted expression for $( v_\lambda, Z_{k}^{\prime }(x|t)v^\mu)$  via the identities 
(\ref{dual_spin_basis}), (\ref{dual_product}) from the last proposition. 

The identity for the vicious walker model now follows from (\ref{levelrankmom}) and level-rank
duality (\ref{levelrank}) for $q=0$.%
\end{proof}

\begin{remark}
We expect that an analogous expansion of the partition function holds also
for the quantum case with $q\neq 0$. However, we are currently lacking the
necessary quantum analogue of the identity (\ref{Cauchy0}).
\end{remark}

\subsection{The homogeneous limit $t_{j}=0$: quantum $K$-theory}

The inversion formulae (\ref{Hr2}), (\ref{Er2}) for the expansions (\ref{Hr}%
), (\ref{Er}) do not hold true in the homogeneous limit when $t_{j}=0$ for
all $j=1,\ldots ,N$. We therefore need to discuss this case separately. We
start with the Pieri formulae, i.e. the action of the transfer matrices in
the spin basis.

Given a toric horizontal (vertical) strip $\theta =\nu /d/\lambda $ denote
by $c(\theta )=|\mathcal{C}_{\theta }|$ the number of columns and by $%
r(\theta )=|\mathcal{R}_{\theta }|$ the number of rows which intersect the
strip.

\begin{corollary}[non-equivariant Pieri rules]
Set $t_{j}=0$ for all $j$. Then%
\begin{eqnarray}
H_{\ell }v_{\mu } &=&\sum_{\substack{ \theta =\lambda /d/\mu \text{ }  \\ 
\text{toric hor strip}}}q^{d}\beta ^{|\theta |-\ell }\binom{r(\theta )-1}{%
|\theta |-\ell }v_{\lambda }  \label{Hr0} \\
E_{\ell ^{\prime }}v_{\mu } &=&\sum_{\substack{ \theta =\lambda /d/\mu \text{
}  \\ \text{toric ver strip}}}q^{d}\beta ^{|\theta |-\ell ^{\prime }}\binom{%
c(\theta )-1}{|\theta |-\ell ^{\prime }}v_{\lambda }  \label{Er0}
\end{eqnarray}%
where $\ell =1,\ldots ,k$ and $\ell ^{\prime }=1,\ldots ,n$.
\end{corollary}

\begin{proof}
Setting $t_{j}=0$ for all $j$ the combinatorial action of the transfer
matrices on $\mathcal{V}_{n}$ simplifies to%
\begin{eqnarray}
H(x)v_{\mu } &=&\sum_{\substack{ \theta =\lambda /d/\mu \text{ } \\ \text{%
toric hor strip}}}q^{d}x^{k-|\theta |}(1+\beta x)^{r(\theta )}v_{\lambda }
\label{H0} \\
E(x)v_{\mu } &=&\sum_{\substack{ \theta =\lambda /d/\mu \text{ } \\ \text{%
toric ver strip}}}q^{d}x^{n-|\theta |}(1+\beta x)^{c(\theta )}v_{\lambda }
\label{E0}
\end{eqnarray}%
Employing in addition that the expansions (\ref{Hr}), (\ref{Er}) of the
transfer matrices on $\mathcal{V}_{n}$ in the variable $x$ now read%
\begin{eqnarray}
H(x) &=&x^{k}\cdot \mathbf{1}_{\mathcal{V}_{n}^{q}}+(1+\beta x)\sum_{\ell
=1}^{k}H_{\ell }x^{k-\ell }  \label{Hexp0} \\
E(x) &=&x^{n}\cdot \mathbf{1}_{\mathcal{V}_{n}^{q}}+(1+\beta x)\sum_{\ell
=1}^{n}E_{\ell }x^{n-\ell }  \label{Eexp0}
\end{eqnarray}%
and the asserted formulae are then easily deduced.
\end{proof}

We now turn to the Bethe ansatz computation. Since the matrix elements of
the $R$-matrix in (\ref{mom_ybe}) do not depend on the $t_{j}$'s the
commutation relations in the row Yang-Baxter algebra, and in particular the
relations in Lemma \ref{lem:yba1} and \ref{lem:yba2}, are unchanged for $%
t_{j}=0$. From this one deduces, along the same lines as before, that the
Bethe ansatz equations are obtained by formally setting $t_{j}=0$ in (\ref%
{BAE}),%
\begin{equation}
y_{i}^{N}\prod_{j\neq i}\frac{1+\beta y_{j}}{1+\beta y_{i}}%
=(-1)^{n-1}q,\qquad i=1,\ldots ,n\;.  \label{BAE0}
\end{equation}%
We have the following result which replaces Lemma \ref{lem:BAE} when $t_{j}=0
$. Suppose $q^{1/N}$ exists and set $\zeta =\exp (2\pi \imath /N)$ where $%
\imath $ is the imaginary unit.

\begin{lemma}
The set of equations (\ref{BAE0}) has $\binom{N}{n}$ \emph{pairwise distinct 
}solutions 
\begin{equation}
y_{\lambda }=(y_{\frac{n+1}{2}+\lambda _{n}^{\prime }-n},\ldots ,y_{\frac{n+1%
}{2}+\lambda _{1}^{\prime }-1})\in \mathbb{C}[\![\beta, q^{\frac{1}{N}}]\!]%
,  \label{Bethe_root_0}
\end{equation}%
where $\lambda \subset (k^{n})$ and up to first order in $\beta $ we have%
\begin{equation}
y_{j}=q^{\frac{1}{N}}\zeta ^{j}+\beta ~(-1)^{n-1}q^{\frac{2}{N}}\zeta^j\sum_{l\neq j}(\zeta^{l}-\zeta^j)+O(\beta ^{2})\;.  \label{Brootexpansion0}
\end{equation}
Moreover, the $r$th term in this expansion is proportional to $q^{r/N}$ and, thus, we can always force convergence for a given $\beta$ provided we specialise q to a sufficiently small number.
\end{lemma}

\begin{proof}
We now make the ansatz $y_{j}=\sum_{r\geq 0}y_{j}^{(r)}\beta^r $. Setting $%
\beta =0$ in (\ref{BAE0}) we obtain the Bethe ansatz equations at the free
fermion point which cease to be interdependent. Clearly, each of the $n$ equations has then $N$
solutions and using the conventions from \cite[Prop 10.4]{KorffStroppel} we
set $y_{j}^{(0)}=q^{1/N}\zeta ^{j}$ with $j\in \{\frac{n+1}{2}+\lambda
_{1}^{\prime }-1,\ldots ,\frac{n+1}{2}+\lambda _{n}^{\prime }-n\}$ for $%
\lambda \subset (k^{n})$. By the analogous arguments as in the previous case
when we expanded the Bethe roots with respect to $q$ we find by
differentiating with respect to $\beta $ and setting $\beta =0$ afterwards
the desired expansion. In particular, when taking the $r$th derivative with
respect to $\beta ,$ the coefficient $\left( \prod_{j\neq i}\frac{1+\beta
y_{j}}{1+\beta y_{i}}\right) _{\beta =0}=1$ in front of the term $\frac{d^{r}%
}{d\beta ^{r}}y_{j}^{N}|_{\beta =0}$ is always nonzero. One then proves by induction 
the stated dependence on $q^{1/N}$.
\end{proof}

This lemma can be used to establish the completeness of the Bethe ansatz when $t_j=0$ and
to derive the results analogous to (\ref{Greduction}), (\ref{specH}), (\ref{specE}), (\ref%
{SpecE}) and (\ref{Cauchy2}), (\ref{res of 1}) by simply setting formally $%
t_{j}=0$ in the respective formulae. Thus, extending the base field of our
quantum space, $\mathcal{V}^{q}=V^{\otimes N}\otimes\mathbb{C}[\![\beta, q^{\frac{1}{N}}]\!]$, we can introduce an algebra structure via (\ref{idempotent_def})
as before by making use of the Bethe vectors. Note, however, that this
extension to $\mathbb{C}[\![\beta, q^{\frac{1}{N}}]\!]$ is only necessary if we require
the existence of idempotents. Alternatively, we can introduce the product
structure via (\ref{combproduct}) by defining the analogue of the operator $%
\boldsymbol{G}_{\lambda }$ for $t_{j}=0$ as follows.

For each $n=0,1,\ldots ,N$ define operators $\boldsymbol{s}_{\lambda }\in 
\operatorname{End}(\mathbb{Z}[\![q]\!]\otimes V_{n})$ for $\lambda \subset
(k^{n}) $ by%
\begin{equation}
\boldsymbol{s}_{\lambda }=\det (\boldsymbol{e}_{\lambda _{i}^{\prime
}-i+j}),\qquad \boldsymbol{e}_{r}=\sum_{j=r}^{n}(-\beta )^{j-r}\binom{j-1}{%
j-r}E_{j}~,  \label{ncSchur}
\end{equation}%
where $r=1,2,\ldots ,n$; compare with (\ref{er2G1r}). We set $\boldsymbol{e}%
_{0}$ to be the identity operator. Note that since the $E_{j}$'s mutually
commute so do the $\boldsymbol{e}_{r}$'s, whence the determinant $%
\boldsymbol{s}_{\lambda }$ is well-defined.

Consider the commutative algebra $qh_{n}^{\ast }/\langle t_{1},\ldots
,t_{N}\rangle $ generated by $\{H_{r}\}_{r=1}^{k}\cup \{E_{r}\}_{r=1}^{n}$
with $t_{j}=0$. For each $\lambda \subset (k^{n})$ define in analogy with (%
\ref{G2Schur}) the operators 
\begin{equation}
\boldsymbol{G}_{\lambda }=\sum_{\alpha }\beta ^{|\alpha |}\prod_{i=1}^{n}%
\binom{i-1}{\alpha _{i}}\boldsymbol{s}_{\lambda +\alpha }~,
\label{Schubert0}
\end{equation}%
where $\boldsymbol{s}_{\lambda +\alpha }$ is defined in terms of the
straightening rules analogous to (\ref{Fstraight}).

\begin{corollary}
Consider $qh_{n}^{\ast }/\langle t_{1},\ldots ,t_{N}\rangle $. The map $%
v_{\lambda }\mapsto \lbrack \mathcal{O}_{\lambda }]$ defines for $\beta =0$
a ring isomorphism with $QH^{\ast }(\operatorname{Gr}_{n,N})$ and for $\beta =-1$
with $QK(\operatorname{Gr}_{n,N})$.
\end{corollary}

\begin{proof}
Recall that the rings $QH^{\ast }(\operatorname{Gr}_{n,N})$ and $QK(\operatorname{Gr}%
_{n,N})$ are multiplicatively generated from the Chern classes (see \cite%
{SiebertTian} for the case of quantum cohomology and \cite[Cor 5.7]{BM} for
quantum K-theory) which under the above maps are identified with the
coefficients $H_r$ and $E_r$ defined in (\ref{Hexp0}) and (\ref{Eexp0}),
respectively. Thus, it suffices to show that the respective rings feature
the same Pieri rule, i.e. that the respective expansions of the product of
such a Chern class with a general class coincide.

Setting $t_{j}=0$ in the functional relation (\ref{ideal def}), (\ref%
{func_eqn}) the resulting ring is well-defined and it follows from our
previous results (\ref{combproduct}), Cor \ref{cor:fusion_matrices} and Thm %
\ref{thm:coord_ring} that $qh_{n}^{\ast }/\langle t_{1},\ldots ,t_{N}\rangle 
$ is isomorphic to the ring with product $v_{\lambda }\circledast v_{\mu }=%
\boldsymbol{G}_{\lambda }v_{\mu }$ with $\boldsymbol{G}_{\lambda }$ given by
(\ref{Schubert0}). Here we implicitly used the fact that the transfer matrices $E$%
, $H$ stay well-defined when setting formally $t_{j}=0$, which in turn can be deduced from the explicit expressions for the $L$-operators (\ref{L}), (\ref{L'}). Furthermore, from the
definition (\ref{Schubert0}) it follows that $\boldsymbol{G}_{1^{r}}=E_{r}$
and, thus, the ring structure is fixed by the Pieri rule (\ref{Er0}) which
for $\beta =0$ coincides with the Pieri rule of $QH^{\ast }(\operatorname{Gr}%
_{n,N})$ \cite[p. 293]{Bertram} and for $\beta =-1$ with the Pieri rule of $%
QK(\operatorname{Gr}_{n,N})$ \cite[Thm 5.4]{BM}.
\end{proof}

The functional relation (\ref{func_eqn1}) when setting $t_{j}=0$ becomes,%
\begin{equation*}
(-1)^{n}(1+\beta x)^{n}H(x)E(\ominus x)=x^{N}(1+\beta
H_{1})+q(-1)^{n}(1+\beta x)^{n}\;.
\end{equation*}%
Using the expansions (\ref{Hexp0}), (\ref{Eexp0}) and comparing powers on
both sides of the functional relation one arrives at the following explicit
relations between the generators 
\begin{equation}
\sum_{a+b=N-r}(-1)^{a}\boldsymbol{e}_{a}(H_{b}+\beta H_{b+1})=\left\{ 
\begin{array}{cc}
0, & r=1,\ldots ,k-1 \\ 
q(-1)^{n}\binom{n}{N-r}\beta ^{N-r}, & r=k,\ldots ,N%
\end{array}%
\right. \;.  \label{ideal0}
\end{equation}%
The expression (\ref{bilinear_form_spin}) of the bilinear form and the
definition of the dual basis (\ref{dual_spin_basis}) simplify to%
\begin{equation}
(v_{\lambda },v_{\mu })=\sum_{\alpha \subset (k^{n})}\frac{G_{\lambda
}(y_{\alpha })G_{\mu }(y_{\alpha })}{e(y_{\alpha },y_{\alpha })}
\label{bilinear_form2}
\end{equation}%
and 
\begin{equation}
(v_{\lambda },(1+\beta H_{1})v_{\mu ^{\vee }})=\delta _{\lambda \mu },
\label{K_dual_basis_2}
\end{equation}%
because factorial Grothendieck polynomials are replaced with ordinary ones.
In particular, the opposite spin basis (\ref{op_spin_basis}) simply becomes $%
v^{\lambda }=v_{\lambda ^{\vee }}$ when $t_{j}=0$. Note that for $\beta =-1$
the definition (\ref{bilinear_form2}) and the relation (\ref{K_dual_basis_2}%
) are different from \cite[Thm 5.14]{BM}. This is not a contradiction, as
the invariance of the bilinear form only fixes it up to a multiplicative
factor, which with respect to the form defined in \emph{loc. cit.}, is $%
(1-q) $.

\begin{remark}
In the homogeneous limit the analogue of the Littlewood-Richardson rule for
stable Grothendieck polynomials is known \cite[Thm 5.4 and Cor 5.5]%
{BuchKtheory}. Therefore we can apply our generalised rim-hook algorithm
from Section \ref{sec:rimhook} also for the computation of the structure
constants of the quantum K-theory ring for Grassmannians.
\end{remark}

\end{document}